\documentclass[11pt]{amsart}

\usepackage{dsfont}

\usepackage{comment}
\usepackage{times}
\usepackage[dvipsnames]{xcolor}
\usepackage{color}		
\usepackage{graphicx}	

\usepackage{amsmath,amssymb,mathrsfs,bm, enumerate}
\usepackage{mathtools}

\usepackage{yhmath} 

\usepackage[colorlinks=true, linkcolor=black, citecolor=black, urlcolor=black]{hyperref}
\usepackage{tikz-cd} 
\usetikzlibrary{decorations.pathreplacing,calligraphy}
\usepackage[all,cmtip]{xy}
\xyoption{arrow}

\usepackage{stmaryrd}
\usepackage{tcolorbox}
\usepackage{graphicx}
\usepackage{amsmath,amsthm,amsfonts,amssymb}
\usepackage[mathscr]{euscript}

\usepackage{xargs}  

\usepackage[letterpaper, top=1.5cm, bottom=2.9cm, left=3cm, right=3cm, heightrounded,bindingoffset=0mm]{geometry}

\usepackage{enumitem}

\theoremstyle{plain}%
\newtheorem{maintheorem}{Theorem}
\newtheorem{main_proposition}[maintheorem]{Proposition}

\newtheorem{theorem}{Theorem}[section]
\newtheorem{remark}[theorem]{Remark}%

\newtheorem{proposition}[theorem]{Proposition}
\newtheorem{lemma}[theorem]{Lemma}
\newtheorem*{setup}{Setup}
\newtheorem{corollary}[theorem]{Corollary}
\newtheorem{example}[theorem]{Example}%

\newtheorem{question}[theorem]{Question}

\theoremstyle{definition}
\newtheorem{definition}[theorem]{Definition}
\newtheorem{main_definition}[maintheorem]{Definition}

\begin{document}

\title[A Matrix Rank Formula for vector bundles of VOA coinvariants and conformal blocks]{A Matrix Rank Formula for vector bundles of vertex operator\\ algebra coinvariants and conformal blocks}
{\tiny{\author[Xiangrui Luo]{Xiangrui Luo}
\address{Xiangrui Luo  \newline \indent   Department of Mathematics, University of Pennsylvania, Philadelphia, PA, 19104}
\email{xiangrui@sas.upenn.edu} }}

\vspace{-2.0 cm}

\begin{abstract} 
We introduce FA-matrices for computing ranks of vector bundles of coinvariants and conformal blocks associated with modules over vertex operator algebras on the moduli space of stable pointed curves, unifying the notions of fusion and averaging matrices and generalizing Ueno's work. To illustrate, we compute ranks of vector bundles determined by pointed VOAs and the tensor product of certain VOAs, as well as other examples. As an application, positivity properties of their first Chern classes are analyzed.
\end{abstract}

\vspace{-5.0 cm}

\maketitle

\section{Introduction}
Given a stable n-pointed coordinatized curve of genus $g$, together with n admissible modules over a vertex operator algebra (VOA) of CFT--type, one can construct a vector space of coinvariants and its dual space of conformal blocks \cite{TUY, NT, DGT22b}. These fit together to form a sheaf on the stack parameterizing families of such curves, called the \textit{sheaf of coinvariants} $\mathbb{V}_{g,n}(V,\{W^{\bullet}\})$. If the VOA satisfies certain mild assumptions, then the sheaf supports a projectively flat connection, with singularities on the boundary, and it will have finite-dimensional fibres. 
The connection allows one to conclude that the dimensions of the fibers are uniform at all smooth curves. To determine whether such a sheaf forms a vector bundle, one must check that the dimensions of the fibers do not increase at curves with singularities. In other words, the sheaf is locally free of finite rank. The rank of a vector bundle is important for a number of other reasons. For instance, it is an invariant, and is a crucial ingredient in computing Chern classes. In genus zero, rank encodes fusion rules. Moreover, there are many applications of knowing the rank for particular VOAs.

For affine VOAs at positive integer levels, sheaves of coinvariants and conformal blocks form vector bundles. More generally, if the sheaf is coherent and $V$ is rational, one can replace the vector spaces of coinvariants at nodal curves by sums of analogous spaces on the (partial) normalization of the curve \cite{TUY,NT,DGT22a}. This factorization theorem gives a recursive algorithm to reduce the rank computations to fusion rules.

One of our main results, Theorem \ref{main_cor_2}, gives a closed rank formula in all settings where factorization can be applied and fusion rules are understood. 

\begin{maintheorem}\label{main_cor_2}(Theorem \ref{cor_2}, matrix rank formula)
    Let $V$ be a strongly rational VOA (c.f., Definition \ref{def_strongly_rational}), $\{W_{1},\cdots,W_{l}\}$ be the collection of all irreducible admissible $V$-modules, up to isomorphism. For each irreducible module $W$ and $g\ge 0$, define an $l\times l$ matrix $R_{W,g}=(\mathrm{rank}\mathbb{V}_{g,3}(V,\{W,W_{i},W_{j}'\}))_{(i,j)}$. Let $S^{\bullet}=(S_{1},\cdots,S_{n})$ be an $n$-tuple of irreducible admissible modules. Then,
$$
    \mathrm{rank}\mathbb{V}_{g,n+2}\left(V,\left\{S^{\bullet}, W_{i}, W_{j}'\right\}\right)=\left(\left(\prod_{i=1}^{n}R_{S_{i},0}\right)R_{V,1}^{g}\right)_{(i,j)}.
$$
Moreover, if $g\ge 1$, one also has that
$$
\mathrm{rank}\mathbb{V}_{g,n}\left(V,\left\{S^{\bullet}\right\}\right)= \mathrm{Tr} \left( \left(\prod_{i=1}^{n}R_{S_{i},0}\right)R_{V,1}^{g-1}\right),
$$
where $\mathrm{Tr}$ denotes the trace of a matrix.
\end{maintheorem}

When $V$ is an affine VOA $L(\mathfrak{g})_\ell$, where $\mathfrak{g}$ is a simple Lie algebra and $\ell$ is a positive integer, the rank of a vector bundle of coinvariants, or equivalently, the rank of its dual bundle (i.e., the vector bundle of conformal blocks), is typically referred to as the \textit{Verlinde formula}. Since vector spaces of conformal blocks for such affine VOAs are canonically isomorphic to vector spaces of generalized theta functions\footnote{\cite[Remark 10.0.1]{DG} contains a more detailed historical overview of this topic.}, their dimensions are of wide interest. In fact, work on this subject has been carried out by a number of researchers, including \cite{TUY,BL, F3,KN,T,Bea96, DGT22a}. The ranks of vector bundles of coinvariants and conformal blocks for more general VOAs are sometimes referred to as the \textit{generalized Verlinde formula}, and Theorem \ref{main_cor_2} gives a new tool for computing them.

There are other approaches for obtaining closed rank formulae for these sheaves. Gao-Liu proved a similar version of Theorem \ref{main_cor_2} using a different method, by applying the factorization theorem to the degenerating process of the corresponding geometric picture underlying the formula \cite[Theorem 3.10]{GL}. Before that, Beauville's work for affine Lie algebras $\mathfrak{g}$ of type A,B,C,D or G gives rank formulas for sheaves of coinvariants defined by modules over $L(\mathfrak{g})_\ell$ \cite[Corollary 9.8]{Bea96}. Ueno computed the rank of sheaves defined by simple affine VOAs $L(\mathfrak{g})_\ell$ at positive integer levels $\ell$ via fusion matrix computations \cite[Theorem 5.21]{U08}. For affine VOAs $L(sl_{r+1})_{l}$ in type $A_{r}$, one can compute the ranks via cohomological computations using Witten's Dictionary \cite{BG}.  Damiolini--Gibney--Tarasca and Damiolini--Gibney computed the rank for a number of examples \cite{DGT3,DG}. .

Theorem \ref{main_cor_2} is proved here via a new tool, called the \textit{FA-matrix} (Definition \ref{main_def_FA_matrix}). Our main applications are to compute ranks of vector bundles of coinvariants for certain types of VOAs, and to analyze their positivity. For instance, the generating functions for the rank of the bundles determined by certain discrete series Virasoro VOAs are given by continued fractions. These continued fractions exhibit surprising symmetry (Theorem \ref{theorem_7}). The first Chern classes of vector bundles on $\overline{\mathcal{M}}_{0,n}$ defined by modules over $\text{Vir}_{p,q}$ with maximal conformal weight are positive, in the sense that they non-negatively intersect all effective curves. Similar results about other more general VOAs are also given.

We next describe the FA-matrices.

\begin{main_definition}\label{main_def_FA_matrix} (Definition \ref{def_FA_matrix})
    Let $V$ be a strongly rational VOA, $W_{1}$, $\ldots$, $W_{l}$ be the collection of all irreducible admissible $V$-modules, up to isomorphism. Let  $S^{\bullet}$ be an $n$-tuple of  irreducible admissible $V$-modules. Define the \textit{FA-matrix associated with $S^{\bullet}$ of genus $g$} as an $l\times l$-matrix, $R_{V, S^{\bullet},g}$, by
    $$
    \left(R_{V, S^{\bullet},g}\right)_{(i,j)}:=
    \mathrm{rank}\mathbb{V}_{g,n+2}\left(V,\left\{S^{\bullet}, W_{i}, W_{j}'\right\}\right), 
    $$
    where $i,j$ run through all $i=1,...,l, j=1,...,l$. If $V$ is clear from context, we write $R_{S^{\bullet},g}$ instead.
\end{main_definition}

The FA-matrix unites the notions of fusion and averaging matrices. The FA-matrix defined by a single irreducible module at genus zero is equal to the fusion matrix, since the space of coinvariants on $\overline{\mathcal{M}}_{0,3}$ is isomorphic to the space of intertwining operators. Ueno defined the \textit{averaging matrix} as $\sum_{W\in \mathcal{W}} \mathrm{Tr}(R_{W',0})R_{W,0}$  \footnote{Ueno did not give such matrix a name. Nowadays, it is sometimes referred to as the \textit{averaging matrix}.}, where $W$ runs through all irreducible $V$-modules, up to isomorphism \cite[Theorem 5.21]{U08}. We prove that the averaging matrix is equal to $R_{V,1}$ (Proposition \ref{proposition_1}).

FA-matrices behave well with respect to the usual matrix operations, and such operations often have geometric interpretations. For instance, the matrix multiplication of FA-matrices corresponds to applying the factorization theorem to coinvariants at points in the image of clutching maps $\overline{\mathcal{M}}_{g_{1},n_{1}+1}\times \overline{\mathcal{M}}_{g_{2},n_{2}+1} \rightarrow \overline{\mathcal{M}}_{g_{1}+g_{2},n_{1}+n_{2}}$ (Lemma \ref{cor_1}). Taking the trace of FA-matrices corresponds to computing ranks of coinvariants at curves of genus one higher (Theorem \ref{thm_5}). These two observations lead to Theorem \ref{main_cor_2}, which generalizes \cite[Theorem 5.21]{U08}, giving a new proof.

 \begin{maintheorem}\label{main_thm_5}(Lemma \ref{cor_1}, Theorem \ref{thm_5})
    Let $V$ be a strongly rational VOA, and $S^{\bullet}, T^{\bullet}$ be finite tuples of irreducible admissible $V$-modules, which need not have the same length. Then,
$$R_{S^{\bullet},g_{1}}R_{T^{\bullet},g_{2}}=R_{S^{\bullet}T^{\bullet},g_{1}+g_{2}},$$
    $$
     \mathrm{rank}\mathbb{V}_{g+1,n}\left(V,\{S^{\bullet}\}\right)= \mathrm{Tr}\left(R_{S^{\bullet},g}\right),
    $$

    for all $g_{1}, g_{2}, g\ge 0$, and $\mathrm{Tr}$ denotes the trace of a matrix.
\end{maintheorem}

We also show that taking the tensor product (or Kronecker product) of two FA-matrices corresponds to taking the tensor product of two VOAs (Lemma \ref{lem_1}). This result leads to Theorem \ref{main_thm_3}.

\begin{maintheorem}\label{main_thm_3} (Theorem \ref{thm_3}, \ref{thm_4}), 
    Let $V_{1}, V_{2}$ be two VOAs of CohFT-type. Let $S^{\bullet}=(S_{1},\cdots,S_{n})$ be an $n$-tuple of $V_{1}$-modules, $T^{\bullet}=(T_{1},\cdots,T_{n})$ be an $n$-tuple of $V_{2}$-modules. Then, for all $g,n\ge 0$,
    $$
\begin{aligned}\mathrm{rank}\mathbb{V}_{g,n}\left(V_{1}\otimes V_{2},\left\{S^{\bullet}\otimes T^{\bullet}\right\}\right)=&\mathrm{rank}\left(\mathbb{V}_{g,n}\left(V_{1},\left\{S^{\bullet}\right\}\right)\otimes \mathbb{V}_{g,n}\left(V_{2},\left\{T^{\bullet}\right\}\right)\right),\\
c_{1}\mathbb{V}_{g,n}\left(V_{1}\otimes V_{2},\left\{S^{\bullet}\otimes T^{\bullet}\right\}\right)=&c_{1}\left(\mathbb{V}_{g,n}\left(V_{1},\left\{S^{\bullet}\right\}\right)\otimes \mathbb{V}_{g,n}\left(V_{2},\left\{T^{\bullet}\right\}\right)\right).
    \end{aligned}
    $$

\end{maintheorem}

In this paper, we assume that vertex operator algebras are strongly rational, and thus the conformal weights of the admissible modules are rational numbers \cite{Miy}. Therefore, the vector bundles of coinvariants are defined on $\overline{\mathcal{M}}_{g,n}$ \cite[VB corollary]{DGT22a}. As stated, the rank of a vector bundle of coinvariants is a key ingredient in the formula for its Chern classes. If the vector bundle of coinvariants is globally generated, then its first Chern class will be base-point free and give rise to a morphism from $\overline{\mathcal{M}}_{g,n}$ to another projective scheme. Any potentially base-point divisor must be \textit{nef} (i.e., it must intersect all curves non-negatively). Among all the curves on $\overline{\mathcal{M}}_{g,n}$, there is a collection of curves corresponding to the numerical equivalence classes of irreducible components of one-dimensional strata of $\overline{\mathcal{M}}_{g,n}$, called F-curves, and one says that a divisor is \textit{F-nef} if it intersects all F-curves non-negatively. The F-conjecture predicts that a divisor is nef if and only if it is F-nef \cite{GKM}.

Damiolini--Gibney showed that if the VOA is strongly generated in degree 1, the vector bundle of coinvariants is globally generated on $\overline{\mathcal{M}}_{0,n}$ \cite{DG}, and hence the coinvariant divisor is base-point free, and is hence nef and F-nef. However, the F-nefness of coinvariant divisors is not understood in positive genus cases, or when the vertex operator algebra is not strongly generated in degree 1.

A \textit{pointed VOA} is a strongly rational VOA whose category of admissible modules is a pointed category \cite{GR}. There are many examples of pointed VOAs including lattice VOAs for positive definite even lattices, holomorphic VOAs, parafermion algebras for $\mathfrak{sl}_{2}$, as well as tensor products of pointed VOAs. The discrete series Virasoro algebras $V_{2,q}$ for $q\ne 3$ are not pointed. These examples are discussed in Section \ref{PointedExamples}.

 \begin{main_proposition}\label{main_prop_5}(Proposition \ref{prop_6})
    Let $V$ be a pointed vertex operator algebra, with fusion ring $\mathbb{Z}[G]$. Let $W^{\bullet}=(W_{x_{1}},\cdots,W_{x_{n}})$ be an $n$-tuple of irreducible $V$-modules, where $x_{i}\in G$.  Then,
    $$
    \mathrm{rank}\mathbb{V}_{g,n}\left(V,\left\{W^{\bullet}\right\}\right)=|G|^{g}\delta_{\prod_{i=1}^{n}x_{i},e},
    $$
    where the multiplication is in $G$, and $e\in G$ denotes the identity element.
\end{main_proposition}

We prove that for a pointed VOA, testing F-nefness on $\overline{\mathcal{M}}_{g,n}$, $g\ge 3$, can be reduced to testing F-nefness on $\overline{\mathcal{M}}_{3,n}$. Moreover, if the VOA is also unitary, it suffices to test F-nefness on $\overline{\mathcal{M}}_{2,n}$.
 
 \begin{maintheorem}\label{main_thm_6}(Theorem \ref{thm_6})
    Let $V$ be a pointed VOA, and $S^{\bullet}$ be an $n$-tuple of irreducible modules. Then, the following are equivalent:
    \begin{enumerate}
        \item $\forall \, g\ge 3, \mathbb{D}_{g,n}(V,\{S^{\bullet}\}) \text{ is F-nef}$,
        \item $\exists\,  g\ge 3, \mathbb{D}_{g,n}(V,\{S^{\bullet}\})\text{ is F-nef.}$
    \end{enumerate}

    F-nefness has been tested for bundles of coinvariants constructed by affine VOAs \cite{F2},  VOAs that are strongly generated in degree one \cite{DG}, Parafermions \cite{Cha}, and for bundles of conformal blocks constructed by discrete series Virasoro VOAs \cite{Choi}.   Theorem \ref{main_thm_6} gives extends our understanding of such positivity to a broader class. 
    
    Integral to this study of positivity is finding a closed formulae for the ranks of vector bundles of coinvariants or conformal blocks defined by modules over a pointed VOA. In particular, Proposition \ref{main_prop_5} extends the rank formula in \cite[Example 5.2.5]{DGT3} from a lattice VOA with cyclic discriminant group to a pointed VOA:

    If the averaging conformal weight of all irreducible $V$-modules, up to isomorphism, is non-negative (e.g., if $V$ is unitary), then the following are equivalent:
    \begin{enumerate}
        \item $\forall\, g\ge 2, \mathbb{D}_{g,n}(V,\{S^{\bullet}\}) \text{ is F-nef}$,
        \item $\exists\,  g\ge 2, \mathbb{D}_{g,n}(V,\{S^{\bullet}\})\text{ is F-nef.}$
    \end{enumerate}
\end{maintheorem}

Unlike \cite{DG} or Proposition \ref{main_thm_6}, we next take an alternative approach. Instead of imposing extra conditions on the VOA, we impose extra conditions on the modules. We say an irreducible module, $W$, is \textit{ of order two} if $W\boxtimes W = V$. We consider the positivity of the symmetric coinvariant divisor obtained by assigning an irreducible module of order two at each marked point of $\overline{\mathcal{M}}_{0,n}$. We prove that the vector bundle of coinvariants is either a line bundle or zero on $\overline{\mathcal{M}}_{0,n}$, and the (F-)nefness of the coinvariant divisor is uniquely determined by the conformal weight.

\begin{main_proposition}\label{main_prop_3}(Lemma \ref{lemma_1}, Proposition \ref{prop_3})
    Let $V$ be a strongly rational VOA, $S$ be an irreducible module of order two, with $a_{S}$ the conformal weight of $S$. Then,
    $$
\mathrm{rank}\mathbb{V}_{0,n}\left(V,\left\{S^{n}\right\}\right)=\delta_{n,\text{even}},
    $$
where $\delta_{n,\text{even}}$ equals $0$ if $n$ is odd, and equals $1$ when $n$ is even. Suppose that $V$ is of CohFT-type, then when $n$ is even, we also have
$$
a_{S} \ge 0 \iff \mathbb{D}_{0,n}\left(V,\left\{S^{n}\right\}\right) \text{ is nef.}
$$
In particular, when $a_{S}<0$ and $n$ is even, $\mathbb{D}_{0,n}\left(V,\left\{S^{n}\right\}\right)$ is not F-nef.
\end{main_proposition}
In particular, we identify several examples of such irreducible modules. For instance, let $V$ be a pointed VOA (e.g., lattice VOA), with fusion ring $\mathbb{Z}[G]$, and $S\in G$ be an element of order two. Moreover, we can take $V$ to be the affine VOA associated with $sl_{2}$, or a Virasoro VOA in the discrete series, and take $S$ as the irreducible module of the maximal conformal weight.

\section{Acknowledgement}
I would like to express my gratitude to my PhD advisor Prof. Gibney. The paper would not have been possible without her support and guidance. I also would like to extend my gratitude to Prof. Damiolini, Prof. Harbater, Prof. Krashen, and Prof. Liu (and in particular, his help with \cite{Bea96}). I am grateful to Brandon Rayhaun for his comments and suggestions to consider pointed VOAs. I would also like to express my thankfulness to my classmates Victor Alekseev, Avik Chakravarty, Daebeom Choi, Zixuan Qu, and Ruofan Zeng. I am grateful to my wife, Ziqing Lu. She helped me write code to compute the rank of the vector bundle of coinvariants of the Virasoro VOAs in the discrete series.

\section{Notation and Conventions}\label{sec:background}


In this paper, we always work with strongly rational, non-negatively graded vertex operator algebras. We follow \cite[Definition 1.1]{DGT22a} for the definition of \textit{non-negatively graded vertex operator algebra (VOA)}, and \cite[Definition 1.2]{DGT22a} for the definition of \textit{admissible modules} over a VOA.

\begin{definition}\cite[Definition 2.3]{Wang93}
    A
     vertex operator algebra is called \textit{rational} if it has only finitely many irreducible modules, and every finitely generated module is a direct sum of irreducibles.
\end{definition}

\begin{definition}\cite[Definition 2.2]{DG}
A vertex operator algebra, $V$, is called \textit{$C_{2}$-cofinite} if the subspace
$
C_{2}(V):=\mathrm{span}_{\mathbb{C}}(A_{(-2)B:A,b\in V})
$
has finite codimension in $V$.
\end{definition}

\begin{definition}\label{def_strongly_rational}\cite[Definition 2.1]{DGT3}
    A VOA is \textit{strongly rational} if it is rational and $C_{2}$-cofinite. If it is also self-contragredient and of CFT-type, then it is called of \textit{CohFT-type}.
\end{definition}

\begin{definition}
    A VOA is \textit{holomorphic} if it is strongly rational and $V$ is the only irreducible module\footnote{Here, we follow the definition in \cite[Section 2.2]{GR}. Some authors only require $V$ to be rational, instead of strongly rational, for example \cite[Section 9.3]{DGT22a}.}.
\end{definition}

    We follow \cite[Section 4]{Wang93} for the definition of \textit{Virasoro VOAs}, and the \textit{discrete series Virasoro VOAs}, \cite{FZ92} for \textit{affine VOAs}, \cite{Dong93} for \textit{lattice VOAs}, and \cite{GR} for \textit{pointed VOAs}.

\begin{definition}
    Let $V$ be a strongly rational vertex operator algebra, and $W^{\bullet}=(W_{1}, \cdots, W_{n})$ be an $n$-tuple of admissible $V$-modules. For any $g,n\ge 0$ such that $2g-2+n>0$, $(V,W^{\bullet})$ defines a vector bundle on $\overline{\mathcal{M}}_{g,n}$, called \textit{vector bundle of coinvariants} \footnote{Under weaker conditions, we may still obtain a sheaf, but it need not be a vector bundle (see \cite{DGT22b, DGT22a, DGK1}). We call it \textit{sheaf of coinvariants}. In this paper, we assume that $V$ is strongly rational, so it is indeed a vector bundle.}, denoted by $\mathbb{V}_{g,n}(V,\{W^{\bullet}\})$ (c.f., \cite[VB Corollary]{DGT22a}). When $V$ is of CohFT-type, we denote its first Chern class by $\mathbb{D}_{g,n}(V,\{W^{\bullet}\}):=c_{1}\mathbb{V}_{g,n}(V,\{W^{\bullet}\})$, and call it the \textit{coinvariant divisor} \cite{DGT3}.
\end{definition}
    We refer to \cite{GKM} for the definition of a divisor on $\overline{\mathcal{M}}_{g,n}$ to be \textit{nef} and \textit{F-nef}. We also refer to \cite{GKM} for the \textit{F-curves} of type 1-6.

We list the standing assumptions in this paper:
\begin{enumerate}
    \item A vertex operator algebra is assumed to be simple, strongly rational and of CFT-type. When we compute its first Chern class, we assume that it is of CohFT-type,
    \item $V$-modules are assumed to be admissible,
    \item Without loss of generality, we assume that the modules are irreducible,
    \item When we write $\overline{\mathcal{M}}_{g,n}$, we assume that $2g+n-2>0$,
    \item We assume $0\in \mathbb{N}$.
\end{enumerate}
By \cite{FZ92,Wang93,DLM, Ara12}, Virasoro VOA satisfies assumption (a) if and only if it lies in the discrete series. For lattice VOA, we always assume that the lattice is positive-definite and even.

When only the rank or the first Chern class of vector bundle of coinvariants is considered, we also use another notation. For a permutation $\sigma\in S_{n}$, $\mathbb{V}_{g,n}(V,\{S_{1},\cdots, S_{n}\})$ and $\mathbb{V}_{g,n}(V,\{S_{\sigma(1)},\cdots, S_{\sigma(n)}\})$ need not be isomorphic as vector bundles on $\overline{\mathcal{M}}_{g,n}$. However, they have the same rank and first Chern class. Thus, we may ignore the order of the modules in these cases. Since $V$ is strongly rational, it only has finitely many irreducible modules, up to isomorphism, denoted by $W_{1}, \cdots, W_{l}$. Thus, it suffices to count the number of times each $W_{i}$ appears in the tuple $S^{\bullet}$, and we denote it by $n_{i}$. We then use the notation $ \mathrm{rank}\mathbb{V}_{g,n}\left(V,\left\{W_{1}^{n_{1}}, \cdots, W_{l}^{n_{l}}\right\}\right)$ and $\mathbb{D}_{g,n}\left(V,\left\{W_{1}^{n_{1}}, \cdots, W_{l}^{n_{l}}\right\}\right)$, where $n=n_{1}+\cdots+n_{l}$.




\section{Formal Factorization Property and Formal FA-Matrix}\label{sec_Algebraic_and_Formal_Verlinde_Formula}

In this section, we employ the framework of formal factorization property, formal FA-property and formal matrix rank formula. We introduce two new tools, called the formal FA-matrix and formal matrix rank formula, which translates many problems of rank computation into the study of matrix multiplication, allowing us to employ tools in linear algebra, matrix theory and combinatorics.

    In this paper, we let $I=\{\lambda_{1},\cdots,\lambda_{l}\}$ be a finite set, with an involution $\lambda \mapsto \lambda'$, for all $\lambda\in I$. Let 
    $$
    \mathbb{N}^{(I)}=\left\{\sum_{i=1}^{l}n_{i}\lambda_{i} \mid n_{i}\in\mathbb{N} \right\}
    $$
    be the free monoid generated by $I$. We assume $0\in I$, but we do not assume $0'=0$.

\begin{definition}\label{def_FA_Matrix}
    Let $\{N_{g}: \mathbb{N}^{(I)}\rightarrow \mathbb{Z}\}_{g\in \mathbb{N}}$ be a sequence of set-theoretic maps.
For $\beta\in \mathbb{N}^{(I)}$, define an $l\times l$-matrix, $\mathcal{N}_{\beta,g}\in \mathrm{M}(l,\mathbb{Z})$, by
    $$
\left(\mathcal{N}_{\beta,g}\right)_{(i,j)}:=N_{g}\left(\beta+\lambda_{i}+\lambda_{j}'\right).
$$
We call $\mathcal{N}_{\beta,g}$ the \textit{formal FA-matrix associated with $\beta$ of genus $g$}. When $\beta\in I$ and $g=0$, we call $\mathcal{N}_{\beta,0}$ the \textit{formal fusion matrix} of $\lambda$, and we call $\mathcal{N}_{0,1}$ the \textit{formal averaging matrix}.
\end{definition}

\begin{definition}\label{def_factorization_property}\cite{Bea96}
Let $\{N_{g}:\mathbb{N}^{(I)}\rightarrow \mathbb{Z}\}_{g\in \mathbb{N}}$ be a sequence of set-theoretic maps. We say that $\{N_{g}\}_{g\in \mathbb{N}}$ satisfies the \textit{formal factorization property} if the following holds:
\begin{enumerate}
        \item (Factorization Property I) For all $x,y\in \mathbb{N}^{(I)}$, for $g=g_{1}+g_{2}$,
        \begin{equation}\label{equ_6}
    N_{g}(x+y)=\sum_{\lambda\in I}N_{g_{1}}\left(x+\lambda\right)N_{g_{2}}\left(y+\lambda'\right).
\end{equation}
        \item (Factorization Property II) For all $x\in \mathbb{N}^{(I)}$, for all $g\ge 1$,
\begin{equation}\label{equ_7}
    N_{g}(x)=\sum_{\lambda\in I}N_{g-1}\left(x+\lambda+\lambda'\right).
\end{equation}    
    \end{enumerate}
\end{definition}

\begin{definition}\label{def_FA_Property}
    Let $\{N_{g}: \mathbb{N}^{(I)}\rightarrow \mathbb{Z}\}_{g\in \mathbb{N}}$ be a sequence of set-theoretic maps, and $M(l,\mathbb{Z})$ be the set of all $l\times l$ matrices over $\mathbb{Z}$. $\{N_{g}\}_{g\ge 0}$ is said to satisfy the \textit{formal FA-property} if the following hold:
    \begin{enumerate}
        \item (FA 1) The following map is a morphism of monoids:
    \begin{equation}\label{equ_3_1}
    \begin{aligned}
    \mathcal{N}:\mathbb{N}^{(I)}\oplus \mathbb{N}&\longrightarrow \mathrm{M}(l,\mathbb{Z})\\
    (\beta,g) &\longmapsto \mathcal{N}_{\beta,g}.
    \end{aligned}
    \end{equation}
    
    \item (FA 2) For all $\beta\in \mathbb{N}^{(I)}, g\ge 0$,
    \begin{equation}\label{equ_4_1}\mathrm{Tr}\left(\mathcal{N}_{\beta,g}\right)=N_{g+1}\left(\beta\right),\end{equation}
     where $\mathrm{Tr}$ stands for the trace of a matrix.
    \end{enumerate}
\end{definition}

\begin{definition}\label{def_Verlinde_Property}
        Let $\{N_{g}: \mathbb{N}^{(I)}\rightarrow \mathbb{Z}\}_{g\in \mathbb{N}}$ be a sequence of set-theoretic maps. For $\beta\in \mathbb{N}^{(I)}$, let $\mathcal{N}_{\beta,g}$ be the FA-matrix. The \textit{formal matrix rank formula holds} on $\{N_{g}\}_{g\ge 0}$ if:
        \begin{enumerate}
            \item (V1) for all $\beta=\sum_{\lambda\in I}c_{\lambda}\lambda\in\mathbb{N}^{(I)}$, for all $\lambda_{i},\lambda_{j}\in I$, and for all $g\ge 0$,    \begin{equation}
    N_{g}\left(\beta+\lambda_{i}+\lambda_{j}'\right)=\left(\left(\prod_{\lambda\in I} \mathcal{N}_{\lambda,0}^{c_{\lambda}}\right)\mathcal{N}_{0,1}^{g}\right)_{(i,j)},
\end{equation}
    \item (V2)  for all $\beta=\sum_{\lambda\in I}c_{\lambda}\lambda\in\mathbb{N}^{(I)}$, and for all $g\ge 1$, 
\begin{equation}
    N_{g}(\beta)=\mathrm{Tr}\left(\left(\prod_{\lambda\in I} \mathcal{N}_{\lambda,0}^{c_{\lambda}}\right)\mathcal{N}_{0,1}^{g-1}\right),
\end{equation}
    \item (V3)
\begin{equation}\label{equ_15}
    \mathcal{N}_{0,1}=\sum_{\lambda\in I}\mathrm{Tr}\left(\mathcal{N}_{\lambda',0}\right)\mathcal{N}_{\lambda,0}.
\end{equation}
        \end{enumerate}
\end{definition}

Now, we present a key result in this paper.

\begin{theorem}\label{thm_1} (matrix rank formula)
    Conditions in Definitions \ref{def_factorization_property}, \ref{def_FA_Property}, and \ref{def_Verlinde_Property} are equivalent.
\end{theorem}
\noindent\textit{Proof.} The proof is in Appendix \ref{app_thm_1}.

In practice, there are many existing results that enable us to check that $\{N_{g}\}_{g\ge 0}$  satisfies the factorization property (see Section \ref{sec_5}). However, the factorization property (i.e. Equation \eqref{equ_6} and Equation \eqref{equ_7}) is recursive and does not provide a closed formula for compute $N_{g}(\alpha)$. The matrix rank formula, on the other hand, is a closed and explicit formula, realizing it as a matrix multiplication.

\begin{definition}\label{def_indexing_function}
    Let $I=\{\lambda_{1},\cdots,\lambda_{l}\}$ be a finite set, with an involution. Let $\phi: \mathbb{N}^{(I)}\rightarrow \mathbb{Z}$ be a set-theoretic map. For $\alpha,\beta\in \mathbb{N}^{(I)}$, define the \textit{indexing function of $\alpha$, with deviation $\beta$} as
    $$
    f_{\phi, \beta,\alpha,(i,j)}(z):=\sum_{k=0}^{\infty}\phi\left(\beta+(n+3)\alpha+\lambda_{i}+\lambda_{j}'\right)z^{n}\in \mathbb{Z}[[n]].
    $$
    If $\phi$ is clear from context, then we write $f_{\beta,\alpha,(i,j)}(z)$. 
    Define
    $$
    f_{\alpha}(z):=\sum_{k=0}^{\infty}\phi((n+3)\alpha)z^{n}.
    $$
    We define $f_{\alpha}(z)$ by taking $\beta=\lambda_{i}=\lambda_{j}=0$. If we are giving a sequence of functions $\{f_{g}\}_{g\in \Lambda}$, we write $f_{\phi, \beta,\alpha,(i,j),g}$. We call it \textit{indexing function of $\alpha$, with deviation $\beta$, of genus $g$}.
\end{definition}
Intuitively, the function will enable us to see how the value changes if we add more and more copies of $\alpha$. The constant $n+3$ exists because for the moduli space of curves, the minimal case is $\overline{\mathcal{M}}_{0,3}$.

\begin{proposition}\label{prop_2}
    Let $I=\{\lambda_{1},...,\lambda_{l}\}$ be a finite set. Let $\beta,\alpha\in \mathbb{N}^{(I)}$ be arbitrary. Suppose that the equivalent conditions in Theorem \ref{thm_1} hold. Then, $f_{I,\beta,\alpha,(\lambda_{i},\lambda_{j}'),g}(z)$ is a rational function.
\end{proposition}
\noindent\textit{Proof.} $$
\begin{aligned}
\frac{\mathcal{N}_{\beta,0}\mathcal{N}_{0,1}^{g} \mathcal{N}_{\alpha,0}^{3}}{\mathrm{Id}_{l}-\mathcal{N}_{\alpha,0}z}
=&\left(\mathcal{N}_{\beta,0}\mathcal{N}_{0,1}^{g}\mathcal{N}_{\alpha,0}^{3}\right)\left(\sum_{n=0}^{\infty}\mathcal{N}_{\alpha,0}^{n}z^{n}\right)= \sum_{n=0}^{\infty}\mathcal{N}_{\beta,0}\mathcal{N}_{0,1}^{g}\mathcal{N}_{\alpha,0}^{3}\mathcal{N}_{\alpha,0}^{n}z^{n} \\
=&\sum_{n=0}^{\infty} \mathcal{N}_{\beta+(3+n) \alpha,g}z^{n} \ \ \ \text{(Theorem \ref{thm_1})} .
\end{aligned}
$$
Taking the $(i,j)$-entry on both size, we have

\begin{equation}\label{equ_3}
\begin{aligned}
\left(\frac{\mathcal{N}_{\beta,0}\mathcal{N}_{0,1}^{g} \mathcal{N}_{\alpha,0}^{3}}{\mathrm{Id}_{l}-\mathcal{N}_{\alpha,0}z}\right)_{(i,j)}=&\sum_{n=0}^{\infty} N_{g}(\beta+(3+n)\alpha+\lambda_{i}+\lambda_{j}')z^{n}= f_{I,\beta,\alpha,(\lambda_{i},\lambda_{j}'),g}(z).
\end{aligned}
\end{equation}
Thus, we conclude that $f_{I,\beta,\alpha,(\lambda_{i},\lambda_{j}'),g}(z)$ is a rational function.
\begin{flushright}
    Q.E.D.
\end{flushright}

\begin{remark}
    Notice that the proof of Proposition \ref{prop_2} is constructive. In practice, we can directly compute $f_{I,\beta,\alpha,(\lambda_{i},\lambda_{j}'),g}(z)$ using the data on $\overline{\mathcal{M}}_{0,3}$.
\end{remark}
Motivated by Equation \eqref{equ_3}, we form the following definition.

\begin{definition}\label{def_Analytic_Base_Matrix} 
    Let $I=\{\lambda_{1},...,\lambda_{l}\}$ be a finite set. Let $\beta,\alpha\in \mathbb{N}^{(I)}$ be arbitrary. Define the $l\times l$ \textit{analytic FA-matrix}, $L_{S^{\bullet}}$ associated with $\beta,\alpha)$ by
    $$
   \left(L_{I,\beta,\alpha}\right)_{(i,j)}:=
    f_{I,\beta,\alpha,(\lambda_{i},\lambda_{j}'),g}(z).
    $$
\end{definition}

\begin{theorem}\label{thm_2} (formal matrix rank formula)
Let $I=\{\lambda_{1},...,\lambda_{l}\}$ be a finite set. Let $\beta,\alpha\in \mathbb{N}^{(I)}$ be arbitrary. Suppose that the equivalent conditions in Theorem \ref{thm_1} hold.
    Then,
\begin{equation}\label{equ_10}
N_{g}(\beta+(n+3)\alpha+\lambda_{i}+\lambda_{j}')=\frac{1}{n!}\left(\left(\frac{\mathcal{N}_{\lambda_{ fo1},0}^{b_{1}}\cdots \mathcal{N}_{\lambda_{l},0}^{b_{l}}\mathcal{N}_{0,1}^{g} \mathcal{N}_{\alpha,0}^{3}}{\mathrm{Id}_{l}-\mathcal{N}_{\alpha,0}z}\right)_{(i,j)}\right)^{(n)}(0),
\end{equation}
for all $n\ge 0$. We call Equation \eqref{equ_10}  the \textit{formal matrix rank formula}.
\end{theorem}
\noindent\textit{Proof.} By Equation \eqref{equ_3},
$$
\begin{aligned}
\left(\frac{\mathcal{N}_{\lambda_{1},0}^{b_{1}}\cdots \mathcal{N}_{\lambda_{l},0}^{b_{l}}\mathcal{N}_{0,1}^{g} \mathcal{N}_{\alpha,0}^{3}}{\mathrm{Id}_{l}-\mathcal{N}_{\alpha,0}z}\right)_{(i,j)}=&f_{I,\beta,\alpha,(\lambda_{i},\lambda_{j}'),g}(z)= \sum_{n=0}^{\infty} N_{g}(\beta+(3+n)\alpha+\lambda_{i}+\lambda_{j}')z^{n}.
\end{aligned}
$$
Applying the Taylor expansion of $f_{I,\beta,\alpha,(\lambda_{i},\lambda_{j}'),g}(z)$ at $z=0$, we get
$$
N_{g}(\beta+(n+3)\alpha+\lambda_{i}+\lambda_{j}')=\frac{1}{n!}\left(\left(\frac{\mathcal{N}_{\lambda_{1},0}^{b_{1}}\cdots \mathcal{N}_{\lambda_{l},0}^{b_{l}}\mathcal{N}_{0,1}^{g} \mathcal{N}_{\alpha,0}^{3}}{\mathrm{Id}_{l}-\mathcal{N}_{\alpha,0}z}\right)_{(i,j)}\right)^{(n)}(0).
$$
\begin{flushright}
    Q.E.D.
\end{flushright}

\section{Vector Bundle of Coinvariants on $\overline{\mathcal{M}}_{g,n}$ and FA-Matrix}\label{sec_5}

In this section, we apply this framework to vector bundles of coinvariants and conformal blocks.

\begin{setup}\label{set_up}
    Let $V$ be a strongly rational VOA, $I=\{W_{1},\cdots,W_{l}\}$ be the collection of irreducible $V$-modules, up to isomorphism. For $n=n_{1}+\cdots+n_{l}$, let
    $$
    N_{g}\left(\sum_{i=1}^{l}n_{i}W_{i}\right) := \mathrm{rank}\mathbb{V}_{g,n}\left(V,\left\{W_{1}^{n_{1}}, \cdots , W_{l}^{ n_{l}}\right\}\right).
    $$
     The factorization theorem holds under this assumption \cite[Theorem 
    7.0.1]{DGT22a}. 
\end{setup}

\begin{remark}
    We may replace the conditions in Setup by any conditions sufficient to invoke the factorization theorem.
\end{remark}

\begin{remark}
    In this set-up, we use $R_{W_{i},g}$ to denote the FA-matrix $\mathcal{N}_{W_{i},g}$.
\end{remark}

Given two tuples of irreducible modules $S^{\bullet}=(S_{1},\cdots,S_{n})$ and $T^{\bullet}=(T_{1},\cdots,T_{m})$, we can define the concatenation $S^{\bullet}T^{\bullet}=(S_{1},\cdots,S_{n}, T_{1},\cdots,T_{n})$. Equivalently, given $S^{\bullet}=(W_{1}^{n_{1}},\cdots, W_{l}^{n_{l}})$ and $T^{\bullet}=(W_{1}^{m_{1}},\cdots, W_{l}^{m_{l}})$, $S^{\bullet}T^{\bullet}=(W_{1}^{n_{1}+m_{1}},\cdots, W_{l}^{n_{l}+m_{l}})$.

\begin{definition}\label{def_FA_matrix}
    Let $V$ be a strongly rational VOA. Let $W_{1}$, $\ldots$, $W_{l}$ all irreducible admissible $V$-modules, up to isomorphism. Let  $S^{\bullet}$ be an $n$-tuple of  irreducible $V$-modules, and $\mathbb{V}_{g,n+2}(V,\{S^{\bullet}\})$ be the vector bundle of coinvariants on $\overline{\mathcal{M}}_{g,n+2}$. Define an $l\times l$-matrix, $R_{V, S^{\bullet},g}$, by
    $$
    \left(R_{V, W^{\bullet},g}\right)_{(i,j)}:=
    \mathrm{rank}\mathbb{V}_{g,n+2}\left(V,\left\{S^{\bullet}, W_{i}, W_{j}'\right\}\right), 
    $$
    where $i,j$ run through all $i=1,...,l, j=1,...,l$. If $V$ is clear from context, we write $R_{S^{\bullet},g}$ instead.
\end{definition}

\begin{lemma}\label{cor_1}
    Let $V$ be a strongly rational VOA, $W_{1},\cdots,W_{l}$ be the collection of irreducible $V$-modules, up to isomorphism. Let $S^{\bullet}, T^{\bullet}$ be two finite tuples of irreducible $V$-modules (not necessarily of the same length). Then, for all $g_{1},g_{2}\ge 0$,
    $$
    R_{S^{\bullet},g_{1}}R_{T^{\bullet},g_{2}}=R_{S^{\bullet} T^{\bullet},g_{1}+g_{2}}.
    $$
\end{lemma}
\noindent\textit{Proof.} This directly follows from Theorem \ref{thm_1}.
\begin{flushright}
    Q.E.D.
\end{flushright}

\begin{theorem}\label{thm_5}
    Let $V$ be a strongly rational VOA,  $S^{\bullet}$ be a finite tuple of irreducible $V$-modules of length $n$. Then, for all $g\ge 0$,
    $$
\mathrm{Tr}\left(R_{S^{\bullet},g}\right) = \mathrm{rank}\mathbb{V}_{g+1,n}(V,\{S^{\bullet}\})
    $$
\end{theorem}
\noindent\textit{Proof.} This directly follows from Theorem \ref{thm_1}.
\begin{flushright}
    Q.E.D.
\end{flushright}

\begin{proposition}\label{proposition_1}
    Let $V$ be a strongly rational VOA, $\mathcal{W}$ be the collection of all irreducible modules, up to isomorphism. Then,
    \begin{equation}\label{equ_25}
    R_{V,1} = \sum_{W\in \mathcal{W}} \mathrm{Tr}(R_{W',0})R_{W,0}.
    \end{equation}
\end{proposition}
\noindent\textit{Proof.} This directly follows from Theorem \ref{thm_1}.
\begin{flushright}
    Q.E.D.
\end{flushright}

Ueno defined the right-hand side of Equation \eqref{equ_25} as the \textit{averaging matrix}. Proposition \ref{proposition_1} shows that the averaging matrix is an FA-matrix. The FA-matrix associated with a single irreducible module at genus zero equals the fusion matrix, since the space of coinvariants on $\overline{\mathcal{M}}_{0,3}$ is isomorphic to the space of intertwining operators. Thus, FA-matrices unites the notion of fusion and averaging matrices.

\begin{theorem}\label{cor_2} (matrix rank formula)
    Let $V$ be a strongly rational VOA. Let $W_{1},\cdots,W_{l}$ be the collection of irreducible $V$-modules, up to isomorphism. Let $W^{\bullet}=(W_{1}^{ n_{1}}, \cdots, W_{l}^{ n_{l}})$. Let $n=n_{1}+\cdots+n_{l}$ Then,
\begin{equation}\label{equ_9}
\begin{aligned}\mathrm{rank}\mathbb{V}_{g,n+2}\left(V,\left\{W_{1}^{ n_{1}} \cdots W_{l}^{ n_{l}},W_{i}, W_{j}'\right\}\right)
= \left(R_{W^{\bullet},g}\right)_{i,j}
=\left(\left(\prod_{i=1}^{n}R_{W_{i},0}^{n_{i}}\right)R_{V,1}^{g}\right)_{(i,j)}.
\end{aligned}
\end{equation} 
Moreover, if $g\ge 1$, we also have
\begin{equation}\label{equ_11}
\begin{aligned}
\mathrm{rank}\mathbb{V}_{g,n}\left(V,\left\{W_{1}^{ n_{1}},\cdots, W_{l}^{ n_{l}}\right\}\right)=\mathrm{Tr}(R_{W^{\bullet},g-1})=\mathrm{Tr} \left( \left(\prod_{i=1}^{n}R_{W_{i},0}^{n_{i}}\right)R_{V,1}^{g-1}\right),
\end{aligned}
\end{equation}
where $\mathrm{Tr}$ denotes the trace of a matrix. We call Equation \eqref{equ_9} and Equation \eqref{equ_11} the \textit{matrix rank formula} for vector bundle of coinvariants on $\overline{\mathcal{M}}_{g,n}$.
\end{theorem}
\noindent\textit{Proof.} This directly follows from Theorem \ref{thm_1}.
\begin{flushright}
    Q.E.D.
\end{flushright}

\begin{remark}
    In case $V$ is the affine VOA $L(\mathfrak{g})_\ell$, where $\mathfrak{g}$ is a simple Lie algebra, and $\ell$ is a positive integer, the rank of a vector bundles of associated coinvariants (or dual bundles of conformal blocks), is typically referred to as the Verlinde formula. Since vector spaces of conformal blocks for such affine VOAs are canonically isomorphic to vector spaces of generalized theta functions, their dimensions are of wide interest. In fact, work on this subject has been carried out by a number of researchers, including \cite{TUY,BL, F3,KN,T,Bea96, DGT22a}. The ranks of vector bundles of coinvariants and conformal blocks for more general VOAs are sometimes referred to as generalized Verlinde formula and Theorem \ref{cor_2} gives a new tool for computing them.
\end{remark}

\begin{proposition}\label{prop_11}
    Let $V$ be a strongly rational VOA. Let $S,T$ be two irreducible modules, and let $S\boxtimes T$ be the tensor product. Then, for all $g\ge 0$,
    $$
    R_{(S,T),g} = R_{S\boxtimes T,g}.
    $$
\end{proposition}
\noindent\textit{Proof.} let $W_{1},\cdots, W_{l}$ be an isomorphism class of all irreducible $V$-modules. Consider the $(i,j)$-entry of $R_{(S,T),0}$.
$$
\begin{aligned}
    \left(R_{S\boxtimes T,0}\right)_{(i,j)} =& \left(R_{\sum_{k=1}^{l}N_{S,T}^{W_{k}}W_{k}}\right)_{(i,j)}
    = \left(\sum_{k=1}^{l}N_{S,T}^{W_{k}} R_{W_{k},0}\right)_{(i,j)}= \sum_{k=1}^{l}N_{S,T}^{W_{k}} \left(R_{W_{k},0}\right)_{(i,j)}\\
    =& \sum_{k=1}^{l}\mathrm{rank}\mathbb{V}_{0,3}\left(V,\left\{S,T,W_{k}'\right\}\right) \mathrm{rank}\mathbb{V}_{0,3}\left(V,\left\{W_{k},W_{i},W_{j}'\right\}\right)\\
    =& \mathrm{rank}\mathbb{V}_{0,4}\left(V,\left\{S,T,W_{i},W_{j}'\right\}\right) \ \ \ \text{(\cite[Factorization Theorem]{DGT22a})}\\
    =& \left(R_{(S,T),0}\right)_{(i,j)}.
\end{aligned}
$$
Thus, we proved the case for $g=0$. For $g\ge 1$, we can apply Theorem \ref{thm_1}:
$$
\begin{aligned}
    R_{S\boxtimes T,g} = R_{S\boxtimes T,0}R_{V,1}^{g}
    = R_{(S,T),0}R_{V,1}^{g}
    = R_{(S,T),g}.
\end{aligned}
$$
\begin{flushright}
    Q.E.D.
\end{flushright}

\begin{proposition}\label{cor_3} (formal matrix rank formula)
    Let $V$ be a strongly rational VOA. Let $W_{1},...,W_{l}
    $ be the collection of irreducible modules, up to isomorphism. Let $S$ be an irreducible module. Let $W^{\bullet}=(W_{i_{1}},\cdots,W_{i_{k}})$ be a finite tuple of irreducible modules. 
    Then,
    \begin{equation}
    f_{W^{\bullet},S^{\bullet},(W_{i},W_{j}'),g}(z) := \sum_{n=0}^{\infty} \mathrm{rank}\mathbb{V}_{g,k+n+5}\left(V,\left\{W^{\bullet}, \left(S^{\bullet}\right)^{ (n+3)}, W_{i} , W_{j}'\right\}\right) z^{n} \in \mathbb{C}[[z]],
    \end{equation}
    is a rational function, for all $g,n\ge 0$. Moreover,
\begin{equation}
\mathrm{rank}\mathbb{V}_{g,k+n+5}\left(V,\left\{W^{\bullet}, S^{n+3}, W_{i}, W_{j}'\right\}\right)=\frac{1}{n!}\left(\left(\frac{R_{W_{i_{1}},0}\cdots R_{W_{i_{k}},0}R_{V,1}^{g} R_{S,0}^{3}}{\mathrm{Id}_{l}-R_{S,0}z}\right)_{(i,j)}\right)^{(n)}(0).
\end{equation}

\end{proposition}
\noindent\textit{Proof.} This directly follows from Theorem \ref{thm_1}.
\begin{flushright}
    Q.E.D.
\end{flushright}

\begin{example}\label{ex_1}
    Let $V$ be a strongly generated VOA. Let $S^{\bullet}=(S_{1}, \cdots, S_{n})$ be an $n$-tuple of irreducible modules. Suppose that $R_{S^{\bullet},0}=\mathrm{Id}$ (equivalently, by Lemma \ref{cor_1}, the product of the fusion matrices of $S_{i}$'s is the identity matrix). Then, $$\mathrm{rank}\mathbb{V}_{1,n}(V,\{S^{\bullet}\})=\text{ the number of irreducible modules, up to isomorphism}.$$
\end{example}
Proof: Let $l$ be the number of irreducible modules, up to isomorphism. Then, the FA-matrices are of size $l\times l$. Then, $\mathrm{rank}\mathbb{V}_{1,n}(V,\{S^{\bullet}\})=\mathrm{Tr}(R_{S^{\bullet},0}) = \mathrm{Tr}(\mathrm{Id}_{l})= l$.
\begin{flushright}
    Q.E.D.
\end{flushright}

Example \ref{ex_1} was first proved in \cite[Example 2.5.1]{DGT3} for the case where $S^{\bullet}$ consists of trivial modules (i.e., $S^{\bullet}=(V, \cdots,V)$).

\section{Vector Bundle of Coinvariants on Tensor Product of VOAs \label{sec_Vector_Bundle_of_Coinvariants_on_Tensor_Product_of_VOAs}}

In this section, we consider the tensor product of two VOAs. Let $V_{1}, V_{2}$ be two strongly rational VOAs. Following \cite{B1,FHL, Milas96}, there is a natural VOA structure on vector space $V_{1}\otimes V_{2}$. Specifically, $V_{1}\otimes V_{2}$ is a natural $\mathbb{N}$-graded vector space, with natural graduation. We define the field associated with $a\otimes b\in V_{1}\otimes V_{2}$ by $Y(a\otimes b,z)=Y(a,z)\otimes Y(b,z)$, with Virasoro element $\omega_{1}\otimes \mathds{1}+\mathds{1}\otimes \omega_{2}$, where $\omega_{1},\omega_{2}$ are the Virasoro elements in $V_{1},V_{2}$, respectively. If $V_{1},V_{2}$ are strongly rational, then $V_{1}\otimes V_{2}$ is also strongly rational \cite{DMZ94,DLM2,Milas96}.

Let $W_{1},...,W_{l}$ be the collection of all irreducible $V_{1}$-modules, and $M_{1},...,M_{t}$ be the collection of all irreducible $V_{2}$-modules, up to isomorphism. $\{W_{i}\otimes M_{j}|1\le i\le l, 1\le j\le t\}\}$ is the collection of all irreducible modules of $V_{1}\otimes V_{2}$, up to isomorphism \cite[Lemma 5.1]{Milas96}. We arrange the set $\{W_{i}\otimes M_{j}|1\le i\le l, 1\le j\le t\}\}$ by lexicographic order on the index. 

Before we formulate the motivating question for this section, we first introduce a notation.

\begin{definition}
    Let $V_{1}, V_{2}$ be strongly rational VOAs, and $S^{\bullet}=(S_{1},\cdots,S_{n})$ (resp. $T^{\bullet}=(T_{1},\cdots,T_{n})$) be an $n$-tuple of $V_{1}$-modules (resp. $V_{2}$-modules). Define $$
    S^{\bullet}\otimes T^{\bullet}:=\left(S_{1}\otimes T_{1},\cdots,S_{n}\otimes T_{n}\right),
    $$
    which is an $n$-tuple of irreducible $V_{1}\otimes V_{2}$-module.
\end{definition}

Our motivating question was raised in \cite{DG} for affine VOAs at positive integral levels. Here, we consider a more general case.
\begin{question}\label{question_1}\cite[Question 1]{DG}
    Let $V_{1}, V_{2}$ be two strongly rational VOAs. Let $S^{\bullet}$ be an $n$-tuple of simple $V_{1}$-modules, $T^{\bullet}$ be an $n$-tuple of simple $V_{2}$-modules. Do we have 
    $$
    \mathbb{V}_{g,n}\left(V_{1}\otimes V_{2},\left\{S^{\bullet}\otimes T^{\bullet}\right\}\right)\simeq \mathbb{V}_{g,n}\left(V_{1},\left\{S^{\bullet}\right\}\right)\otimes \mathbb{V}_{g,n}\left(V_{2},\left\{T^{\bullet}\right\}\right)
    $$ 
    as vector bundles on $\overline{\mathcal{M}}_{g,n}$?
\end{question}

In this section, we prove some special cases and provide evidence for a positive answer.

\begin{lemma}\label{lem_1}
    Let $V_{1}, V_{2}$ be strongly rational VOAs. Let $S^{\bullet}$ (resp. $T^{\bullet})$ be a finite tuple of irreducible $V_{1}$-modules (resp. $V_{2}$-modules). Suppose that $S^{\bullet}$ and $T^{\bullet}$ have the same length. Then,
\begin{equation}\label{equ_24}
    R_{V_{1}\otimes V_{2}, S^{\bullet}\otimes T^{\bullet},g}=R_{V_{1}, S^{\bullet},g}\otimes R_{V_{2}, T^{\bullet},g},
\end{equation}
    for all $g\ge 0$, where the right-hand side of Equation \eqref{equ_24} denotes the usual tensor product of matrices.
\end{lemma}
\noindent\textit{Proof.} The proof is in Appendix \ref{app_lem_1}.

\begin{remark}
    When $S^{\bullet}$ and $T^{\bullet}$ have different lengths, we may add sufficiently many trivial modules to the shorter tuple until they have the same length and then apply the propagation of vacua.
\end{remark}

\begin{theorem}\label{thm_3}
    Let $V_{1}, V_{2}$ be strongly rational VOAs. Let $S^{\bullet}$ (resp. $T^{\bullet})$ be a finite tuple of irreducible $V_{1}$-modules (resp. $V_{2}$-modules). Suppose that $S^{\bullet}$ and $T^{\bullet}$ have the same length. Then, for all $g,n\ge 0$,
    $$
\begin{aligned}\mathrm{rank}\mathbb{V}_{g,n}\left(V_{1}\otimes V_{2},\left\{S^{\bullet}\otimes T^{\bullet}\right\}\right)=&\mathrm{rank}\left(\mathbb{V}_{g,n}\left(V_{1},\left\{S^{\bullet}\right\}\right)\otimes \mathbb{V}_{g,n}\left(V_{2},\left\{T^{\bullet}\right\}\right)\right)\\
=&\mathrm{rank}\mathbb{V}_{g,n}\left(V_{1},\left\{S^{\bullet}\right\}\right)\cdot \mathrm{rank} \mathbb{V}_{g,n}\left(V_{2},\left\{T^{\bullet}\right\}\right).
    \end{aligned}
    $$
\end{theorem}
\noindent\textit{Proof.} Let $V_{1}=W_{1}, \cdots, W_{l}$ be the collection of all irreducible $V_{1}$-modules, up to isomorphism. Let $V_{2}=M_{1}, \cdots, M_{l'}$ be the collection of all irreducible $V_{2}$-modules, up to isomorphism. Then,
$$
\begin{aligned}
\mathrm{rank}\left(\mathbb{V}_{g}\left(V_{1},\left\{S^{\bullet}\right\}\right)\otimes \mathbb{V}_{g,n}\left(V_{2},\left\{T^{\bullet}\right\}\right)\right) =& \mathrm{rank}\mathbb{V}_{g,n}\left(V_{1},\left\{S^{\bullet}\right\}\right) \cdot \mathrm{rank}\mathbb{V}_{g,n}\left(V_{2},\left\{T^{\bullet}\right\}\right)\\
    =& \left(R_{V_{1},S^{\bullet},g}\right)_{(1,1)}\cdot \left(R_{V_{2},T^{\bullet},g}\right)_{(1,1)}\ \ \ \ \text{(Theorem \ref{thm_1})}\\
    =& \left(R_{V_{1},S^{\bullet},g}\otimes R_{V_{2},T^{\bullet},g}\right)_{1,1} \ \ \ \  \text{(matrix Kronecker product formula)}\\
    =& \left(R_{V_{1}\otimes V_{2},S^{\bullet}\otimes T^{\bullet},g}\right)_{1,1} \ \ \ \text{(Lemma \ref{lem_1})}\\
    =& \mathrm{rank}\mathbb{V}_{g,n}\left(V_{1}\otimes V_{2},\left\{S^{\bullet}\otimes T^{\bullet}\right\}\right)\ \ \ \ \text{(Theorem \ref{thm_1}).}
\end{aligned}
$$
\begin{flushright}
    Q.E.D.
\end{flushright}

\begin{theorem}\label{thm_4}
    Let $V_{1}, V_{2}$ be VOAs of CohFT-type. Let $S^{\bullet}$ (resp. $T^{\bullet})$ be a finite tuple of irreducible $V_{1}$-modules (resp. $V_{2}$-modules). Suppose that $S^{\bullet}$ and $T^{\bullet}$ have the same length. Then,
    $$
    c_{1}\mathbb{V}_{g,n}\left(V_{1}\otimes V_{2}, \left\{S^{\bullet}\otimes T^{\bullet}\right\}\right) = c_{1}\left(\mathbb{V}_{g,n}\left(V_{1},\left\{S^{\bullet}\right\}\right)\otimes \mathbb{V}_{g,n}\left(V_{2},\left\{T^{\bullet}\right\}\right)\right).
    $$
    
\end{theorem}
\noindent\textit{Proof.} The proof is in Appendix \ref{Appendix_B}.

\begin{corollary}\label{cor_4}
Let $V_{1},V_{2}$ be two VOAs of CohFT-type. Let $S^{\bullet}$ be an $n$-tuple of irreducible $V_{1}$-module, $T^{\bullet}$ be an $n$-tuple of irreducible $V_{2}$-modules. Suppose that $\mathbb{V}_{g,n}\left(V_{1},\left\{S^{\bullet}\right\}\right),\mathbb{V}_{g,n}\left(V_{2},\left\{T^{\bullet}\right\}\right) $ are both line bundles on $\overline{\mathcal{M}}_{g,n}$. Then,
    $$
    \mathbb{V}_{g,n}\left(V_{1}\otimes V_{2}, \left\{S^{\bullet}\otimes T^{\bullet}\right\}\right) \simeq  \mathbb{V}_{g,n}\left(V_{1},\left\{S^{\bullet}\right\}\right)\otimes \mathbb{V}_{g,n}\left(V_{2},\left\{T^{\bullet}\right\}\right).
    $$
\end{corollary}
\noindent\textit{Proof.} This directly follows from Theorem \ref{thm_4}, since on $\overline{\mathcal{M}}_{g,n}$, two line bundles are isomorphic if and only if their first Chern classes are the same.
\begin{flushright}
    Q.E.D.
\end{flushright}

\begin{corollary}\label{cor_5}
Let $V_{1},V_{2}$ be two VOAs of CohFT-type. Let $S^{\bullet}$ be an $n$-tuple of irreducible $V_{1}$-modules, $T^{\bullet}$ be an $n$-tuple of irreducible $V_{2}$-modules. Then, on $\mathcal{M}_{g,n}$, we have
 $$
\mathrm{Ch}\left(\mathbb{V}_{g,n}\left(V_{1}\otimes V_{2}, \left\{S^{\bullet}\otimes T^{\bullet}\right\}\right)\right) = \mathrm{Ch}\left(\mathbb{V}_{g,n}\left(V_{1},\left\{S^{\bullet}\right\}\right)\otimes \mathbb{V}_{g,n}\left(V_{2},\left\{T^{\bullet}\right\}\right)\right),
 $$
 for all $g,n \ge 0$, where $\mathrm{Ch}$ is the Chern character.
\end{corollary}
\noindent\textit{Proof.}
Let $c_{1},c_{2}$ be the central charge of $V_{1},V_{2}$, respectively. Let $S^{\bullet} = (S_{1}, \cdots, S_{n}), T^{\bullet}=(T_{1}, \cdots, T_{n})$. Let $a_{1},\cdots,a_{n}$ be the conformal weights of $S_{1},\cdots,S_{n}$, and $b_{1},\cdots,b_{n}$ be the conformal weights of $T_{1},\cdots,T_{n}$. Thus, the central charge of $V_{1}\otimes V_{2}$ is $c_{1}+c_{2}$, and the conformal weight of $S_{i}\otimes T_{i}$ is $a_{i}+b_{i}$, for all $i\in \{1,\cdots,n\}$. In the following proof, we use $V_{12}=V_{1}\otimes V_{2}$, and $\mathbb{V}$ for $\mathbb{V}_{g,n}$, and $\mathrm{rk}$ for $\mathrm{rank}$.
$$
\begin{aligned}
\mathrm{Ch}\left(\mathbb{V}\left(V_{12}, \left\{S^{\bullet}\otimes T^{\bullet}\right\}\right)\right)\overset{\text{\cite[Cor. 9.1]{DGT22b}}}{=}&\mathrm{rk}\mathbb{V}\left(V_{12}, \left\{S^{\bullet}\otimes T^{\bullet}\right\}\right)\exp\left(\frac{c_{1}+c_{2}}{2}\lambda+\sum_{i=1}^{n}(a_{i}+b_{i})\right)\\
 \overset{\text{Thm \ref{thm_3}}}{=}& \mathrm{rk}\mathbb{V}\left(V_{1},\left\{S^{\bullet}\right\}\right) \mathrm{rk} \mathbb{V}\left(V_{2},\left\{T^{\bullet}\right\}\right)\exp\left(\frac{c_{1}}{2}\lambda+\sum_{i=1}^{n}a_{i}\right)\\
 &\exp\left(\frac{c_{2}}{2}\lambda+\sum_{i=1}^{n}b_{i}\right)\\
=&\left(\mathrm{rk}\mathbb{V}\left(V_{1},\left\{S^{\bullet}\right\}\right)\exp\left(\frac{c_{1}}{2}\lambda+\sum_{i=1}^{n}a_{i}\right)\right)\\
 &\left(\mathrm{rk}\mathbb{V}\left(V_{2},\left\{T^{\bullet}\right\}\right)\exp\left(\frac{c_{2}}{2}\lambda+\sum_{i=1}^{n}b_{i}\right)\right)\\
 \overset{\text{\cite[Cor. 9.1]{DGT22b}}}{=}&\mathrm{Ch}\left(\mathbb{V}_{g,n}\left(V_{1},\left\{S^{\bullet}\right\}\right)\right)\mathrm{Ch}\left( \mathbb{V}_{g,n}\left(V_{2},\left\{T^{\bullet}\right\}\right)\right)\\
=&\mathrm{Ch}\left(\mathbb{V}_{g,n}\left(V_{1},\left\{S^{\bullet}\right\}\right)\otimes \mathbb{V}_{g,n}\left(V_{2},\left\{T^{\bullet}\right\}\right)\right)
\end{aligned}
$$
\begin{flushright}
 Q.E.D.
\end{flushright}
 

\begin{corollary}\label{cor_7}
    Let $V_{1}, V_{2}$ be two VOAs of CohFT-type. Let $S^{\bullet}$ be an $n$-tuple of irreducible $V_{1}$-module, and $T^{\bullet}$ be an $n$-tuple of irreducible $V_{2}$-modules. Suppose that $\mathbb{D}_{g,n}(V_{1},\{S^{\bullet}\})$ and $\mathbb{D}_{g,n}(V_{2},\{T^{\bullet}\})$ are both nef (resp. F-nef), then $\mathbb{D}_{g,n}(V_{1}\otimes V_{2},\{S^{\bullet}\otimes T^{\bullet}\})$ is nef (resp. F-nef).
\end{corollary}
\noindent\textit{Proof.}
$$
\begin{aligned}
    \mathbb{D}_{g,n}(V_{1}\otimes V_{2},\{S^{\bullet}\otimes T^{\bullet}\}) =& c_{1}\mathbb{V}_{g,n}(V_{1}\otimes V_{2},\{S^{\bullet}\otimes T^{\bullet}\})\\
    =& c_{1}\left(\mathbb{V}_{g,n}(V_{1},\{S^{\bullet}\})\otimes \mathbb{V}_{g,n}(V_{2},\{T^{\bullet}\})\right) (\text{Theorem \ref{thm_4}})\\
    =&\mathrm{rk}\mathbb{V}_{g,n}(V_{1},\{S^{\bullet}\})\mathbb{D}_{g,n}(V_{2},\{T^{\bullet}\})+\mathrm{rk}\mathbb{V}_{g,n}(V_{2},\{T^{\bullet}\})\mathbb{D}_{g,n}(V_{1},\{S^{\bullet}\}).
\end{aligned}
$$
A linear combination of nef (resp. F-nef) divisors, with non-negative coefficients, is nef (resp. F-nef).
\begin{flushright}
    Q.E.D.
\end{flushright}

\section{Pointed vertex operator algebra}

\begin{definition}
    A finite tensor category is called \textit{pointed} if all its simple objects are invertible.
\end{definition}

\begin{definition}\cite{GR}
    A \textit{pointed VOA} is a strongly rational VOA such that $\mathbf{Mod}_{V}$ is a pointed category. Equivalently, it is a strongly rational VOA whose fusion ring is a group ring.
\end{definition}
Let $V$ be a pointed VOA, with fusion ring $\mathbb{Z}[G]$.  There is a bijection between isomorphism classes of irreducible $V$-modules and elements in $G$. For each $g\in G$, we write $W_{g}$ as isomorphism class of irreducible modules corresponding to $g$. We write the operation in $G$ multiplicatively.

\subsection{Ranks of vector bundles of coinvariants associated with pointed VOAs}

\begin{definition}
    Let $C_{p}$ be a cyclic group of $p\ge 2$, and $\sigma\in C_{p}$ be a permutation. Let $\mathrm{Id}=(E_{1},\cdots, E_{n})$ be the identity matrix of size $p$, written in column vector form. Define $R_{\sigma}=(E_{\sigma(1)},\cdots,E_{\sigma(n)})$.
\end{definition}

\begin{proposition}\label{prop_6}
    Let $V$ be a pointed vertex operator algebra, with fusion ring $\mathbb{Z}[G]$. Let $W^{\bullet}=(W_{x_{1}},\cdots,W_{x_{n}})$ be an $n$-tuple of irreducible $V$-modules, where $x_{i}\in G$.  Then,
    $$
    \mathrm{rank}\mathbb{V}_{g,n}\left(V,\left\{W^{\bullet}\right\}\right)=|G|^{g}\delta_{\prod_{i=1}^{n}x_{i},e},
    $$
    where the multiplication is in $G$, and $e\in G$ is the multiplicative identity.
\end{proposition}
\noindent\textit{Proof.} By Theorem \ref{thm_1}, the FA-matrix defines a representation of the group ring $\mathbb{Z}$. That is, for $x\in G$, we have a ring homomorphism
$
    R_{-,0}: \mathbb{Z}[G] \longrightarrow \mathrm{Mn}(\mathbb{R}),
$ by $x \longmapsto R_{W_{x},0}$.
For $x,y\in G$, 
$
R_{W_{x},0}R_{W_{y},0} = R_{W_{x}\boxtimes W_{y},0}=R_{W_{xy},0}.
$
Thus,
$
\begin{aligned}
    \prod_{i=1}^{n} R_{W_{i},0} = R_{W_{\prod_{i=1}^{n}x_{i}},0}.
\end{aligned}
$
Thus, we have
$$
\begin{aligned}
\mathrm{rank}\mathbb{V}_{0,n}\left(V,\left\{W_{x_{1}},\cdots,W_{x_{n}}\right\}\right)=&
\mathrm{rank}\mathbb{V}_{0,n+2}\left(V,\left\{W_{x_{1}},\cdots,W_{x_{n}},V,V\right\}\right)\\
=&\mathrm{rank}\mathbb{V}_{0,3}\left(V,\left\{W_{x_{i}\cdots x_{n}},V,V\right\}\right)\\
=&\delta_{W_{x_{i}\cdots x_{n},V}}= \delta_{x_{1}\cdots x_{n},e}.
\end{aligned}
$$
Therefore, we proved the case for $g=0$. 

Using standard linear algebra result, we may check
$
\begin{aligned}
R_{V,1} =& \sum_{s\in G}\mathrm{Tr}\left(R_{W_{s^{-1}},0}\right)R_{W_{s},0}= |G|\cdot \mathrm{Id}_{|G|}.
\end{aligned}
$
Therefore, by Theorem \ref{thm_1}, 
$
\mathrm{rank}\mathbb{V}_{g,n}\left(V,\left\{W_{x_{1}},\cdots,W_{x_{n}}\right\}\right)=|G|^{g}\delta_{x_{1}\cdots x_{n},e}.
$
\begin{flushright}
    Q.E.D.
\end{flushright}

For a pointed VOA, the rank is either $|G|^{g}$ or zero, depending whether $\prod_{i}^{n} x_{i} = e$. Without loss of generality, from now on, we assume $\prod_{i}^{n} x_{i} = e$. When $V$ is a lattice VOA associated with a positive-definite, even lattice $L$, $G=L'/L$ is the discriminant group and is often written additively.

\begin{corollary}\label{cor_6}
    The vector bundle of coinvariants for a pointed VOA on $\overline{\mathcal{M}}_{g,n}$ is either zero, or a vector bundle of rank $m^{g}$, where $m$ is the number of irreducible modules, up to isomorphism. Thus, on $\overline{\mathcal{M}}_{0,n}$, it is either zero or a line bundle.
\end{corollary}
\noindent\textit{Proof.} This directly follows from Proposition \ref{prop_6}.
\begin{flushright}
    Q.E.D.
\end{flushright}

    Damiolini--Gibney--Tarasca computed a similar formula for a lattice VOA whose discriminant group is cyclic
    \cite[Example 5.2.5]{DGT3}. Since all lattice VOAs are pointed, we generalize this result.

\subsection{Coinvariants and conformal block divisor associated with a pointed VOA.}

\begin{definition}
    Let $V$ be a pointed vertex operator algebra. Let $W_{1},\cdots,W_{m}$ be the collection of irreducible $V$-modules, up to isomorphism. Let $a_{W_{i}}$ be the conformal weight of $W_{i}$, for all $i$. Define 
    $$
    a_{\mathrm{average}} = \frac{1}{m}\left(\sum_{i=1}^{m}a_{W_{i}}\right),
    $$
    and we call it the \textit{average conformal weight}. If the conformal weights are all real numbers, define
    $$
    a_{\max} = \max \left\{a_{W_{i}}\right\}_{i=1}^{m}
    $$
    and we call it the \textit{maximal conformal weight}.
\end{definition}

\begin{proposition}\label{prop_5}
   Let $V$ be a pointed VOA, with central charge $c$ and fusion ring $\mathbb{Z}[G]$. Let $m=|G|$, and $\beta_{1},\cdots,\beta_{n}\in G$ be arbitrary such that $\prod_{i=1}^{n}\beta_{i}=e$. Then,
    \begin{equation}\label{equ_12}
    \mathbb{D}_{g,n}(V,\{W_{\beta_{1}},\cdots,W_{\beta_{n}}\}) = m^{g}\left(\frac{c}{2}\lambda+\sum_{i=1}^{n}a_{\beta_{i}}\psi_{i}\right)-b_{irr}\delta_{irr}-\sum_{i,I}b_{i:I}\delta_{i:I},
    \end{equation}
    where
    \begin{equation}
    b_{irr}=m^{g}a_{\mathrm{average}},
    \end{equation}
    and
    \begin{equation}
    b_{i:I}=m^{g}a_{\prod_{j\in I} \beta{j}}.
    \end{equation}
\end{proposition}
\noindent\textit{Proof.} 
Equation \eqref{equ_12} is the formula for the coinvariant divisor in \cite[Corollary 2]{DGT3}. Now, we compute the coefficient $b_{irr}$ using the rank formula in Proposition \ref{prop_6}.

$$
\begin{aligned}
    b_{irr} =& \sum_{\gamma\in G}a_{\gamma}\mathrm{rank}\mathbb{V}_{g-1,n}\left(V,\{W_{\beta_{1}},\cdots,W_{\beta_{n}},W_{\gamma},W_{\gamma}'\}\right) \ \ \ \ \text{\cite[Corollary 2]{DGT3}}\\
    =& \sum_{\gamma\in G}a_{\gamma}\mathrm{rank}\mathbb{V}_{g-1,n}\left(V,\{W_{\beta_{1}},\cdots,W_{\beta_{n}},W_{\gamma},W_{\gamma^{-1}}\}\right)\\
    =& \sum_{\gamma\in G}a_{\gamma}m^{g-1}= m^{g-1}\left(\sum_{\gamma\in G}a_{\gamma}\right)= m^{g}a_{\mathrm{average}}.
\end{aligned}
$$

$$
    \begin{aligned}
        b_{i:I} =& \sum_{\gamma\in G}a_{\gamma}\mathrm{rank}\mathbb{V}_{i,|I|+1}\left(V,\left\{W_{\beta}^{I},W_{\gamma}\right\}\right)\mathrm{rank}\mathbb{V}_{g-i,|I|+1}\left(V,\left\{W^{I^{c}},W_{\gamma}'\right\}\right) \text{\cite{DGT3}}\\
        =& a_{\prod_{j\in I} \beta_{j}^{-1}}\cdot m^{i} \cdot m^{g-i}= m^{g}a_{\prod_{j\in I} \beta_{j}^{-1}} = m^{g}a_{\prod_{j\in I} \beta_{j}}.
    \end{aligned}
$$
The last equality is because an irreducible module and its dual module has the same conformal weight.
\begin{flushright}
    Q.E.D.
\end{flushright}

\subsection{Positivity of coinvariant divisors associated with a pointed VOA}
On $\overline{\mathcal{M}}_{g,n}$, there are six types of F-curves, where some of them do not occur in low genus cases. Following \cite{GKM}, we call them \textit{type $t$ F-curves} or \textit{F-curves of type $t$}, where $t\in \{1,\cdots,6\}$.

\begin{definition}
    For $t\in \{1,\cdots,6\}$, let $F_{t}(\overline{\mathcal{M}}_{g,n})$ be the collection of divisors that intersect all type $t$ F-curves non-negatively. Let $
    F\left(\overline{\mathcal{M}}_{g,n}\right)=\bigcap_{t=1}^{6}F_{t}\left(\overline{\mathcal{M}}_{g,n}\right). 
    $
    We call divisors in $F_{t}\left(\overline{\mathcal{M}}_{g,n}\right)$ the \textit{$F_{t}$-divisors}, and we call divisors in $F\left(\overline{\mathcal{M}}_{g,n}\right)$ the \textit{F-divisor}.
\end{definition}

\begin{lemma}\label{lem_3}
    Let $V$ be a pointed VOA. Let $S^{\bullet}$ be an $n$-tuple of irreducible modules. Then,
    \begin{enumerate}
        \item $\forall g\ge 1, \mathbb{D}_{g,n}(V,\{S^{\bullet}\})\in F_{1}(\overline{\mathcal{M}}_{g,n}) \iff \exists g\ge 1, \mathbb{D}_{g,n}(V,\{S^{\bullet}\})\in F_{1}(\overline{\mathcal{M}}_{g,n}),$
        \item $\forall g\ge 3, \mathbb{D}_{g,n}(V,\{S^{\bullet}\})\in F_{2}(\overline{\mathcal{M}}_{g,n}) \iff \exists g\ge 3, \mathbb{D}_{g,n}(V,\{S^{\bullet}\})\in F_{2}(\overline{\mathcal{M}}_{g,n}),$
        \item $\forall g\ge 2, \mathbb{D}_{g,n}(V,\{S^{\bullet}\})\in F_{3}(\overline{\mathcal{M}}_{g,n}) \iff \exists g\ge 2, \mathbb{D}_{g,n}(V,\{S^{\bullet}\})\in F_{3}(\overline{\mathcal{M}}_{g,n}),$
        \item $\forall g\ge 2, \mathbb{D}_{g,n}(V,\{S^{\bullet}\})\in F_{4}(\overline{\mathcal{M}}_{g,n}) \iff \exists g\ge 2, \mathbb{D}_{g,n}(V,\{S^{\bullet}\})\in F_{4}(\overline{\mathcal{M}}_{g,n}),$
        \item $\forall g\ge 1, \mathbb{D}_{g,n}(V,\{S^{\bullet}\})\in F_{5}(\overline{\mathcal{M}}_{g,n}) \iff \exists g\ge 1, \mathbb{D}_{g,n}(V,\{S^{\bullet}\})\in F_{5}(\overline{\mathcal{M}}_{g,n})$,
        \item $\forall g\ge 0, \mathbb{D}_{g,n}(V,\{S^{\bullet}\})\in F_{6}(\overline{\mathcal{M}}_{g,n}) \iff \exists g\ge 0, \mathbb{D}_{g,n}(V,\{S^{\bullet}\})\in F_{6}(\overline{\mathcal{M}}_{g,n}).$
    \end{enumerate}
\end{lemma}
\noindent\textit{Proof.} 
The proof of all six statements are similar. To illustrate, we prove (a).
Let $g\ge 1$ be arbitrary, and $S^{\bullet}$ be an arbitrary $n$-tuple of irreducible $V$-modules. We first write the coinvariant divisor in the form of Proposition \ref{prop_5}. Let $a$ be the coefficient of $\lambda$ class, and $c$ be the central charge of $V$. Then,
$$
\begin{aligned}
    a-12b_{irr}+b_{1,\emptyset} = \frac{m^{g}c}{2} -12m^{g}a_{\mathrm{average}}+0= \frac{m^{g}}{2}\left(c-24a_{\mathrm{average}}\right).
\end{aligned}
$$
By \cite[Theorem 2.1]{GKM}, for $g\ge 1$,
$$
a-12b_{irr}+b_{1,\emptyset} \ge 0 \iff \mathbb{D}_{g,n}(V,\{S^{\bullet}\})\in F_{1}(\overline{\mathcal{M}}_{g,n}).
$$
The statement follows since the positivity of the left-hand side is independent of $g$.
\begin{flushright}
    Q.E.D.
\end{flushright}

\begin{theorem}\label{thm_6}
    Let $V$ be a pointed VOA, and $S^{\bullet}$ be an $n$-tuple of irreducible modules. Then, the following are equivalent:
    \begin{enumerate}
        \item $\forall \, g\ge 3, \mathbb{D}_{g,n}(V,\{S^{\bullet}\}) \text{ is F-nef}$,
        \item $\exists\,  g\ge 3, \mathbb{D}_{g,n}(V,\{S^{\bullet}\})\text{ is F-nef.}$
    \end{enumerate}
    If the averaging conformal weight of all irreducible $V$-modules, up to isomorphism, is non-negative (e.g., if $V$ is unitary), then the following are equivalent:
    \begin{enumerate}
        \item $\forall\, g\ge 2, \mathbb{D}_{g,n}(V,\{S^{\bullet}\}) \text{ is F-nef}$,
        \item $\exists\,  g\ge 2, \mathbb{D}_{g,n}(V,\{S^{\bullet}\})\text{ is F-nef.}$
    \end{enumerate}
\end{theorem}
\noindent\textit{Proof.} This directly follows from Lemma \ref{lem_3}.
\begin{flushright}
    Q.E.D.
\end{flushright}
\begin{remark}
    For Theorem \ref{thm_6}, the same statement is true for conformal block divisors $-\mathbb{D}_{g,n}(V,\{S^{\bullet}\})$. 
\end{remark}

Next, we prove some equivalent and sufficient conditions for a coinvariant divisor to intersect a particular type of F-curves non-negatively.

\begin{proposition}\label{prop_9}
Let $V$ be a pointed VOA, with central charge $c$. The following are equivalent:
    \begin{enumerate}
        \item $c\ge 24 \cdot a_{\mathrm{average}}$,
        \item  There exists a coinvariant divisor that intersects type 1 F-curves non-negatively \footnote{Type 1 F-curves exist only on $\overline{\mathcal{M}}_{g,n}$ for $g\ge 1$. Thus, this statement implicitly assumes that $g\ge 1$.},
       \item  All coinvariant divisors intersect type 1 F-curves non-negatively.
    \end{enumerate}
\end{proposition}
\noindent\textit{Proof.} Let $g\ge 1$ be arbitrary, and $S^{\bullet}$ be an arbitrary $n$-tuple of irreducible $V$-modules. We first write the coinvariant divisor $\mathbb{D}_{g,n}(V,\{S^{\bullet}\})$ in the form of Proposition \ref{prop_5}. Let $a$ be the coefficient of $\lambda$ class, and $c$ be the central charge of $V$. Then,
$$
\begin{aligned}
    a-12b_{irr}+b_{1,\emptyset} = \frac{m^{g}c}{2} -12m^{g}a_{\mathrm{average}}+0= \frac{m^{g}}{2}\left(c-24a_{\mathrm{average}}\right).
\end{aligned}
$$
By \cite[Theorem 2.1]{GKM}, for $g\ge 1$, we have
$$
\mathbb{D}_{g,n}(V,\{S^{\bullet}\})\in F_{1}(\overline{\mathcal{M}}_{g,n})\iff a-12b_{irr}+b_{1,\emptyset} \ge 0 \iff c\ge 24a_{\mathrm{average}}.
$$
 The statement follows since the inequality $c\ge 24 a_{\mathrm{average}}$ only depends on the VOA, and does not depend on genus or the choice of irreducible modules on each marked point.
\begin{flushright}
    Q.E.D.
\end{flushright}

\begin{proposition}\label{prop_10}
For a $V$ be a pointed VOA, the following are equivalent:
    \begin{enumerate}
        \item $ a_{\mathrm{average}}\ge 0$,
        \item  There exists a coinvariant divisor that intersects type 2 F-curves non-negatively\footnote{Type 2 F-curves exist only on $\overline{\mathcal{M}}_{g,n}$ for $g\ge 3$. Thus, this statement implicitly assumes that $g\ge 3$.},
       \item  All coinvariant divisors intersect type 2 F-curves non-negatively.
    \end{enumerate}
\end{proposition}
\noindent\textit{Proof.} Let $g\ge 3$ be arbitrary, and $S^{\bullet}$ be an arbitrary $n$-tuple of irreducible $V$-modules. Since
$
    b_{irr} = m^{g}a_{\mathrm{average}},
$
by \cite[Theorem 2.1]{GKM}, for $g\ge 3$, we have
$$
\mathbb{D}_{g,n}(V,\{S^{\bullet}\})\in F_{2}(\overline{\mathcal{M}}_{g,n})\iff b_{\mathrm{irr}} \ge 0 \iff a_{\mathrm{average}}\ge 0.
$$
 The statement follows since the inequality $a_{\mathrm{average}}\ge 0$ only depends on the VOA, and does not depend on genus or the choice of irreducible modules on each marked point.
\begin{flushright}
    Q.E.D.
\end{flushright}

\begin{proposition}
Let $V$ be a pointed VOA such that the conformal weight of all irreducible modules are non-negative. Then, all coinvariant divisors intersect type 3 F-curves non-negatively on $\overline{\mathcal{M}}_{g,n}$, for all $g\ge 2$, and all $n\ge 0$.
\end{proposition}
\noindent\textit{Proof.} This directly follows from Lemma \ref{lem_3} and \cite[Theorem 2.1]{GKM}.
\begin{flushright}
    Q.E.D.
\end{flushright}

\begin{proposition}
Let $V$ be a pointed VOA, with $2a_{\mathrm{average}} \ge  a_{\max}$. Then, all coinvariant divisors intersect type 4 F-curves non-negatively on $\overline{\mathcal{M}}_{g,n}$, for $g\ge 2$, for all $n\ge 0$.
\end{proposition}  
\noindent\textit{Proof.} This directly follows from Lemma \ref{lem_3} and \cite[Theorem 2.1]{GKM}.
\begin{flushright}
    Q.E.D.
\end{flushright}

\begin{proposition}
    Let $V$ be a pointed VOA, with fusion ring $\mathbb{Z}[G]$. Suppose the conformal weight map
    $$
    \begin{aligned}
    a: G \longrightarrow& \mathbb{R}\\
    \gamma \longmapsto& a_{\gamma}
    \end{aligned}
    $$
    is sub-additive. That is,  $a_{\gamma_{1}}+a_{\gamma_{2}}\ge a_{\gamma_{1}\gamma_{2}}$, for all $\gamma_{1},\gamma_{2}\in G$. 
    Then, all coinvariant divisors intersect type 5 F-curves non-negatively on $\overline{\mathcal{M}}_{g,n}$, for $g\ge 1$, for all $n$.
\end{proposition}
\noindent\textit{Proof.} This directly follows from Lemma \ref{lem_3} and \cite[Theorem 2.1]{GKM}.
\begin{flushright}
    Q.E.D.
\end{flushright}

\begin{proposition}
    Let $V$ be a pointed VOA that is strongly generated in degree 1. Then, all coinvariant divisors intersects type 6 F-curves non-negatively on $\overline{\mathcal{M}}_{g,n}$, for $g\ge 0, n\ge 0$.
\end{proposition}
\noindent\textit{Proof.} This directly follows from Lemma \ref{lem_3} and \cite[Theorem 1]{DG}.
\begin{flushright}
    Q.E.D.
\end{flushright}

Among all the coinvariant divisors, $\mathbb{D}_{g,1}(V,\{V\})$ is of particular interest. It is obtained by assigning the trivial module to the unique marked point of $\overline{\mathcal{M}}_{g,1}$. It is of interest because coinvariant divisors on $\overline{\mathcal{M}}_{g}$ are defined this way via propagation of vacua.

\begin{proposition}\label{prop_8}
Let $V$ be pointed VOA, with central charge $c$. Then, for $g\ge 1$, $$\mathbb{D}_{g,1}(V,\{V\}) \text{ is F-nef} \iff c \ge 24a_{\mathrm{average}}\ge 0.$$
\end{proposition}
\noindent\textit{Proof.}
By Lemma \ref{lem_3} and \cite[Theorem 2.1]{GKM}, $\mathbb{D}_{g,1}(V,\{V\}) $ intersects F-curves of type 3-6 non-negatively, when they exist. By Proposition \ref{prop_10}, $\mathbb{D}_{g,1}(V,\{V\}) $ intersects type 2 F-curves non-negatively if and only if $a_{\mathrm{average}}\ge 0$. By Proposition \ref{prop_9}, $\mathbb{D}_{g,1}(V,\{V\}) $ intersects type 1 F-curves non-negatively if and only if $c \ge 24a_{\mathrm{average}}$.
\begin{flushright}
    Q.E.D.
\end{flushright}

Next, we consider the case where the coinvariant divisor associated is not F-nef. We assume that $a_{\mathrm{average}}\ge 0$. In this case, $\mathbb{D}_{g,1}(V,\{V\})$ intersects F-curves of type 2-6 non-negatively, whenever such curves exist. By \cite{GKM}, $\mathbb{D}_{g,1}(V,\{V\}) +r\lambda$ is F-nef, for $r$ sufficiently large. However, if we arbitrarily add the Hodge class $\lambda$, then the resulting divisor may not be coinvariant divisor, and thus may not have geometric interpretation. The next proposition gives us a way to add $\lambda$ so that the resulting divisor remains a coinvariant divisor.

\begin{proposition}
    Let $V$ be a pointed VOA with $a_{\mathrm{average}}\ge 0$. Let $H$ be a holomorphic VOA. Then, $\exists r_{0}\ge 0,\forall r\ge r_{0}, \forall g\ge 1, \mathbb{D}_{g,1}(V\otimes H^{\otimes r},\{V\otimes H^{\otimes r}\})$ is F-nef. In particular, $r_{0} = \left\lceil\frac{24a_{\mathrm{average}}-c_{V}}{c_{H}}\right\rceil$ is the best lower bound.
\end{proposition}

\noindent\textit{Proof.} Let $c_{V}$ be the central charge of $V$. Then, the central charge of $V\otimes H^{\otimes r}$ is $c_{V}+rc_{H}$. By Proposition \ref{prop_9}, $\mathbb{D}_{g,1}(V\otimes H^{\otimes r},\{V\otimes H^{\otimes r}\})$ is F-nef if and only if $c_{V\otimes H^{\otimes r}}\ge a_{\mathrm{average}}$, which is equivalent to $r\ge \left\lceil\frac{24a_{\mathrm{average}}-c_{V}}{c_{H}}\right\rceil$. 
\begin{flushright}
    Q.E.D.
\end{flushright}

A root lattice VOA is a pointed VOA that is strongly generated in degree $1$, so its coinvariant divisors are globally generated on $\overline{\mathcal{M}}_{0,n}$ \cite{DG}. Therefore, we next consider its F-nefness in the positive genus case. Any positive-definite, even, integral root lattice is a direct sum of root lattices of type $A_{r} (r\ge 1), D_{k} (k\ge 4), E_{l} (l=6,7,8)$\footnote{This is a classis result in lattice theory, which can be found in many standard textbooks (e.g., \cite[Theorem 2.25]{SS}).}. We test the F-nefness on each class.

\begin{example}\label{example_2}
    Let $L$ be an $A_{r}, D_{k} $ or $E_{l} $ root lattices other than $E_{8}$. Then, all coinvariant divisors associated with the lattice VOA $V_{L}$ on $\overline{\mathcal{M}}_{g,n}$ is not F-nef, in the positive genus case. In particular, they all intersect type 1 F-curves negatively.
\end{example}
\noindent\textit{Proof.} This directly follows from Proposition \ref{prop_9}, since the central charge of root lattice VOA and the conformal weight of all its irreducible modules have been computed.
\begin{flushright}
    Q.E.D.
\end{flushright}

\begin{remark}
    $E_{8}$ root lattice has trivial discriminant group, so the lattice VOA is holomorphic. By \cite[Example 5.1.1]{DGT3}, the coinvariant divisors is F-nef. Example \ref{example_2} may have been already known. Here, we gave an alternative proof.
\end{remark}

\subsection{Examples of Pointed VOAs}\label{PointedExamples}
In this subsection, we present examples of pointed VOAs.

\begin{example}\label{exHol}
    Holomorphic VOAs, such as the moonshine module, are pointed.
\end{example}

\begin{example}\label{exLattice}
    Lattice VOAs are pointed \cite{FLM2,B1,DL93,Dong93,DLM}.
\end{example}

\begin{example}
    For $k\ge 2$, the parafermion algebra $K(\mathfrak{sl}_{2},k)$ is a pointed VOA. If $k\not=4$, $K(\mathfrak{sl}_{2},k)$ is neither a lattice VOA nor is it holomorphic.
By \cite{ALY}, for $k\ge 2$, $K(\mathfrak{sl}_{2},k)$ is simple, self-dual, strongly rational, and of CFT-type, and by \cite{AYY} are pointed, with fusion ring $\mathbb{Z}[\mathbb{Z}/k\mathbb{Z}]$. Moreover, since $K(\mathfrak{sl}_{2},k)$ has $k$ irreducible modules, it is never holomorphic. Moreover, since for a lattice VOA, the central charge is the rank of lattice. Since the central charge of $K(\mathfrak{sl}_{2},k)$ is $\frac{2(k-1)}{k+2}$ \cite[Section 3]{DLWY}.  Since this  is not an integer when $k\not=4$, these parafermions are not lattice VOAs either.
\end{example}

\begin{lemma}\label{prop_14}
    Let $V_{1}, V_{2}$ be pointed VOAs, with central charge $c_{1}$ and $c_{2}$, respectively. Then, 
    \begin{enumerate}
        \item $V_{1}\otimes V_{2}$ is a pointed VOA,
        \item $V_{1}\otimes V_{2}$ is holomorphic if and only if both $V_{1}$ and $V_{2}$ are holomorphic,
        \item if $c_{1}+c_{2}\not\in \mathbb{N}$, then $V_{1}\otimes V_{2}$ is not a lattice VOA.   
    \end{enumerate}
\end{lemma}
\noindent\textit{Proof.} For $V_{1},V_{2}$ strongly rational, we know that $V_{1}\otimes V_{2}$ is also strongly rational \cite{DMZ94,DLM2,Milas96}. Since the tensor product (or Kronecker product) of two permutation matrices is still a permutation matrix, the claim follows from
Lemma \ref{lem_1}. (b) follows from the classification theorem on the irreducible modules on tensor product of VOAs \cite{FHL, Milas96}. The central charge for $V_{1}\otimes V_{2}$ is $c_{1}+c_{2}$. However, the central charge of a lattice VOA must be a natural number, since the central charge of a lattice VOA is the rank of a lattice.
\begin{flushright}
    Q.E.D.
\end{flushright}

Lastly, we present some examples of strongly rational VOAs that are not pointed.
\begin{example}
    Let $V_{2,2k+1}$ be a Virasoro VOA in the boundary of discrete series (i.e., of central charge $c_{2,2k+1}$). Then, $V_{2,2k+1}$ is pointed if and only if $k = 1$.  Indeed, $V_{2,3}$ is holomorphic \cite[Theorem 4.2]{Wang93}. In all other cases, one can check that the fusion matrix for the irreducible module with the minimal conformal weight is an upper-left triangular matrix (See Proposition \ref{lem}). However, the fusion matrix of any irreducible module of a pointed matrix is a permutation matrix.
\end{example}

\section{Symmetric Coinvariant Divisors and Irreducible Module of Order two}\label{sec_Permutation_Module_of_Order_2} 
In the previous section, we consider the case where VOA is pointed. In this section, we no longer impose any additional assumptions on the VOA. Instead, we impose conditions on the modules, and study the symmetric coinvariant divisor on $\overline{\mathcal{M}}_{0,n}$.

\begin{definition}
    A coinvariant divisor on $\overline{\mathcal{M}}_{0,n}$ is called \textit{symmetric} if the irreducible modules associated to each marked point are the same.
\end{definition}
Symmetric coinvariant divisors are important and have been studied because they could be viewed as a divisor on $\tilde{M}_{0,n}/\mathrm{Sym}(n)$, where the group $\mathrm{Sym}(n)$ acts by permuting the $n$ marked points.

\begin{definition}
    Let $V$ be a strongly rational VOA, $W$ be an irreducible module. We say that $W$ is an \textit{irreducible module of order two} if $W\neq V$ and $W\boxtimes W = V$.
\end{definition}
We observe that this condition implies that $W=W'$, where $W'$ denotes the dual module of $W$.

\begin{lemma}
    For an irreducible module, $W$, of a strongly rational VOA, the following are equivalent:
    \begin{enumerate}
        \item $W$ is an irreducible module of order two.
        \item $R_{W,0}$ is a permutation matrix of order two.
    \end{enumerate}
\end{lemma}
\noindent\textit{Proof.} Let $W_{1},\cdots,W_{l}$ be the collection of all irreducible $V$-modules, up to isomorphism.

Part 1: $(a) \Rightarrow (b)$.

By Proposition \ref{prop_11}, we have
$
\begin{aligned}
    R_{W,0}R_{W,0} = R_{W\boxtimes W,0}
    = R_{V,0}
    =\mathrm{Id},
\end{aligned}
$
so $R_{W,0}$ is a matrix of order two. To show that it is a permutation matrix, it suffices to prove the following linear-algebraic claim.

Claim: let $A\in \mathrm{M}_{n\times n}(\mathbb{N})$ such that $A^{2}=\mathrm{Id}$, then $A$ is a permutation matrix of order $2$.

We consider the $(i,j)$-entry of $A^{2}$.
$
\begin{aligned}
    \delta_{i,j} = \mathrm{Id}_{(i,j)}=\left(A^{2}\right)_{(i,j)} =& \sum_{k=1}^{l} A_{(i,k)}A_{(k,j)}.
\end{aligned}
$
Take $i=j=1$, we have $1 = \sum_{k=1}^{l} A_{(1,k)}A_{(k,1)}$. Thus, there is only one term that is non-zero, and we call such $k$ by $k_{1}$. That is, $A_{(1,k_{1})} = A_{(k_{1},1)}=1$.

Moreover, $0 = \sum_{k=1}^{l} A_{(1,k)}A_{(k,j)}$, for $j\not=1$. Since $A_{(1,k_{1})}=1$, we conclude that $A_{(k_{1},j)}=0$. Since such $j$ is arbitrary, we conclude that $A_{(k_{1},j)}=\delta_{j,1}$. That is, we showed that for the $k_{1}$'s row of $A$, the $(k_{1},1)$-entry is $1$, and the rest entries are all zero. By the same argument, we can prove that $A_{(j,k_{1})}=\delta_{j,1}$. Thus, we show that for the $k_{1}$-th row or $k_{1}$-th column, there is only one $1$, and the rest are zero. Instead of considering $(1,1)$-entry, we can apply the same argument to to all $(i,i)$-entry, and conclude that $A$ is a permutation matrix. Thus, we proved the claim.

Part 2: $(b) \Rightarrow (a)$.

By Proposition \ref{prop_11}, we have $R_{V,0}=\mathrm{Id}=R_{W,0}R_{W,0} = R_{W\boxtimes W,0}.$
Therefore, $W\boxtimes W = V$.
\begin{flushright}
    Q.E.D.
\end{flushright}

\begin{lemma}\label{lemma_1}
    Let $V$ be a strongly rational VOA, $W$ be an irreducible module of order $2$. Then,
    $$
    \begin{aligned}
    \mathrm{rank}\mathbb{V}_{0,2l+2}\left(V,\left\{W^{2l}, S, T'\right\}\right)=&\delta_{S,T},\ 
    \mathrm{rank}\mathbb{V}_{0,2l+2}\left(V,\left\{W^{2l+1}, T\right\}\right)= \delta_{W,T'},\\
    \mathrm{rank}\mathbb{V}_{0,2l+2}\left(V,\left\{W^{2l+1}, T'\right\}\right)=& \delta_{W,T},\ 
    \mathrm{rank}\mathbb{V}_{0,n}\left(V,\left\{W^{n}\right\}\right)= \delta_{n,\text{even}}.
    \end{aligned}
    $$
\end{lemma}
\noindent\textit{Proof.} 
Let $W_{1},\cdots, W_{l}$ be the collection of irreducible modules, up to isomorphism. Let's write $S=W_{i},T=W_{j}$, we have
$$
\begin{aligned}
 \mathrm{rank}\mathbb{V}_{0,2l+2}\left(V,\left\{W^{2l}\right\}, S, T'\right)=& (R_{W^{2l},0})_{i,j}= \left(R_{W,0}^{2l}\right)_{i,j}= \left(\mathrm{Id}_{l}\right)_{i,j}
 = \delta_{i,j}
 = \delta_{S,T}.
 \end{aligned}
$$
Thus, we have proved the first equation. The rest equations follow from the first equation.
\begin{flushright}
    Q.E.D.
\end{flushright}

\begin{remark}
    By Lemma \ref{lemma_1}, if $n$ is odd, then $\mathrm{rank}\mathbb{V}_{0,n}\left(V,\left\{W^{n}\right\}\right)=0$. Therefore, without loss of generality, in the remaining section, we assume $n$ is even.
\end{remark}

\begin{lemma}\label{lemma_2}
    Let $V$ be a VOA of CohFT-type, $W$ be an irreducible module of order $2$, with conformal weight $a_{W}$, $b_{0,I}$ be the coefficient of $\delta_{0:I}$, where $I \subset \{1,2,...,n\}$. Then, for all $1\le |I| \le l-1$, for all even $n\ge 4$,
    $
    \begin{aligned}
b_{0,I}=a_{W} \delta_{|I|, \text{odd}}.
\end{aligned}
$
\end{lemma}
\noindent\textit{Proof.} Case 1: $2\le |I| \le n-2$.

Case 1.1: $|I|$ is odd.
$$
    \begin{aligned}
    b_{0,I}=&\sum_{S\in \mathcal{S}}a_{S}\mathrm{rank}\mathbb{V}_{0,|I|+1}\left(V,\left\{W^{|I|}, S\right\}\right)\mathrm{rank}\mathbb{V}_{0,n-|I|+1}\left(V,\left\{W^{I^{c}}, S'\right\}\right) \text{\cite[Corollary 2]{DGT3}}\\
    =& a_{W}\mathrm{rank}\mathbb{V}_{0,|I|+1}\left(V,\left\{W^{|I|},W'\right\}\right)\mathrm{rank}\mathbb{V}_{0,|I|+1}\left(V,\left\{W^{n-|I|},W'\right\}\right) \ \ \ (\text{Lemma }\ref{lemma_1})\\
    =& a_{W}\mathrm{rank}\mathbb{V}_{0,|I|+1}\left(V,\left\{W^{|I|+1}\right\}\right)\mathrm{rank}\mathbb{V}_{0,|I|+1}\left(V,\left\{W^{n-|I|+1}\right\}\right)= a_{W}\ \ \ (\text{Lemma }\ref{lemma_1}).
    \end{aligned} 
$$

Case 1.2: $|I|$ is even.
$$
    \begin{aligned}
    b_{0,I}=&\sum_{S\in \mathcal{S}}a_{S}\mathrm{rank}\mathbb{V}_{0,|I|+1}\left(V,\left\{W^{|I|}, S\right\}\right)\mathrm{rank}\mathbb{V}_{0,n-|I|+1}\left(V,\left\{W^{I^{c}}, S'\right\}\right) \text{\cite[Corollary 2]{DGT3}}\\
    =& a_{W}\mathrm{rank}\mathbb{V}_{0,|I|+1}\left(V,\left\{W^{|I|},W'\right\}\right)\mathrm{rank}\mathbb{V}_{0,|I|+1}\left(V,\left\{W^{n-|I|},W'\right\}\right) \ \ \ (\text{Lemma }\ref{lemma_1})\\
    =& a_{W}\mathrm{rank}\mathbb{V}_{0,|I|+1}\left(V,\left\{W^{|I|+1}\right\}\right)\mathrm{rank}\mathbb{V}_{0,|I|+1}\left(V,\left\{W^{n-|I|+1}\right\}\right)= 0\ \ \ (\text{Lemma }\ref{lemma_1}).
    \end{aligned} 
$$

Case 2: $|I|=1$ or $|I|=n-1$.

In this case, $\delta_{0,I}=-\psi_{i}$. Let $|I|=\{i\}$, for some $i\in \{1,\cdots,n\}$. By \text{\cite[Corollary 2]{DGT3}},
$$
\begin{aligned}
b_{0,\{i\}} =& a_{W}\mathrm{rank}\mathbb{V}_{0,n}\left(V,\left\{W^{n}\right\}\right)
= a_{W}\ \ \ (\text{Lemma }\ref{lemma_1}).
\end{aligned}
$$

\begin{flushright}
    Q.E.D.
\end{flushright}

\begin{example}
    Let $V$ be a VOA of CohFT-type, $W$ be an irreducible module of order two, with conformal weight $a_{W}$. Then, 
    $$
    \deg \mathbb{V}_{0,4}\left(V,\left\{W^{4}\right\}\right) = 4a_{W}.
    $$
    Thus, $\mathbb{V}_{0,4}\left(V,\left\{W^{4}\right\}\right)$ is ample if and only if $a_{W}>0$, and is nef if and only if $a_{W}\ge 0$.
\end{example}
\noindent\textit{Proof.} $$
\begin{aligned}
    \deg \mathbb{V}_{0,4}\left(V,\left\{W^{4}\right\}\right) =& \mathrm{rank}\mathbb{V}_{0,4}\left(V,\left\{W^{4}\right\}\right)(4a_{W})-\left(b_{0,\{1,2\}}+b_{0,\{1,3\}}+b_{0,\{1,4\}}\right) \text{\cite{DGT3}},\\
    =&4a_{W} (\text{Lemma \ref{lemma_2}})
\end{aligned}
$$
\begin{flushright}
    Q.E.D.
\end{flushright}

\begin{proposition}\label{prop_3}
    Let $V$ be a VOA of CohFT-type, $W$ be a irreducible module of order two, and $a_{W}$ be the conformal weight. Then, for $n$ even,
    $$
a_{W} \ge 0 \iff \mathbb{D}_{0,n}\left(V,\left\{W^{n}\right\}\right) \text{ is nef.}
$$
In particular, when $a_{W}<0$, $\mathbb{D}_{0,n}\left(V,\left\{W^{n}\right\}\right)$ is not F-nef.
\end{proposition}
\noindent\textit{Proof.}
Let $I,J,K,L$ be a non-empty partition of $[n]$. 

Case 1: all of $|I|,|J|,|K|,|L|$ are even.
$$
b_{0,I} =b_{0,J} =b_{0,K} =b_{0,L} =0,
b_{0,I\cup J}=b_{0,I\cup K}=b_{0,I\cup L}=0.
$$
Thus, in this case, we have
$
b_{0,I} + b_{0,J} +b_{0,K} + b_{0,L} = b_{0,I\cup J}+b_{0,I\cup K}+b_{0,I\cup L}.
$

Case 2: two of $|I|,|J|,|K|,|L|$ are odd, and two are even.

Case 2.1: $|I|, |J|$ are odd, $|K|,|L|$ are even.
$$
\begin{aligned}
b_{0,I} + b_{0,J} +b_{0,K} + b_{0,L} =& 2a_{W}, \  b_{0,I\cup J}+b_{0,I\cup K}+b_{0,I\cup L} = 2a_{W}. 
\end{aligned}
$$
Thus, in this case, we have
$
b_{0,I} + b_{0,J} +b_{0,K} + b_{0,L} = b_{0,I\cup J}+b_{0,I\cup K}+b_{0,I\cup L}.
$

Case 2.2: $|I|, |J|$ are even, $|K|,|L|$ are odd.
$$
\begin{aligned}
b_{0,I} + b_{0,J} +b_{0,K} + b_{0,L} =& 2a_{W}, \  b_{0,I\cup J}+b_{0,I\cup K}+b_{0,I\cup L} =& 2a_{W}. 
\end{aligned}
$$
Thus, in this case, we have
$
b_{0,I} + b_{0,J} +b_{0,K} + b_{0,L} = b_{0,I\cup J}+b_{0,I\cup K}+b_{0,I\cup L}.
$

Case 3: all of $|I|,|J|,|K|,|L|$ are odd.
$$
\begin{aligned}
b_{0,I} + b_{0,J} +b_{0,K} + b_{0,L} =& 4a_{W}, \  b_{0,I\cup J}+b_{0,I\cup K}+b_{0,I\cup L} = 0. 
\end{aligned}
$$

By Case 1 - 3, 
$
b_{0,I} + b_{0,J} +b_{0,K} + b_{0,L} \ge b_{0,I\cup J}+b_{0,I\cup K}+b_{0,I\cup L} \iff a_{W}\ge 0.
$
Therefore,
$$
\mathbb{D}_{0,n}\left(V,\left\{W^{n}\right\}\right) \text{ is F-nef} \iff a_{W}\ge 0.
$$

Claim: if $a_{W}\ge 0$, $\mathbb{D}_{0,n}\left(V,\left\{W^{n}\right\}\right)$ is nef.

Let $\mathcal{R}$ be the fusion ring of $V$, and let $\mathcal{S}$ be the subring generated by $W$. Since $W$ is an irreducible module of order two (i.e., $W\boxtimes W=V$), the simple modules contained in $\mathcal{S}$ are $W$ and $V$, both of which has non-negative conformal weight. Let $\psi=\sum_{i=1}^{n}\psi_{i}$, and for each $i\in\{2,\cdots,\lfloor\frac{n}{2}\rfloor\}$, define $B_{i}=\sum_{I\subset [n],|I|=i}\delta_{0,I}$. Then, we can write $$\mathbb{D}_{0,n}\left(V,\left\{W^{n}\right\}\right)=b_{1}\psi -\sum_{2\le i \le\lfloor\frac{n}{2}\rfloor }b_{i}B_{i},$$
where $b_{1} = a_{W}$ and $b_{i}=n \cdot a_{W}\delta_{i,\text{odd}}$. Therefore, for all $i,j\ge 2$, $b_{i}+b_{j}\ge b_{i+j}$, so $\mathbb{D}_{0,n}\left(V,\left\{W^{n}\right\}\right)$ is nef \cite[Theorem 3.6]{Cha}. Thus, we proved the claim.

\begin{flushright}
    Q.E.D.
\end{flushright}

Next, we consider examples of irreducible modules of order two.
\begin{example}\label{lemma_10}
    Let $V_{p,q}$ be a Virasoro VOA in the discrete series ($p\not=2$). Let $W_{\max}$ be the module of maximal conformal weight. Then, $W_{\max}$ is a irreducible module of order two.
\end{example}
\noindent\textit{Proof.} Let $V=V_{p,q}$, and $W_{1},\cdots, W_{l}$ be the collection of irreducible modules, up to isomorphism. 

Claim: $\forall$ irreducible module $W_{i}$, $\exists!$ a unique $W_{\tilde{i}}$, $\mathrm{rank}\mathbb{V}_{0,3}(V,\{W_{\max},W_{i},W_{\tilde{j}}')=\delta_{j,\tilde{i}}$.

Part I: existence.

By \cite{Wang93}, the conformal weight of irreducible modules is 
$
h_{p,q,m,n} = \frac{(np-mq)^{2}-(p-q)^{2}}{4pq},
$ so $W_{\max}=L_{c,h_{1,q-1}}$.
Let $W=L_{c,h_{m,n}}$ be an arbitrary minimal module of conformal weight $h_{m,n}$. Consider $\tilde{W}=L_{c,h_{m,q-n}}$. We have $\mathrm{rank}\mathbb{V}_{0}\left(V,\left\{W_{\max},W,\widetilde{W}\right\}\right)=1$ \cite{Wang93}.

Part II: uniqueness.

Suppose that $\mathrm{rank}\mathbb{V}_{0}(V,\{W_{\max},L_{c,h_{m,n}},L_{c,h_{m',n'}})=1$. Then, one of the condition for $m$ and $m'$ is
$
m<m'+1, m'<m+1
$
which is equivalent to $m=m'$. One of the condition for $n$ and $n'$ is
$
q-1<n+n', q-1+n+n'<2q,
$
which is equivalent to
$
q-1<n+n', n+n'<q+1,
$
and is equivalent to
$
n'=q-n.
$
We can see that the solution is unique. Thus, we have proved the Claim.

By the claim, $R_{W_{\max}}$ has exactly one $1$ in each row, and the rest are $0$'s, and $R_{W_{\max}}$ has exactly one $1$ in each column, and the rest are $0$'s. That is, $R_{W\max,0}$ is a permutation matrix, so $R_{\max,0}^{-1}=R_{\max,0}^{T}$, where $R_{\max,0}^{T}$ is the transpose. Notice that for Virasoro VOAs in the discrete series, all the irreducible modules are self-dual, so $R_{\max,0}$ is symmetric (i.e., $R_{\max,0}=R_{\max,0}^{T}$). Thus, $R_{\max,0}=R_{\max,0}^{-1}$ (i.e., $W_{\max}$ is a irreducible module of order two.
\begin{flushright}
    Q.E.D.
    \end{flushright}
\begin{remark}
    When $p=2$, $W_{\max} = V$, which is not a irreducible module of order two. However, the symmetric coinvariant divisor is still nef, by the propagation of vacua.
\end{remark}

\begin{example}
    Let $V_{L}$ be a lattice VOA. Suppose that there exists a $\lambda\in L'/L$ such that $\lambda+\lambda=0$ (e.g. $L'/L\simeq \mathbb{Z}/2m\mathbb{Z}$). Then, $W_{\lambda}$ is a irreducible module of order two.
\end{example}

\begin{example}\label{example_1}
    Let $V$ be an affine VOA associated with $sl_{2}$, with arbitrary level $l$. Let $W$ be the module of maximal conformal weight. Then, $W$ is a irreducible module of order two.
\end{example}


\section{Symmetric coinvariant divisors and Virasoro VOAs}

Virasoro VOAs form an important class of examples of vertex operator algebras. They correspond to the representations of the Virasoro Lie algebra. A Virasoro VOA is strongly rational if and only if its central charge is $c_{p,q}:=1-\frac{6(p-q)^{2}}{pq}$, for some relatively prime integers $p,q\ge 2$ \cite{FZ92,Wang93}, and we denote it by $V_{p,q}$. We call such VOAs the \textit{Virasoro VOAs in the discrete series}. If $p=2$ or $q = 2$, we say that the VOA \textit{lies on the boundary of the discrete series}. Since $V_{p,q}\cong V_{q,p}$, without loss of generality, we may assume $p<q$.

 In this section, we study how the vector bundles of coinvariants change as the Virasoro VOA moves through (the boundary of) the discrete series.

\begin{definition}
    Let $V_{p,q}$ be the Virasoro VOA of central charge $c_{p,q}$. Let $W_{\min}$ (resp. $W_{\max}$) be the irreducible module of the minimal (resp. maximal) conformal weight. Define 
    $$
f_{V_{p,q},W_{\min}}(z)=\sum_{n=0}^{\infty}\mathrm{rank} \mathbb{V}_{0,n+3}\left(V_{p,q},\left\{W_{\min}^{n+3}\right\}\right)z^{n}.
    $$
    We call it the \textit{generating function associated with $W_{\min}$}. We define $f_{V_{p,q},W_{\max}}(z)$ similarly.
\end{definition}

For $W_{\min}$, the generating function gradually changes as one moves along the boundary of discrete series. On the other hand, the pattern is more uniform for $W_{\max}$.

\begin{theorem}\label{theorem_7}
    Let $V_{2,2l+1}$ be a Virasoro VOA on the boundary of discrete series. Then,\\
when $l$ is even,
$$
f_{V_{2,2l+1},W_{\min}}(z)=\frac{1}{-z-\frac{1}{-z-\frac{1}{\frac{\ddots}{-z+1}}}},
$$
when $l$ is odd,
$$
f_{V_{2,2l+1},W_{\min}}(z)=\frac{1}{-z-\frac{1}{-z-\frac{1}{\frac{\ddots}{-z-1}}}},
$$
where each finite continued fraction on the right-hand side has $l$ layers.
\end{theorem}

\noindent\textit{Proof.} The proof is in Appendix \ref{app_theorem_7}.

\begin{proposition}
    Let $V_{p,q}$ be a Virasoro VOA in the discrete series. If $p=2$ or $q=2$ (i.e., $V_{p,q}$ lies on the boundary of the discrete series), then
    $$
f_{V_{p,q},W_{\max}}(z)=\frac{1}{1-z}.
    $$
    If $p,q\neq 2$ (i.e., $V_{p,q}$ lies in the interior of the discrete series), then
    $$
f_{V_{p,q},W_{\max}}(z)=\frac{z}{1-z^{2}}.
    $$
\end{proposition}
\noindent\textit{Proof.} When $p=2$ or $q=2$, $W_{\max}=V_{p,q}$ \cite{Wang93}. By propagation of vacua, $\mathrm{rank} \mathbb{V}_{0,n}\left(V_{p,q},\left\{W_{\max}^{n}\right\}\right) = 1$, for all $n\ge 3$. Thus, the generating function is
$$
f_{V_{p,q},W_{\max}}(z)=\sum_{n=0}^{\infty}z^{n}=\frac{1}{1-z}.
$$

When $p,q\neq 2$, by Example \ref{lemma_10}, $\mathrm{rank} \mathbb{V}_{0,n}\left(V_{p,q},\left\{W_{\max}^{n}\right\}\right) = \delta_{n,\mathrm{even}}$, for all $n\ge 3$. Thus, the generating function is
$$
f_{V_{p,q},W_{\max}}(z)=\sum_{n=0}^{\infty}\delta_{n+3,\mathrm{even}}z^{n}=\sum_{n=0}^{\infty}z^{2n+1}=\frac{z}{1-z^{2}}.
$$
\begin{flushright}
    Q.E.D.
\end{flushright}

\section{Some examples of rank formulae}
\subsection{Virasoro VOAs in the discrete series.}
\begin{example}
    Let $V=V_{p,q}$ be the Virasoro VOA of central charge $p,q$. Let $W_{\max}$ be the module of maximal conformal weight. Then, we have
    $$
    \mathrm{rank}\mathbb{V}_{0,n}\left(V_{p,q},\left\{V^{ i}, W_{\max}^{ 2k}, W, S\right\}\right)=\left\{
    \begin{aligned}
        1,& \text{ if }W=S,\\
        0,& \text{ otherwise},
    \end{aligned}
    \right.
    $$ for all $i,k\ge 0$, on $\overline{\mathcal{M}}_{0,n}$, where $n=i+2k+2\ge 3$. And,
    $$
    \begin{aligned}\mathrm{rank}\mathbb{V}_{1,n}\left(V_{p,q},\left\{V^{ i}, W_{\max}^{ 2k}\right
    \}\right)=&\frac{(p-1)(q-1)}{2}
    =\# \text{ irreducible modules, up to isomorphism},
    \end{aligned}
    $$
    for all $i,k\ge 0$ on $\overline{\mathcal{M}}_{1,n}$, where $n=i+2k\ge 1$.
\end{example}
\noindent\textit{Proof.}
By Propagation of vacua, 
$
\mathrm{rank}\mathbb{V}_{0,n}\left(V,\left\{V^{ i}, W_{\max}^{ 2k}, W, S\right\}\right)=\mathrm{rank}\mathbb{V}_{0,n}\left(V,\left\{ W_{\max}^{ 2k}, W, S\right\}\right).
$ By Example \ref{lemma_10}, $R_{\max,0}^{2}=\mathrm{Id}$. The claim follows from Proposition \ref{cor_2}. Since all irreducible modules of a Virasoro VOA are self-dual, we can replace the condition $W=S'$ by $W=S$.

\begin{flushright}
    Q.E.D.
\end{flushright}

\begin{lemma}\label{lem_2}\cite{Wang93}
    Let $V=V_{2,2t+1}$ be the Virasoro VOA of central charge $c_{2,2t+1}$, for $t\ge 1$. $V$ has $t$-many irreducible modules \cite[Theorem 4.2]{Wang93}, labeled as $W_{1},...,W_{t}$ by increasing order of conformal weight.
    Then, the FA-matrices is the following:
    \begin{enumerate}
        \item if $t$ is odd, then
        \begin{equation}\label{equation_19}
    (R_{W_{k},0})_{(i,j)}=\left\{
    \begin{aligned}
        1,& \text{ if }i+j\le t-k+2, \\
        1,&\text{ if }i+j \equiv k+1 \mod 2, \text{ and } |i-j|\le t-k+1, \text{ and }i+j\le t+k+1,\\
        0,& \text{ otherwise}.
    \end{aligned}
    \right.
    \end{equation}

    \item if $t$ is even, then
        \begin{equation}\label{equation_20}
    (R_{W_{k},0})_{(i,j)}=\left\{
    \begin{aligned}
        1,& \text{ if }i+j\le t-k+2, \\
        1,&\text{ if }i+j \equiv k \mod 2, \text{ and } |i-j|\le t-k+1, \text{ and }i+j\le t+k+1,\\
        0,& \text{ otherwise}.
    \end{aligned}
    \right.
    \end{equation}
    \end{enumerate}
    
\end{lemma}

Equation \eqref{equation_19} and \eqref{equation_20} may appear to be complicated, but the pattern is very clear after computing several examples. It allows us to directly write down the FA-matrix on $\overline{\mathcal{M}}_{0,3}$, without computing the fusion rule for each entry.

\begin{example}
    Let $V_{2,q}$ be the Virasoro VOA of central charge $q=2t+1$. For each $t$, we label the irreducible modules are $W_{1},\cdots, W_{t}$. Then, $R_{W_{1},0},...,R_{W_{t},0}$ (from the left to right) are:
    \begin{enumerate}
        \item for $t=1$, we have
        $$
\left(\begin{matrix}
\boldsymbol{1}
\end{matrix}\right)
        $$
        \item for $t=2$, we have
        $$
\left(\begin{matrix}
1 & \boldsymbol{1} \\
\boldsymbol{1} & 0
\end{matrix}\right), \left(\begin{matrix}
\boldsymbol{1} & 0 \\
0 & 1
\end{matrix}\right)
        $$
        \item for $t=3$, we have
        $$
    \left(\begin{matrix}
1 & 1 & \boldsymbol{1} \\
1 & \boldsymbol{1} & 0 \\
\boldsymbol{1} & 0 & 0
\end{matrix}\right), 
\left(\begin{matrix}
1 & \boldsymbol{1} & 0 \\
\boldsymbol{1} & 0 & 1 \\
0 & 1 & 0
\end{matrix}\right), \left(\begin{matrix}
\boldsymbol{1} & 0 & 0 \\
0 & 1 & 0 \\
0 & 0 & 1
\end{matrix}\right)
        $$
    
    \item for $t=4$, we have
    $$
\left(\begin{matrix}
1 & 1 & 1 & \boldsymbol{1} \\
1 & 1 & \boldsymbol{1} & 0 \\
1 & \boldsymbol{1} & 0 & 0 \\
\boldsymbol{1} & 0 & 0 & 0
\end{matrix}\right), \left(\begin{matrix}
1 & 1 & \boldsymbol{1} & 0 \\
1 & \boldsymbol{1} & 0 & 1 \\
\boldsymbol{1} & 0 & 1 & 0 \\
0 & 1 & 0 & 0
\end{matrix}\right),\left(\begin{matrix}
1 & \boldsymbol{1} & 0 & 0 \\
\boldsymbol{1} & 0 & 1 & 0 \\
0 & 1 & 0 & 1 \\
0 & 0 & 1 & 0
\end{matrix}\right),\left(\begin{matrix}
\boldsymbol{1} & 0 & 0 & 0 \\
0 & 1 & 0 & 0 \\
0 & 0 & 1 & 0 \\
0 & 0 & 0 & 1
\end{matrix}\right).
    $$
    \end{enumerate}
\end{example}

Applying Theorem \ref{cor_2}, we obtain an explicit closed rank formula for Virasoro VOA of central charge $c_{2,2k+1}$ on $\overline{\mathcal{M}}_{g,n}$, for all $q\ge 3$, all $g,n$. Moreover, one can analyze the pattern of the multiplication of these matrices to further simplify the rank formula.

\begin{example}
    Let $V=V_{p,q}$ be the Virasoro VOA of central charge $c_{p,q}$. Let $S^{\bullet}$ be an $n$-tuple of irreducible modules such that $R_{S^{\bullet}}=0$. Then,
    $$
    \mathrm{rank}\mathbb{V}_{1,n}\left(V,\{S^{\bullet}\}\right) = \frac{(p-1)(q-1)}{2}.
    $$
    In particular,
    $$
    \mathrm{rank}\mathbb{V}_{1,1}\left(V,\{V\}\right) = \frac{(p-1)(q-1)}{2}.
    $$
\end{example}
\noindent\textit{Proof.} By Example \ref{ex_1}, $\mathrm{rank}\mathbb{V}_{1,n}\left(V,\{S^{\bullet}\}\right)= $ the number of irreducible modules, up to isomorphism. By \cite{Wang93}, that number is equal to $ \frac{(p-1)(q-1)}{2}$.
\begin{flushright}
    Q.E.D.
\end{flushright}

\begin{example}
    Let $V$ be the Virasoro VOA of central charge $c_{2,5}$ (i.e., Yang-Lee Model). By \cite{Wang93}, $V_{2,5}$ has only one non-trivial irreducible module $W$. Let $F(m)$ be the Fibonacci sequence, where $F(1) = 0, F(2) = 1$. Then,
    $$
    \mathrm{rank}\mathbb{V}_{0,n}\left(V,\left\{W^{n}\right\}\right) = F(n-1).
    $$
\end{example}
\noindent\textit{Proof.} By \cite{Wang93}, the fusion matrices are
$
R_{W,0}=\left(\begin{matrix}
1 & 1 \\
1 & 0
\end{matrix}\right),
$
and $R_{V,0}=\mathrm{Id}_{2}$. By Theorem \ref{cor_2}, 
$$ 
\mathrm{rank}\mathbb{V}_{0,n}\left(V,\left\{V^{i},W^{j}\right\}\right) = \left(\left(\begin{matrix}
1 & 1 \\
1 & 0
\end{matrix}\right)^{n}\right)_{(2,2)}
$$

The rest follows from analyzing matrix power of $R_{W,0}$.
\begin{flushright}
    Q.E.D.
\end{flushright}

\subsection{Affine VOA associated with $sl_{2}$}

Let  $L_{sl_{2}}(l,0)$ denote the affine VOA associated with $sl_{2}$ with level $l\ge 1$, and $W_{p}:=L_{sl_{2}}(l,p\frac{\alpha}{2})$, where $p\in \{0,1,...,l\}$. These $W_{p}$'s are exactly all the irreducible $L_{sl_{2}}(l,0)$-modules, up to isomorphism, and $W_{p}'=W_{p}$, for all $p$ \cite{FZ92}.

\begin{example}
    Let $V=L_{sl_{2}}(l,0)$ be affine VOA associated with $sl_{2}$ with level $l\ge 1$. Let $S^{\bullet}$ be an $n$-tuple of irreducible modules such that $R_{S^{\bullet}}=0$. Then,
    $$
    \mathrm{rank}\mathbb{V}_{1,n}\left(V,\{S^{\bullet}\}\right) = l+1.
    $$
    In particular,
    $$
    \mathrm{rank}\mathbb{V}_{1,1}\left(V,\{V\}\right) = l+1.
    $$
\end{example}
\noindent\textit{Proof.} By Example \ref{ex_1}, $\mathrm{rank}\mathbb{V}_{1,n}\left(V,\{S^{\bullet}\}\right)= $ the number of irreducible modules, up to isomorphism. By \cite{FZ92}, that number is equal to $ l+1$.
\begin{flushright}
    Q.E.D.
\end{flushright}

\begin{lemma}\label{lemma_2a}
    Let $V=L_{sl_{2}}(l,0)$ be the affine VOA associated with $sl_{2}$, with level $l\ge 1$. Let $W_{p}:=L_{sl_{2}}(l,p\frac{\alpha}{2})$, for all $p\in \{0,...,l\}$ be the collection of all isomorphism class of irreducible modules. Then,
\begin{equation}\label{equation_12}
(R_{W_{p},0})_{i+1,j+1}=
\left\{
\begin{aligned}
    1,& \text{ if }i+j \equiv p \mod 2 \text{, and }|i-j|\le p, \text{ and } p\le i+j\le 2l-p\\
    0,& \text{ otherwise.}
\end{aligned}
\right.
\end{equation}
\end{lemma}
\noindent\textit{Proof.} 
Let $L_{\widehat{sl}_{2}}(l,0)$ be the affine VOA associated with $sl_{2}$ with level $l$, where we fix an arbitrary level $l\ge 1$. Let $W_{p}:=L_{\widehat{sl}_{2}}(l,p\frac{\alpha}{2}),W_{i}:=L_{\widehat{sl}_{2}}(l,i\frac{\alpha}{2}),W_{j}:=L_{\widehat{sl}_{2}}(l,j\frac{\alpha}{2})$ be three arbitrary irreducible modules, where $0\le p,i,j \le l$. Let $N_{p,i}^{j}:=N_{W_{p},W_{i}}^{W{j}}$ be the fusion rule. Then, by \cite[Lemma 4.2]{Bea96}, we have
\begin{equation}\label{equation_13}
N_{p,i}^{j}=\left\{
\begin{aligned}
    1&, \text{ if } p+i+j = 2m, \text{ for some }m \le l, \text{ and }p,i,j \le m\\
    0&, \text{ otherwise}.
\end{aligned}
\right.
\end{equation}
By the construction of vector bundle of coinvariants, $(R_{W_{p},0})_{i,j} = N_{p,i}^{j}$, since on $\overline{\mathcal{M}}_{0,3}$, the rank of vector bundle of coinvariants equals the fusion rule. Thus, all we need to show is that the right hand side of Equation \eqref{equation_12} equals the right hand side of Equation \eqref{equation_13}. Define
$$
f(i,j)=\left\{
\begin{aligned}
    1,& \text{ if }i+j \equiv p \mod 2 \text{, and }|i-j|\le p, \text{ and } p\le i+j\le 2l-p\\
    0,& \text{ otherwise.}
\end{aligned}
\right.
$$
$$
g(i,j)=\left\{
\begin{aligned}
    1&, \text{ if } p+i+j = 2m, \text{ for some }m \le l, \text{ and }p,i,j \le m\\
    0&, \text{ otherwise}.
\end{aligned}
\right.
$$
Claim: $f(i,j)=g(i,j)$, for all $0\le i,j \le l$.

Let $i,j\in \{0,1,...,l\}$ be arbitrary.\\
Case 1: $i+j \equiv p \mod 2 \text{, and }|i-j|\le p, \text{ and } p\le i+j\le 2l-p$.\\
In this case, $f(i,j)=1$. $i+j \equiv p \mod 2$, so $i+j=p+2t$, for some $t\in \mathbb{Z}$. Then, $i+j+p = 2p+2t=2(p+t)$. Let $m=p+t$. Notice that
$
2l \ge i+j+p = (i+j-p)+2p = 2t+2p = 2(t+p) = 2m.
$
Thus, we have $m\le l$. $p\le i+j$, so $t\ge 0$. Thus,
$
p\le p+t = m.
$
Lastly, $|i-j|<p$ implies $i\le p+j$. So
$
2i \le p+i+j = 2p+2t = 2(p+t) = 2m.
$
Thus, we have proved $i\le m$. By symmetry, $j\le m$. Thus, we checked that this $m = p+t$ satisfies all conditions of $g$, so $g(i,j) = 1$.

Case 2: $p+i+j = 2m, \text{ for some }m \le l, \text{ and }p,i,j \le m$.\\
In this case, $g(i,j)=1$. $p+i+j = 2m$ implies $i+j\equiv p \mod 2$. By $p\le m = \frac{p+i+j}{2}$, we have $p\le i+j.$ By $i\le m = \frac{p+i+j}{2}$, we have $i\le p+j$, so $i-j\le p$. Similarly, $j-i\le p$. Thus, we have
$
|i-j|\le p.
$
Notice that $\frac{p+i+j}{2}=m\le l$, we have
$
i+j\le 2l-p.
$
Therefore, such $(i,j)$ satisfies all conditions in $f$, so $f(i,j)=1$.

By Case 1, $f(i,j)=1$ implies $g(i,j) = 1$. By Case 2, $g(i,j)=1$ implies $f(i,j)=1$. Thus, $f(i,j)=1$ if and only if $g(i,j) = 1$, so $f=g$.
\begin{flushright}
    Q.E.D.
\end{flushright}

Lemma \ref{lemma_2a} differs slightly from \cite[Lemma 4.2]{Bea96} since it gives a pattern for fusion matrices. 

\begin{example}
    Let $V=L_{sl_{2}}(l,0)$ be the affine VOA associated with $sl_{2}$ of level $l$, then, the
    \begin{enumerate}
    \item for level $l=1$, the fusion matrix of $W_{0},W_{1}$ (from left to right) are
$$
\left(\begin{matrix}
\boldsymbol{1} & 0 \\
0 & 1
\end{matrix}\right),\left(\begin{matrix}
0 & \boldsymbol{1} \\
\boldsymbol{1} & 0
\end{matrix}\right)
$$
\item for level $l=2$, the fusion matrix of $W_{0},\cdots,W_{l}$ (from left to right) are
$$
\left(\begin{matrix}
\boldsymbol{1} & 0 & 0 \\
0 & 1 & 0 \\
0 & 0 & 1
\end{matrix}\right),
\left(\begin{matrix}
0 & \boldsymbol{1} & 0 \\
\boldsymbol{1} & 0 & 1 \\
0 & 1 & 0
\end{matrix}\right),
\left(\begin{matrix}
0 & 0 & \boldsymbol{1} \\
0 & \boldsymbol{1} & 0 \\
\boldsymbol{1} & 0 & 0
\end{matrix}\right).
$$
\item for level $l=3$, the fusion matrix of $W_{0},\cdots,W_{l}$ (from left to right) are
$$
\left(\begin{matrix}
\boldsymbol{1} & 0 & 0 & 0 \\
0 & 1 & 0 & 0 \\
0 & 0 & 1 & 0 \\
0 & 0 & 0 & 1
\end{matrix}\right),
\left(\begin{matrix}
0 & \boldsymbol{1} & 0 & 0 \\
\boldsymbol{1} & 0 & 1 & 0 \\
0 & 1 & 0 & 1 \\
0 & 0 & 1 & 0
\end{matrix}\right),
\left(\begin{matrix}
0 & 0 & \boldsymbol{1} & 0 \\
0 & \boldsymbol{1} & 0 & 1 \\
\boldsymbol{1} & 0 & 1 & 0 \\
0 & 1 & 0 & 0
\end{matrix}\right),
\left(\begin{matrix}
0 & 0 & 0 & \boldsymbol{1} \\
0 & 0 & \boldsymbol{1} & 0 \\
0 & \boldsymbol{1} & 0 & 0 \\
\boldsymbol{1} & 0 & 0 & 0
\end{matrix}\right).
$$
    \end{enumerate}
\end{example}

Applying Theorem \ref{cor_2}, we obtain an explicit closed rank formula for the vector bundle of coinvariants associated with the affine VOA $L_{sl_{2}}(l,0)$ on $\overline{\mathcal{M}}_{g,n}$. Moreover, one can analyze the multiplication of these matrices to further simplify the rank formula.

\subsection{Examples of line bundles.}\label{Examples_of_line_bundles}

When the vector bundle of coinvariants on $\overline{\mathcal{M}}_{g,n}$ is a line bundle, it often enjoys many special properties. For instance, Theorem \ref{main_thm_3} shows that the answer to \cite[Question 1]{DG} is positive for line bundles. 

\begin{example}\label{example_4}
    Let $V$ be a holomorphic VOA. Then, the vector bundle of coinvariant is a line bundle on $\overline{\mathcal{M}}_{g,n}$, for all $g,n\ge 0$ \footnote{Example \ref{example_4} was first proved in \cite[Example 5.1.1]{DGT3}. Here, we give an linear-algebraic proof.}.
\end{example}
\noindent\textit{Proof.} $V$ is holomorphic, so $R_{V,0}$ and $R_{V,1}$ are both identity matrix of size one. Therefore, by Theorem \ref{cor_2}, $\mathrm{rank}\mathbb{V}_{g,n}\left(V,\{V^{n}\}\right)= \left(R_{V,0}^{n} R_{V,1}^{g}\right)_{(1,1)} = 1$, for all $g,n\ge 0$.
\begin{flushright}
    Q.E.D.
\end{flushright}

\begin{example} 
    Let $V$ be a strongly rational VOA. Then, $\mathbb{V}_{0,n}(V,\{V^{n}\})$ is a line bundle, for all $n$.
\end{example}
\noindent\textit{Proof.} This directly follows from propagation of vacua \cite[Theorem 4.3]{DGT22a}.
\begin{flushright}
    Q.E.D.
\end{flushright}

\begin{example}
   For pointed VOA, the vector bundle of coinvariants on $\overline{\mathcal{M}}_{0,n}$ is zero or a line bundle.
\end{example}
\noindent\textit{Proof.} This directly follows from Theorem \ref{main_prop_5} (or Proposition \ref{prop_6}), by taking $g=0$.
\begin{flushright}
    Q.E.D.
\end{flushright}

\begin{example}
    Let $V$ be a strongly rational VOA, $S$ an irreducible module of order two. The vector bundle of coinvariants $\mathbb{V}_{0,n}(V,\{S^{n}\})$ is a line bundle when $n$ is even, and zero when $n$ is odd.
\end{example}
\noindent\textit{Proof.} This directly follows from Theorem \ref{main_prop_3} (or Lemma \ref{lemma_1}), by taking $g=0$.
\begin{flushright}
    Q.E.D.
\end{flushright}

\begin{example}
    Let $V=V_{p,q}$ be a Virasoro VOA in the discrete series of central charge $c_{p,q}$. Let $W_{\max}$ be the irreducible module of the maximal conformal weight. Then, for all $n$
    \begin{enumerate}
        \item if $p=2$, $\mathrm{rank}\mathbb{V}_{0,n}(V_{2,q},\{W_{\max}^{n}\}) = 1$,
        \item if $p\not=2$, $\mathrm{rank}\mathbb{V}_{0,n}(V_{p,q},\{W_{\max}^{n}\}) = \delta_{n,\mathrm{even}}$.
    \end{enumerate}
\end{example}
\noindent\textit{Proof.} The case for $p\not=2$ is proved in Example \ref{lemma_10}. When $p=2$, $W_{\max} = V$ \cite{Wang93}. Therefore, $\mathrm{rank}\mathbb{V}_{0,n}(V_{2,q},\{W_{\max}^{n}\}) = 1$, by propagation of vacua \cite[Theorem 4.3]{DGT22a}.
\begin{flushright}
    Q.E.D.
\end{flushright}

\begin{example}
    Let $V=L_{sl_{2}}(l,0)$ be the affine VOA associated with $sl_{2}$ at level $l\ge 1$. Let $W_{\max}$ be the irreducible module of the maximal conformal weight. Then, for all $n\ge 1$,
    $$
    \mathrm{rank}\mathbb{V}_{0,n}(V_{p,q},\{W_{\max}^{n}\}) = \delta_{n,\mathrm{even}}.
    $$
\end{example}
\noindent\textit{Proof.} $W_{\max}$ is an irreducible module of order two \cite{Bea96}, so the claim follows from Theorem \ref{cor_2}.
\begin{flushright}
    Q.E.D.
\end{flushright}

\appendix
\section{Proof of Theorem \ref{thm_1}}\label{app_thm_1}
\noindent\textit{Proof.} Part I: $(1)\Rightarrow (2)$.

Let $\beta_{1},\beta_{2}\in \mathbb{N}^{(I)}, g_{1},g_{2}\ge 0$ be arbitrary. Consider the $(i,j)$-entry of $\mathcal{N}_{\beta_{1},g_{1}}\mathcal{N}_{\beta_{2},g_{2}}$.
$$
\begin{aligned}
\left(\mathcal{N}_{\beta_{1},g_{1}}\mathcal{N}_{\beta_{2},g_{2}}\right)_{(i,j)}=
&\sum_{k=1}^{l} N_{g_{1}}(\beta_{1}+\lambda_{i}+\lambda_{k}')N_{g_{2}}(\beta_{2}+\lambda_{k}+\lambda_{j}') \\
=& \sum_{k=1}^{l} N_{g_{1}}((\beta_{1}+\lambda_{i})+\lambda_{k}')N_{g_{2}}((\beta_{2}+\lambda_{j}')+\lambda_{k})\\
=&N_{g_{1}+g_{2}}(\beta_{1}+\lambda_{i}+\beta_{2}+\lambda_{j}')=N_{g_{1}+g_{2}}((\beta_{1}+\beta_{2})+\lambda_{i}+\lambda_{j}')\\
=&(\mathcal{N}_{g_{1}+g_{2},\beta_{1}+\beta_{2}})_{(i,j)}.
\end{aligned}
$$
Let $\beta\in \mathbb{N}^{(I)}, g\ge 0$ be arbitrary. By assumption, we have the equality
$$
N_{g+1}(\beta) = \sum_{\lambda\in I}N_{g}\left(\beta+\lambda+\lambda'\right).
$$
Thus, we have
$$
\begin{aligned}
N_{g+1}(\beta) =& \sum_{\lambda\in I}N_{g}(\beta+\lambda+\lambda')= \sum_{i=1}^{l} \left(\mathcal{N}_{\beta,g}\right)_{(i,i)}= \mathrm{Tr}\left(\mathcal{N}_{\beta,g}\right).
\end{aligned}
$$

Part II: $(2) \Rightarrow (1)$.

Let $\beta\in \mathbb{N}^{(I)}, g\ge 0$ be arbitrary.
$$
\begin{aligned}
    N_{g+1}(\beta)=&\mathrm{Tr}(\mathcal{N}_{\beta,g})=\sum_{i=1}^{l}(\mathcal{N}_{\beta,g})_{(i,i)}=\sum_{i=1}^{l}N_{g}\left(\beta+\lambda_{i}+\lambda_{i}'\right).
\end{aligned}
$$
Thus, we have proved the Factorization Property II in Definition \ref{def_factorization_property}. Let $\alpha,\beta\in \mathbb{N}^{(I)}$ be arbitrary. By assumption, we have
$
\begin{aligned}
    \mathcal{N}_{\beta+\alpha,g_{1}+g_{2}}=\mathcal{N}_{\beta,g_{1}}\mathcal{N}_{\alpha,g_{2}}.
\end{aligned}
$
Taking $(i,j)$-entry on both side, 
$$
\begin{aligned}
    N_{g_{1}+g_{2}}(\beta+\alpha+\lambda_{i}+\lambda_{j}')=&\left(\mathcal{N}_{\beta,g_{1}}\mathcal{N}_{\alpha,g_{2}}\right)_{(i,j)}=\sum_{k=0}^{l}\left(\mathcal{N}_{\beta,g_{1}}\right)_{(i,k)}\left(\mathcal{N}_{\alpha,g_{2}}\right)_{(k,j)}\\
    =& \sum_{k=0}^{l}N_{g_{1}}(\beta+\lambda_{i}+\lambda_{k}')N_{g_{2}}(\alpha+\lambda_{k}+\lambda_{j}').
\end{aligned}
$$
This equality holds for all $i,j\in \{1,\cdots,l\}$. Since we assume $0\in I$, we take $\lambda_{i}=0,\lambda_{j}=0'$. 
$$
N_{g_{1}+g_{2}}(\beta+\alpha)=\sum_{k=0}^{l}N_{g_{1}}(\beta+\lambda_{k}')N_{g_{2}}(\alpha+\lambda_{k}).
$$
Thus, we have proved Factorization Property 1 in Definition \ref{def_factorization_property}.

Part III: $(2)\Rightarrow (3)$.

Let $\beta\in\mathbb{N}^{(I)}$, and for all $g\ge 0$ be arbitrary. Let's write $\beta=\sum_{\lambda\in I}c_{\lambda}\lambda$. Then,
$$
\begin{aligned}
    \left(\prod_{\lambda\in I} \mathcal{N}_{\lambda,0}^{n_{\lambda}}\right)\mathcal{N}_{0,1}^{g}=& \mathcal{N}_{\sum_{\lambda\in I}c_{\lambda}\lambda,\underbrace{1+\cdots+1}_{g \text{ times}}}\ \ \ \ \text{(FA 1)}=\mathcal{N}_{\beta,g}.
\end{aligned}
$$
Taking the $(i,j)$-entry on both side, we get
$$
\begin{aligned}
\left(\left(\prod_{\lambda\in I} \mathcal{N}_{\lambda,0}^{n_{\lambda}}\right)\mathcal{N}_{0,1}^{g}\right)_{(i,j)} =& \left(\mathcal{N}_{\beta,g}\right)_{(i,j)} = N_{g}(\beta+\lambda_{i}+\lambda_{j}').
\end{aligned}
$$
Thus, we have proved (V1) in Definition \ref{def_Verlinde_Property}. Now, assume $g\ge 1$.
$$
\begin{aligned}
\mathrm{Tr}\left(\left(\prod_{\lambda\in I} \mathcal{N}_{\lambda,0}^{c_{\lambda}}\right)\mathcal{N}_{0,1}^{g-1}\right)\overset{\text{FA 2}}{=}& \mathrm{Tr}\left(\mathcal{N}_{\sum_{\lambda\in I}c_{\lambda}\lambda,\underbrace{1+\cdots+1}_{g-1 \text{ times}}}\right)\\
    =&\mathrm{Tr}\left(\mathcal{N}_{\beta,g-1}\right)\overset{\text{FA 2}}{=} N_{g-1+1}(\beta)= N_{g}(\beta).
\end{aligned}
$$
Thus, we have proved (V2) in Definition \ref{def_Verlinde_Property}.

Claim: $\mathcal{N}_{0,1}=\sum_{\lambda\in I}\mathrm{Tr}\left(\mathcal{N}_{\lambda',0}\right)\mathcal{N}_{\lambda}$.

Considering the $(i,j)$-entry of the right-hand side, we get
$$
\begin{aligned}
    \left(\sum_{\lambda\in I}\mathrm{Tr}(\mathcal{N}_{\lambda',0})\mathcal{N}_{\lambda,0}\right)_{(i,j)}=& \sum_{\lambda\in I}\mathrm{Tr}(\mathcal{N}_{\lambda',0})(\mathcal{N}_{\lambda,0})_{(i,j)}= \sum_{\lambda\in I} \mathrm{Tr}\left(\mathcal{N}_{\lambda',0}\right) N_{0}\left(\lambda+\lambda_{i}+\lambda_{j}'\right)\\
    =& \sum_{\lambda\in I} N_{1}\left(\lambda'\right)N_{0}\left(\lambda+\lambda_{i}+\lambda_{j}'\right) = N_{1}\left(\lambda_{i}+\lambda_{j}'\right)= \left(\mathcal{N}_{0,1}\right)_{(i,j)}.
\end{aligned}
$$
Thus, we have proved (V3) in Definition \ref{def_Verlinde_Property}.

Part IV: $(3)\Rightarrow (2)$.

Claim 1: $\{\mathcal{N}_{\lambda,0}\}_{\lambda\in I}$ all commute.

Taking the $(i,j)$-entry of $\mathcal{N}_{\lambda_{1},0}\mathcal{N}_{\lambda_{2},0}$, we get
$$
\begin{aligned}
\left(\mathcal{N}_{\lambda_{1},0}\mathcal{N}_{\lambda_{2},0}\right)_{(i,j)}\overset{V1}{=}&N_{0}\left(\lambda_{1}+\lambda_{2}+\lambda_{i}+\lambda_{j}'\right)= N_{0}\left(\lambda_{2}+\lambda_{1}+\lambda_{i}+\lambda_{j}'\right)\\
\overset{V1}{=}& N_{0}\left(\lambda_{2}+\lambda_{1}+\lambda_{i}+\lambda_{j}'\right)=\left(\mathcal{N}_{\lambda_{2},0}\mathcal{N}_{\lambda_{1},0}\right)_{(i,j)}.
\end{aligned}
$$

Claim 2: $\{\mathcal{N}_{\lambda,0}\}_{\lambda\in I}\cup \{\mathcal{N}_{0,1}\}$ all commute.

Let $\gamma\in I$ be arbitrary.
By V3, $\mathcal{N}_{0,1}=\sum_{\lambda\in I}\mathrm{Tr}\left(\mathcal{N}_{\lambda',0}\right)\mathcal{N}_{\lambda,0}$, so 
$$
\begin{aligned}
\mathcal{N}_{\gamma,0}\mathcal{N}_{0,1}\overset{V3}{=}&\mathcal{N}_{\gamma,0}\left(\sum_{\lambda\in I}\mathrm{Tr}\left(\mathcal{N}_{\lambda',0}\right)\mathcal{N}_{\lambda,0}\right)
=\sum_{\lambda\in I}\mathrm{Tr}\left(\mathcal{N}_{\lambda',0}\right)\mathcal{N}_{\gamma,0}\mathcal{N}_{\lambda,0}\\
\overset{\text{Claim 1}}{=}& \sum_{\lambda\in I}\mathrm{Tr}\left(\mathcal{N}_{\lambda',0}\right)\mathcal{N}_{\lambda,0}\mathcal{N}_{\gamma,0}= \left(\sum_{\lambda\in I}\mathrm{Tr}\left(\mathcal{N}_{\lambda',0}\right)\mathcal{N}_{\lambda,0}\right)\mathcal{N}_{\gamma,0}\overset{V3}{=}\mathcal{N}_{0,1}\mathcal{N}_{\lambda,0}.
\end{aligned}
$$

Claim 3: (FA 1) holds.

Let $\alpha,\beta\in \mathbb{N}^{(I)}$, and $g_{1},g_{2}\ge 0$ be arbitrary. Let's write $\alpha=\sum_{\lambda\in I}\alpha_{\lambda}\lambda, \beta=\sum_{\lambda\in I}\beta_{\lambda}\lambda$. Then,
$$
(\mathcal{N}_{\alpha,g_{1}})_{(i,j)}=N_{g_{1}}(\alpha+\lambda_{i}+\lambda_{j})
\overset{V1}{=}\left(\left(\prod_{\lambda\in I} \mathcal{N}_{\lambda,0}^{\alpha_{\lambda}}\right)\mathcal{N}_{0,1}^{g_{1}}\right)_{(i,j)},
$$
for all $(i,j)$-entries. Thus, we have
$$
\mathcal{N}_{\alpha,g_{1}}=\left(\prod_{\lambda\in I} \mathcal{N}_{\lambda,0}^{\alpha_{\lambda}}\right)\mathcal{N}_{0,1}^{g_{1}},\text{ and }\mathcal{N}_{\beta,g_{2}}=\left(\prod_{\lambda\in I} \mathcal{N}_{\lambda,0}^{\beta_{\lambda}}\right)\mathcal{N}_{0,1}^{g_{2}}.
$$
Therefore,
\begin{equation}\label{equ_4}
\begin{aligned}
\mathcal{N}_{\alpha,g_{1}}\mathcal{N}_{\beta,g_{2}}=&\left(\left(\prod_{\lambda\in I} \mathcal{N}_{\lambda,0}^{\alpha_{\lambda}}\right)\mathcal{N}_{0,1}^{g_{1}}\right)\left(\left(\prod_{\lambda\in I} \mathcal{N}_{\lambda,0}^{\beta_{\lambda}}\right)\mathcal{N}_{0,1}^{g_{2}}\right) 
\end{aligned}
\end{equation}

On the other hand, we have $
\alpha+\beta=\sum_{\lambda\in I}\alpha_{\lambda}\lambda+\sum_{\lambda\in I}\beta_{\lambda}\lambda.
$
Therefore,
$$
\begin{aligned}
    \left(\mathcal{N}_{\alpha+\beta,g_{1}+g_{2}}\right)_{(i,j)}=&N_{g_{1}+g_{2}}(\alpha+\beta+\lambda_{i}+\lambda_{j})=\left(\left(\prod_{\lambda\in I} \mathcal{N}_{\lambda,0}^{\alpha_{\lambda}+\beta_{\lambda}}\right)\mathcal{N}_{0,1}^{g_{1}+g_{2}}\right)_{(i,j)}.
\end{aligned}
$$
Therefore, we have

\begin{equation}\label{equ_5}
    \mathcal{N}_{\alpha+\beta,g_{1}+g_{2}}=\left(\prod_{\lambda\in I} \mathcal{N}_{\lambda,0}^{\alpha_{\lambda}+\beta_{\lambda}}\right)\mathcal{N}_{0,1}^{g_{1}+g_{2}}.
\end{equation}

Lastly, we conclude
$$
\begin{aligned}
\mathcal{N}_{\alpha,g_{1}}\mathcal{N}_{\beta,g_{2}}=&\left(\left(\prod_{\lambda\in I} \mathcal{N}_{\lambda,0}^{\alpha_{\lambda}}\right)\mathcal{N}_{0,1}^{g_{1}}\right)\left(\left(\prod_{\lambda\in I} \mathcal{N}_{\lambda,0}^{\beta_{\lambda}}\right)\mathcal{N}_{0,1}^{g_{2}}\right) \ \ \ \ \ \text{(by Equation \eqref{equ_4})}\\
\overset{\text{Claim 2}}{=}& \left(\prod_{\lambda\in I} \mathcal{N}_{\lambda,0}^{\alpha_{\lambda}+\beta_{\lambda}}\right)\mathcal{N}_{0,1}^{g_{1}+g_{2}}
= \mathcal{N}_{\alpha+\beta,g_{1}+g_{2}}  \ \ \ \ \ \ \ \ \ \ \text{(by Equation \eqref{equ_5}).}
\end{aligned}
$$
Thus, we have proved (FA 1).

Claim 4: (FA 2) holds.

Let $\beta\in \mathbb{N}^{(I)}, g\ge 0$ be arbitrary. Let's write $\beta=\sum_{\lambda\in I}b_{\lambda}\lambda$.
$$
\begin{aligned}
    \mathrm{Tr}\left(\mathcal{N}_{\beta,g}\right)=& \mathrm{Tr}\left(\left(\prod_{\lambda\in I}\mathcal{N}_{\lambda,0}^{b_{\lambda}}\right)\mathcal{N}_{0,1}^{g}\right)= N_{g+1}(\beta) \ \ \ \ \ \ (V2).
\end{aligned}
$$
Thus, we have proved (FA 2).
\begin{flushright}
    Q.E.D.
\end{flushright}

\section{Proof of Lemma \ref{lem_1}}\label{app_lem_1}
\noindent\textit{Proof of Lemma \ref{lem_1}.} Part 1: both $S^{\bullet}$ and $T^{\bullet}$ consist of a single irreducible module (i.e. $n=1$).

Without loss of generality, we assume $S^{\bullet}=W_{1}, T^{\bullet}=M_{1}$.

Claim 1: $
    R_{V_{1}\otimes V_{2}, W_{1}\otimes M_{1},0}=R_{V_{1}, W_{1},0}\otimes R_{V_{2}, M_{1},0}$.

The tensor product (or Kronecker product) of $R_{V_{1},W_{1}}$ and $R_{V_{2},M_{1}}$ equals the block matrix
$$
R_{V_{1},W_{1}} \otimes  R_{V_{2},M_{1}} = 
\begin{pmatrix}
    (R_{V_{1},W_{1}})_{(1,1)} R_{V_{2},M_{1}} & \cdots & (R_{V_{1},W_{1}})_{(1,l)} R_{V_{2},M_{1}}\\
    \vdots & \ddots & \vdots\\
    (R_{V_{1},W_{1}})_{(l,1)} R_{V_{2},M_{1}} & \cdots & (R_{V_{1},W_{1}})_{(l,l)} R_{V_{2},M_{1}}
\end{pmatrix}
$$
Its $(i,j)$-block equals
$$
(R_{V_{1},W_{1}})_{(i,j)} \cdot R_{V_{2},M_{1}} = 
\begin{pmatrix}
    (R_{V_{1},W_{1}})_{(i,j)}(R_{V_{2},M_{1}})_{(1,1)} & \cdots &(R_{V_{1},W_{1}})_{(i,j)}(R_{V_{2},M_{1}})_{(1,t)}\\
    \vdots & \ddots & \vdots\\
    (R_{V_{1},W_{1}})_{(i,j)}(R_{V_{2},M_{1}})_{(t,1)} & \cdots & (R_{V_{1},W_{1}})_{(i,j)}(R_{V_{2},M_{1}})_{(t,t)}
\end{pmatrix}
$$
By \cite[Theorem 5.5]{Milas96}, we have the fusion rule
$
\begin{aligned}
\mathcal{N}_{W_{1},W_{i}}^{W_{j}}\mathcal{N}_{M_{1},M_{i'}}^{M_{j'}}
=& \mathcal{N}_{W_{1}\otimes M_{1}, W_{i}\otimes M_{i'}}^{W_{j}\otimes M_{j'}}.
\end{aligned}
$
Thus, the $(i,j)$-block of $R_{V_{1},W_{1}} \otimes  R_{V_{2},M_{1}}$
equals
$$
\begin{pmatrix}
   \mathcal{N}_{W_{1},W_{i}}^{W_{j}}\mathcal{N}_{M_{1},M_{1}}^{M_{1}} & \cdots & \mathcal{N}_{W_{1},W_{i}}^{W_{j}}\mathcal{N}_{M_{1},M_{1}}^{M_{t}}\\
   \vdots & \ddots & \vdots\\
   \mathcal{N}_{W_{1},W_{i}}^{W_{j}}\mathcal{N}_{M_{1},M_{t}}^{M_{1}} & \cdots & \mathcal{N}_{W_{1},W_{i}}^{W_{j}}\mathcal{N}_{M_{1},M_{t}}^{M_{t}}\end{pmatrix} =
   \begin{pmatrix}
       \mathcal{N}_{W_{1}\otimes M_{1},W_{i}\otimes M_{1}}^{W_{j}\otimes M_{1}} & \cdots & \mathcal{N}_{W_{1}\otimes M_{1},W_{i}\otimes M_{1}}^{W_{j}\otimes M_{t}}\\
       \vdots & \ddots  & \vdots\\
       \mathcal{N}_{W_{1}\otimes M_{1},W_{i}\otimes M_{t}}^{W_{j}\otimes M_{1}} & \cdots & \mathcal{N}_{W_{1}\otimes M_{1},W_{i}\otimes M_{t}}^{W_{i}\otimes M_{t}}
   \end{pmatrix}
$$
Thus, we have proved Claim 1.

Claim 2: $R_{V_{1},V_{1},1}\otimes R_{V_{2},V_{2},1}=R_{V_{1}\otimes V_{2},V_{1}\otimes V_{2},1}$.
$$
\begin{aligned}
    R_{V_{1},V_{1},1}\otimes R_{V_{2},V_{2},1} =& \left(\sum_{i\in I}\mathrm{Tr}\left(R_{V_{1},W_{i}',0}\right)R_{V_{1},W_{i},0}\right)\otimes \left(\sum_{j\in J}\mathrm{Tr}\left(R_{V_{2},W_{j}',0}\right)R_{V_{1},W_{j},0}\right) \ \ \ \text{(Theorem \ref{thm_1})}\\
    =& \sum_{(i,j)\in I\times J} \left(\mathrm{Tr}\left(R_{V_{1},W_{i}',0}\right)R_{V_{1},W_{i},0}\right)\otimes \left(\mathrm{Tr}\left(R_{V_{2},W_{j}',0}\right)R_{V_{1},W_{j},0}\right)\\
    =& \sum_{(i,j)\in I\times J} \mathrm{Tr}\left(R_{V_{1},W_{i}',0}\right)\mathrm{Tr}\left(R_{V_{2},W_{j}',0}\right)\cdot \left(R_{V_{1},W_{i},0}\otimes R_{V_{1},W_{j},0}\right).
\end{aligned}
$$
Thus,
$$
\begin{aligned}
     R_{V_{1},V_{1},1}\otimes R_{V_{2},V_{2},1}=& \sum_{(i,j)\in I\times J} \mathrm{Tr}\left(R_{V_{1},W_{i}',0}\otimes R_{V_{2},W_{j}',0}\right)\cdot \left(R_{V_{1},W_{i},0}\otimes R_{V_{1},W_{j},0}\right)\\
    =& \sum_{(i,j)\in I\times J} \mathrm{Tr}\left(R_{V_{1}\otimes V_{2},W_{i}'\otimes W_{j}',0}\right) R_{V_{1}\otimes V_{2},W_{i}\otimes W_{j},0}  \ \ \ \ \ \ \ \ \ \ \ \text{(Claim 1)}\\
    =& \sum_{(i,j)\in I\times J} \mathrm{Tr}\left(R_{V_{1}\otimes V_{2},(W_{i}\otimes W_{j})',0}\right) R_{V_{1}\otimes V_{2},W_{i}\otimes W_{j},0}\\
    =& R_{V_{1}\otimes V_{2},V_{1}\otimes V_{2},1} \ \ \ \ \ \ \ \ \ \ \ \ \ \ \ \ \ \ \ \text{(Theorem \ref{thm_1})}.
\end{aligned}
$$
Thus, we have proved Claim 2.

Claim 3: $R_{V_{1},V_{1},g}\otimes R_{V_{2},V_{2},g}=R_{V_{1}\otimes V_{2},V_{1}\otimes V_{2},g}$, for all $g\ge 0$.

We apply induction on $g\ge 0$. The base cases $g=0$ and $g=1$ were proved in Claim 1 and Claim 2. Suppose that the claim holds for $k\le g$, we prove it for $k=g+1$.
$$
\begin{aligned}
    R_{V_{1},V_{1},g+1}\otimes R_{V_{2},V_{2},g+1}=&\left(R_{V_{1},V_{1},g}R_{V_{1},V_{1},1}\right)\otimes\left( R_{V_{2},V_{2},g}R_{V_{2},V_{2},1}\right)\ \ \ \ \text{(Theorem \ref{thm_1})}\\
    =& \left(R_{V_{1},V_{1},g} \otimes R_{V_{2},V_{2},g}\right)\left(R_{V_{1},V_{1},1}\otimes R_{V_{2},V_{2},1}\right)  \ \ \ \  \text{\cite[Proposition 2]{LGTDO}}\\
    =& R_{V_{1}\otimes V_{2},V_{1}\otimes V_{2},g}R_{V_{1}\otimes V_{2},V_{1}\otimes V_{2},1}  \ \ \ \ \ \text{(induction hypothesis)}\\
    =& R_{V_{1}\otimes V_{2},V_{1}\otimes V_{2},g+1}\ \ \ \ \text{(Theorem \ref{thm_1})}.
\end{aligned}
$$
Thus, Claim 3 holds.

Claim 4: $
    R_{V_{1}\otimes V_{2}, W_{a}\otimes M_{b},0}=R_{V_{1}, W_{a},0}\otimes R_{V_{2}, M_{b},0},
    $
    for all $1\le a\le l, 1\le b\le t$, and for all $g\ge 0$.
By Theorem \ref{thm_1}, we have the following equalities:
$$
\begin{aligned}
    R_{V_{1}\otimes V_{2},W_{1}\otimes M_{1},g} =& R_{V_{1}\otimes V_{2},W_{1}\otimes M_{1},0}R_{V_{1}\otimes V_{2},V_{1}\otimes V_{2},1}^{g} \ \ \ \ \ \\
    =& R_{V_{1}\otimes V_{2},W_{1}\otimes M_{1},0}\left(\sum_{(i,j)\in I\times J} \mathrm{Tr}\left( R_{V_{1}\otimes V_{2},(W_{i}\otimes W_{j})',0}\right) R_{V_{1}\otimes V_{2},(W_{i}\otimes W_{j},0}\right)
\end{aligned}
$$
$$
\begin{aligned}
R_{V_{1}, W_{1},g}\otimes R_{V_{2}, M_{1},g}=& \left(R_{V_{1}, W_{1},0}R_{V_{1},V_{1},g}\right)\otimes \left(R_{V_{2}, M_{1},0}R_{V_{2}, V_{2},g}\right)\ \ \ \ \text{(Theorem \ref{thm_1})}\\
=& (R_{V_{1}, W_{1},0}\otimes R_{V_{2}, M_{1},0})(R_{V_{1},V_{1},g}\otimes R_{V_{2}, V_{2},g}) \text{\cite[Proposition 2]{LGTDO}}\\
=& R_{V_{1}\otimes V_{2},W_{1}\otimes W_{2},0}\left(R_{V_{1},V_{1},g}\otimes R_{V_{2}, V_{2},g}\right) \ \ \ \text{(Claim 1)}\\
=& R_{V_{1}\otimes V_{2},W_{1}\otimes W_{2},0} R_{V_{1}\otimes V_{2},V_{1}\otimes V_{2},g} \ \ \ \text{(Claim 3)}\\
=& R_{V_{1}\otimes V_{2},W_{1}\otimes W_{2},0}\ \ \ \ \text{(Theorem \ref{thm_1})}.
\end{aligned}
$$
Thus, Claim 4 holds, and we conclude the proof of Part 1.

Part 2: assume that the length of $S^{\bullet}$ and $T^{\bullet}$ is $n$, where $n\ge 0$ is arbitrary.

Claim 5: $R_{V_{1}\otimes V_{2}, S^{\bullet}\otimes T^{\bullet},0}=R_{V_{1}, S^{\bullet},0}\otimes R_{V_{2}, T^{\bullet},0}$

We will prove Part 2 by induction on $n$. Part 1 gives the base case for $n=1$. Suppose that Claim 5 holds for $1,...,n$. We will prove that it also holds for $n+1$.
$$
\begin{aligned}
    R_{V_{1},S^{\bullet},0}\otimes R_{V_{2}, T^{\bullet},0} = & \left(R_{V_{1},S_{1},0}R_{V_{2},S_{2}\cdots S_{n},0}\right)\otimes \left(R_{V_{2}, T_{1},0}R_{V_{2}, T_{2}\cdots T_{n},0}\right)\ \ \ \ \text{(Theorem \ref{thm_1})}\\
    =& \left(R_{V_{1},S_{1},0}\otimes R_{V_{2}, T_{1},0}\right)\left(R_{V_{2},S_{2}\cdots S_{n},0}\otimes R_{V_{2}, T_{2}\cdots T_{n},0}\right) \ \ \ \ \ \text{\cite[Proposition 2]{LGTDO}}\\
    =& R_{V_{1}\otimes V_{2},S_{1}\otimes T_{1},0}R_{V_{1}\otimes V_{2}, S_{2}\otimes T_{2}\cdots S_{n}\otimes T_{n},0} \ \ \ \text{(induction hypothesis)}\\
    =& R_{V_{1}\otimes V_{2}, S^{\bullet}\otimes T^{\bullet},0}\ \ \ \ \text{(Theorem \ref{thm_1})}.
\end{aligned}
$$
Thus, we have proved Claim 5.

Claim 6: $R_{V_{1}\otimes V_{2}, S^{\bullet}\otimes T^{\bullet},g}=R_{V_{1}, S^{\bullet},g}\otimes R_{V_{2}, T^{\bullet},g}$, for all $g\ge 0$.
$$
\begin{aligned}
    R_{V_{1}, S^{\bullet},g}\otimes R_{V_{2}, T^{\bullet},g} =& \left(R_{V_{1}, S^{\bullet},0}R_{V_{1},V_{1},g}\right)\otimes \left(R_{V_{2}, T^{\bullet},0}R_{V_{2}, V_{2},g}\right)\ \ \ \ \text{(Theorem \ref{thm_1})}\\
    =& \left(R_{V_{1}, S^{\bullet},0}\otimes R_{V_{2}, T^{\bullet},0}\right)\left(R_{V_{1}, V_{1},g}\otimes R_{V_{2}, V_{2},g}\right) \ \ \ \ \ \text{\cite[Proposition 2]{LGTDO}}.
\end{aligned}
$$
$$
\begin{aligned}
    R_{V_{1}, S^{\bullet},g}\otimes R_{V_{2}, T^{\bullet},g}=& R_{V_{1}\otimes V_{2}, S^{\bullet}\otimes T^{\bullet},0}\left(R_{V_{1}, V_{1},g}\otimes R_{V_{2}, V_{2},g}\right) \ \ \ \ \ \text{(Claim 5)}\\
    =& R_{V_{1}\otimes V_{2}, S^{\bullet}\otimes T^{\bullet},0}R_{V_{1}\otimes V_{2},V_{1}\otimes V_{2},g} \ \ \ \ \ \text{(Claim 3)}\\
    =&R_{V_{1}\otimes V_{2}, S^{\bullet}\otimes T^{\bullet},g}\ \ \ \ \text{(Theorem \ref{thm_1})}.
\end{aligned}
$$
Thus, we have proved Claim 6, and here concludes the entire proof.
\begin{flushright}
    Q.E.D.
\end{flushright}

\section{Proof of Theorem \ref{thm_4}}\label{Appendix_B}
\noindent\textit{Proof of Theorem \ref{thm_4}.} 
let $V_{12}:=V_{1}\otimes V_{1}$, $c_{V_{i}}$ be the central charge of $V_{i}$, for $i=1,2$, and $c_{V_{12}}$ be the central charge of $V_{12}$. Let $W_{1},\cdots, W_{l_{1}}$ be the collection of all irreducible $V_{1}$-modules, up to isomorphism, and $M_{1},\cdots, M_{l_{2}}$ be the collection of all irreducible $V_{2}$-modules, up to isomorphism.
$$
\begin{aligned}
c_{1}\left(\mathbb{V}_{g,n}\left(V_{1},\left\{S^{\bullet}\right\}\right)\otimes \mathbb{V}_{g,n}\left(V_{2},\left\{T^{\bullet}\right\}\right)\right) =& \mathrm{rk}\mathbb{V}_{g,n}(V_{2},\left\{T^{\bullet}\right\}) c_{1}\mathbb{V}_{g,n}(V_{1},\left\{S^{\bullet}\right\}) + \\
&\mathrm{rk}\mathbb{V}_{g,n}(V_{1},\left\{S^{\bullet}\right\}) c_{1}\mathbb{V}_{g,n}(V_{2},\left\{T^{\bullet}\right\}).
\end{aligned}
$$

Step 1: we check the coefficient of $\lambda$.
$$
\begin{aligned}
    \text{Coeff}_{\lambda} c_{1}\mathbb{V}_{g,n}(V_{12}, \{S^{\bullet}\otimes T^{\bullet}\}) =& \mathrm{rk}\mathbb{V}_{g,n}\left(V_{12},\{S^{\bullet}\otimes T^{\bullet}\}\right)\frac{c_{V_{12}}}{2}
    = \mathrm{rk}\mathbb{V}_{g,n}(V_{12},\{S^{\bullet}\otimes T^{\bullet}\})\frac{c_{V_1}+ c_{V_2}}{2}\\
    =&\mathrm{rk}\mathbb{V}_{g,n}(V_{1},\{S^{\bullet}\})\mathrm{rk}\mathbb{V}_{g,n}(V_{2},\{T^{\bullet}\})\frac{c_{V_1}+ c_{V_2}}{2} \ \ \ (\text{Thm } \ref{thm_3})
\end{aligned}
$$

$$
\begin{aligned}
&\text{Coeff}_{\lambda} c_{1}\left(\mathbb{V}_{g,n}(V_{1},\{S^{\bullet}\})\otimes \mathbb{V}_{g,n}(V_{2},\{T^{\bullet}\})\right)\\
=& \mathrm{rk}\mathbb{V}_{g,n}(V_{2},\{T^{\bullet}\})\text{Coeff}_{\lambda} \mathbb{V}_{g,n}(V_{1},\{S^{\bullet}\}) +
\mathrm{rk}\mathbb{V}_{g,n}(V_{1},\{S^{\bullet}\})\text{Coeff}_{\lambda} \mathbb{V}_{g,n}(V_{2},\{T^{\bullet}\})\\
=& \mathrm{rk}\mathbb{V}_{g,n}(V_{2},\{T^{\bullet}\})\mathrm{rk}\mathbb{V}_{g,n}(V_{1},\{S^{\bullet}\})\frac{c_{V_1}}{2}+\mathrm{rk}\mathbb{V}_{g,n}(V_{1},\{S^{\bullet}\})\mathrm{rk}\mathbb{V}_{g,n}(V_{2},\{T^{\bullet}\})\frac{c_{V_2}}{2}\\
=&\mathrm{rk}\mathbb{V}_{g,n}(V_{1},\{S^{\bullet}\})\mathrm{rk}\mathbb{V}_{g,n}(V_{2},\{T^{\bullet}\})\frac{c_{V_1}+ c_{V_2}}{2}=\text{Coeff}_{\lambda} c_{1}\mathbb{V}_{g,n}(V_{12}, \{S^{\bullet}\otimes T^{\bullet}\}).
\end{aligned}
$$

Step 2: we check the coefficient of $\delta_{irr}$ (in this step, we assume $g\ge 1$). Define
$$
\begin{aligned}
    b_{irr} =& \text{Coeff}_{\delta_{irr}}c_{1}\mathbb{V}_{g,n}(V_{12}, \{S^{\bullet}\otimes T^{\bullet}\}), b_{1,irr} = \text{Coeff}_{\delta_{irr}}c_{1}\mathbb{V}_{g,n}(V_{1}, \{S^{\bullet}\}),\\
    b_{2,irr} =& \text{Coeff}_{\delta_{irr}}c_{1}\mathbb{V}_{g,n}(V_{2}, \{T^{\bullet}\}), 
b_{12,irr}=\text{Coeff}_{\delta_{irr}}c_{1}\left(\mathbb{V}_{g,n}(V_{1},\{S^{\bullet}\})\otimes \mathbb{V}_{g,n}(V_{2},\{T^{\bullet}\})\right).
\end{aligned}
$$
Then,
$$
\begin{aligned}
    b_{irr} =& \sum_{i,j}a_{W_{i}\otimes W_{j}}\mathrm{rk}\mathbb{V}_{g-1,n+2}\left(V_{12},\left\{S^{\bullet}\otimes T^{\bullet},W_{i}\otimes W_{j},(W_{i}\otimes W_{j})'\right\}\right)\\
    =& \sum_{i,j}\left(a_{W_{i}}+a_{W_{j}}\right)\mathrm{rk}\mathbb{V}_{g-1,n+2}\left(V_{12},\left\{S^{\bullet}\otimes T^{\bullet},W_{i}\otimes W_{j},W_{i}'\otimes W_{j}'\right\}\right)\\
    =& \sum_{i,j}a_{W_{i}}\mathrm{rk}\mathbb{V}_{g-1,n+2}\left(V_{12},\left\{S^{\bullet}\otimes T^{\bullet},W_{i}\otimes W_{j},W_{i}'\otimes W_{j}'\right\}\right)+\\
    & \sum_{i,j}a_{W_{j}}\mathrm{rk}\mathbb{V}_{g-1,n+2}\left(V_{12},\left\{S^{\bullet}\otimes T^{\bullet},W_{i}\otimes W_{j},W_{i}'\otimes W_{j}'\right\}\right).
\end{aligned}
$$
Therefore, 
$$
\begin{aligned}
    b_{irr} =& \sum_{i,j} a_{W_{i}}\mathrm{rk}\mathbb{V}_{g-1,n+2}\left(V_{1},\left\{S^{\bullet},W_{i},W_{i}'\})\mathrm{rk}\mathbb{V}_{g-1}(V_{2},\{T^{\bullet},M_{j},M_{j}'\right\}\right)+\\
    &\sum_{i,j} a_{W_{j}}\mathrm{rk}\mathbb{V}_{g-1,n+2}\left(V_{1},\left\{S^{\bullet},W_{i},W_{i}'\})\mathrm{rk}\mathbb{V}_{g-1}(V_{2},\{T^{\bullet},M_{j},M_{j}'\right\}\right) (\text{Thm } \ref{thm_3})\\
    =& \sum_{j}\left(\mathrm{rk}\mathbb{V}_{g-1,n+2}\left(V_{2},\left\{T^{\bullet},M_{j},M_{j}'\})\left(\sum_{i}a_{W_{i}}\mathrm{rk}\mathbb{V}_{g-1}(V_{1},\{S^{\bullet},W_{i},W_{i}'\right\}\right)\right)\right)+\\
    & \sum_{i}\left(\mathrm{rk}\mathbb{V}_{g-1,n+2}\left(V_{1},\left\{S^{\bullet},W_{i},W_{i}'\right\}\right)\left(\sum_{j}a_{M_{j}}\mathrm{rk}\mathbb{V}_{g-1,n+2}\left(V_{2},\left\{T^{\bullet},M_{j},M_{j}'\right\}\right)\right)\right).
\end{aligned}
$$
$$
\begin{aligned}
b_{irr}=&\sum_{j}\mathrm{rk}\mathbb{V}_{g-1,n+2}\left(V_{2},\left\{T^{\bullet},M_{j},M_{j}'\right\}\right)b_{1,irr}+\sum_{i}\mathrm{rk}\mathbb{V}_{g-1,n+2}\left(V_{1},\left\{S^{\bullet},W_{i},W_{i}'\right\}\right)b_{2,irr}\\
    =& \left(\sum_{j}\mathrm{rk}\mathbb{V}_{g-1,n+2}\left(V_{2},\left\{T^{\bullet},M_{j},M_{j}'\right\}\right)\right)b_{1,irr}+\left(\sum_{i}\mathrm{rk}\mathbb{V}_{g-1,n+2}\left(V_{1},\left\{S^{\bullet},W_{i},W_{i}'\right\}\right)\right)b_{2,irr}\\
    =& \mathrm{rk}\mathbb{V}_{g,n}\left(V_{2},\left\{T^{\bullet}\right\}\right)b_{1,irr}+\mathrm{rk}\mathbb{V}_{g,n}\left(V_{1},\left\{S^{\bullet}\right\}\right)b_{2,irr}     = b_{12,irr}
\end{aligned}
$$

Let $i_{0}\in \{0,\cdots,g\}, I\subseteq [n]=\{1,\cdots,n\}$.

Step 3: we check the coefficient of $\delta_{i_{0}:I}$.

Define
$$
\begin{aligned}
    b_{i_{0}:I} =& \text{Coeff}_{\delta_{i_{0}:I}}c_{1}\mathbb{V}_{g,n}(V_{12}, \{S^{\bullet}\otimes T^{\bullet}\}),
    b_{i_{0},I}^{1} = \text{Coeff}_{\delta_{i_{0}:I}}c_{1}\mathbb{V}_{g,n}(V_{1}, \{S^{\bullet}\})\\
    b_{i_{0},I}^{2} =& \text{Coeff}_{\delta_{i_{0}:I}}c_{1}\mathbb{V}_{g,n}(V_{2}, \{T^{\bullet}\}),
b_{i_{0},I}^{12}=\text{Coeff}_{\delta_{i_{0}:I}}c_{1}\left(\mathbb{V}_{g,n}(V_{1},S^{\bullet})\otimes \mathbb{V}_{g,n}(V_{2},T^{\bullet})\right).
\end{aligned}
$$
Then,
$$
\begin{aligned}
    b_{i_{0}:I} 
    = \sum_{i,j}&\left(a_{W_{i}\otimes M_{j}}\right)\mathrm{rk}\mathbb{V}_{i_{0},|I|+1}(V_{12},\{W^{I}\otimes M^{I},W_{i}\otimes M_{j}\})\mathrm{rk}\mathbb{V}_{g-i_{0},n-|I|+1}(V_{12},\{W^{I^{c}}\otimes M^{I^{c}},(W_{i}\otimes M_{j})'\})\\
    = \sum_{i,j}&\left(a_{W_{i}}+ a_{M_{j}}\right)\mathrm{rk}\mathbb{V}_{i_{0},|I|+1}(V_{12},\{W^{I}\otimes M^{I},W_{i}\otimes M_{j}\})\mathrm{rk}\mathbb{V}_{g-i_{0},n-|I|+1}(V_{12},\{W^{I^{c}}\otimes M^{I^{c}},W_{i}'\otimes M_{j}'\})\\
    = \sum_{i,j}&a_{W_{i}}\mathrm{rk}\mathbb{V}_{i_{0},|I|+1}(V_{12},\{W^{I}\otimes M^{I},W_{i}\otimes M_{j}\})\mathrm{rk}\mathbb{V}_{g-i_{0},n-|I|+1}(V_{12},\{W^{I^{c}}\otimes M^{I^{c}},W_{i}'\otimes M_{j}'\})+\\
    \sum_{i,j}&a_{M_{j}}\mathrm{rk}\mathbb{V}_{i_{0},|I|+1}(V_{12},\{W^{I}\otimes M^{I},W_{i}\otimes M_{j}\})\mathrm{rk}\mathbb{V}_{g-i_{0},n-|I|+1}(V_{12},\{W^{I^{c}}\otimes M^{I^{c}},W_{i}'\otimes M_{j}'\}).
\end{aligned}
$$
Thus, we have
$$
\begin{aligned}
    b_{i_{0}:I} = \sum_{i,j}&\left(a_{W_{i}}\mathrm{rk}\mathbb{V}_{i_{0},|I|+1}(V_{1},\{W^{I},W_{i}\})\mathrm{rk}\mathbb{V}_{i_{0},|I|+1}(V_{2},\{M^{I},M_{j}\})\right.\\
    &\left.\mathrm{rk}\mathbb{V}_{g-i_{0},n-|I|+1}(V_{1},\{W^{I^{c}},W_{i}'\})\mathrm{rk}\mathbb{V}_{g-i_{0},n-|I|+1}(V_{2},\{M^{I^{c}},M_{j}'\})\right)+\\
    \sum_{i,j}&\left(a_{W_{j}}\mathrm{rk}\mathbb{V}_{i_{0},|I|+1}(V_{1},\{W^{I},W_{i}\})\mathrm{rk}\mathbb{V}_{i_{0},|I|+1}(V_{2},\{M^{I},M_{j}\})\right.\\
    &\left.\mathrm{rk}\mathbb{V}_{g-i_{0},n-|I|+1}(V_{1},\{W^{I^{c}},W_{i}'\})\mathrm{rk}\mathbb{V}_{g-i_{0},n-|I|+1}(V_{2},\{M^{I^{c}},M_{j}'\})\right) \ \ (\text{Theorem } \ref{thm_3})
    \end{aligned}
$$
    Thus, we have
$$
\begin{aligned}
b_{i_{0},I}=\sum_{j}&\left(\mathrm{rk}\mathbb{V}_{i_{0},|I|+1}(V_{2},\{M^{I},M_{j}\})\mathrm{rk}\mathbb{V}_{g-i_{0},n-|I|+1}(V_{2},\{M^{I^{c}},M_{j}'\})\left(\sum_{i}a_{W_{i}}\mathrm{rk}\mathbb{V}_{i_{0},|I|+1}(V_{1},\{W^{I},W_{i}\})\cdot\right.\right.\\
    &\left.\left.\mathrm{rk}\mathbb{V}_{g-i_{0},n-|I|+1}(V_{1},\{W^{I^{c}},W_{i}'\})\right)\right)+\\
    \sum_{i}&\left(\mathrm{rk}\mathbb{V}_{i_{0},|I|+1}(V_{1},\{W^{I},W_{i}\})\mathrm{rk}\mathbb{V}_{g-i_{0},n-|I|+1}(V_{1},\{W^{I^{c}},W_{i}'\})\left(\sum_{j}a_{W_{j}}\mathrm{rk}\mathbb{V}_{i_{0},|I|+1}(V_{2},\{M^{I},M_{j}\})\cdot\right.\right.\\
    &\left.\left.\mathrm{rk}\mathbb{V}_{g-i_{0},n-|I|+1}(V_{2},\{M^{I^{c}},M_{j}'\})\right)\right).
\end{aligned}
$$
Therefore, we have
$$
\begin{aligned}
b_{i_{0}:I}
    =\sum_{j}&\mathrm{rk}\mathbb{V}_{i_{0},|I|+1}(V_{2},\{M^{I},M_{j}\})\mathrm{rk}\mathbb{V}_{g-i_{0},n-|I|+1}(V_{2},\{M^{I^{c}},M_{j}'\})b_{i_{0}:I}^{1}+   \\
    \sum_{i}&\mathrm{rk}\mathbb{V}_{i_{0},|I|+1}(V_{1},\{W^{I},W_{i}\})\mathrm{rk}\mathbb{V}_{g-i_{0},n-|I|+1}(V_{1},\{W^{I^{c}},W_{i}'\})b_{i_{0}:I}^{2}\\
    = \sum_{j}&\left(\mathrm{rk}\mathbb{V}_{i_{0},|I|+1}(V_{2},\{M^{I},M_{j}\})\mathrm{rk}\mathbb{V}_{g-i_{0},n-|I|+1}(V_{2},\{M^{I^{c}},M_{j}'\})\right)b_{i_{0}:I}^{1}+\\
    \sum_{i}&\left(\mathrm{rk}\mathbb{V}_{i_{0},|I|+1}(V_{1},\{W^{I},W_{i}\})\mathrm{rk}\mathbb{V}_{g-i_{0},n-|I|+1}(V_{1},\{W^{I^{c}},W_{i}'\})\right)b_{i_{0}:I}^{2}\\
    = \mathrm{r}\mathrm{k}\mathbb{V}&_{g,n}(V_{2},\{T^{\bullet}\})b_{i_{0}:I}^{1}+\mathrm{rk}\mathbb{V}_{g,n}(V_{1},\{S^{\bullet}\})b_{i_{0}:I}^{2}= b_{i:I}^{12}.
\end{aligned}
$$

\begin{flushright}
    Q.E.D.
\end{flushright}

\section{Proof of Theorem \ref{theorem_7}}\label{app_theorem_7}
\begin{lemma}\label{lem}
    Let $V_{2,2l+1}$ be the Virasoro VOA of central charge $c_{2,2l+2}$ ($l\ge 1$), and $W_{\min}$ be the irreducible module with minimal conformal weight. Then, the FA-matrix associated with $R_{W_{\min},0}$ is a symmetric, upper-left triangular $l\times l$ matrix 
$$
    R_{W_{\min},0}=
  \begin{pmatrix}
    1 & 1 & \cdots & 1 & 1 \\
    1 & 1 & \cdots & 1 & 0\\
    &\vdots \\
    1 & 0 & \cdots & 0 & 0
  \end{pmatrix}.
$$
\end{lemma}
\noindent\textit{Proof:}  $V_{2,2l+1}$ has $l$ irreducible modules \cite[Theorem 4.2]{Wang93}, and we label them from $W_{\min}=W_{0}$ to $W_{l-1}$, as increasing order of conformal weight. It is equivalent to show that for a Virasoro VOA of central charge $c_{2,2l+2}$, we have
$$
        \mathrm{rank}\mathbb{V}_{0}(V_{2,2l+1},\{W_{\min}, W_{k-t}, W_{i}\})=
        \left\{\begin{aligned}
            0&,\text{ if }i\ge t,\\
            1&,\text{ if }i<t
        \end{aligned}
        \right.,
        $$ for all $t\in \{0,1,...,k-1\}$. By \cite[Theorem 4.3]{Wang93}, the fusion rules are$$
\begin{aligned}
L(c,h_{1,k})\times L(c,h_{1,t})=& N_{(1,k),(1,t)}^{(1,1)}L(c,h_{1,1})+\cdots+ N_{(1,k),(1,t)}^{(1,2k)}L(c,h_{1,2k})\\
=& \left(N_{(1,k),(1,t)}^{(1,1)}+N_{(1,k),(1,t)}^{(1,2k)}\right)W_{k}+\cdots + \left(N_{(1,k),(1,t)}^{(1,k)}+N_{(1,k),(1,t)}^{(1,k+1)}\right)W_{\min}
\end{aligned}
$$
For $i\in \{0,...,k-1\}$, the coefficient of $W_{i}$ is 
$$
C_{i}:=N_{(1,k),(1,t)}^{(1,k-i)}+N_{(1,k),(1,t)}^{(1,k+i+1)}
$$
Case 1: $i$ and $t$ have the same parity (i.e., are both even or both odd).

$k-i+k+t$ is even, so 
$
C_{i}=N_{(1,k),(1,t)}^{(1,k+i+1)}.
$
Let's consider the inequalities of the admissible triple:
\begin{enumerate}
    \item $0<k,t,k+1+1<2k+1$
    \item $k+t+k+i+1<2q=4k+2$
    \item $k<k+i+1+t,t<k+k+i+1,k+i+1<k+t$
\end{enumerate}
We can see that the only non-trivial constraint is $i+1<t$, which is equivalent to $i<t$, since $i,t$ have the same parity. That is, in Case 1, $C_{i}=1$ if and only if $i<t$, and is $0$ otherwise.

Case 2: $i,t$ have the different parity.
$$
C_{i}=N_{(1,k),(1,t)}^{k-i}
$$
Similarly, let's consider the conditions of admissible triple
\begin{enumerate}
    \item $0<k,t,k-1<q=2k+1$
    \item $k+t+k-i<2q=2k+2$
    \item $k<t+k-i, t<k+k-i, k-i<k+t$
\end{enumerate}
We can see that the only nontrivial inequality is $k<t+k-i$, which is equivalent to $i<t$. Thus, $C_{i}=1$ is and only if $i<t$, and is $0$ otherwise.
\begin{flushright}
    Q.E.D.
\end{flushright}

\noindent\textit{Proof of Theorem \ref{theorem_7}.} By \cite[Theorem 4.2]{Wang93}, there are $l$ irreducible modules of $V_{2,2l+1}$, up to isomorphism. Let's number them as $W_{\min},...,W_{l-1}$, by an increasing order of conformal weight.

Let
$$
f_{l}(z):=f_{V_{2,2l+1},W_{\min}}(z)=\sum_{n=0}^{\infty}\mathrm{rank}\mathbb{V}_{0,n+3}\left(V_{2,2l+1},\left\{W_{\min}^{n+3}\right\}\right)z^{n}
$$
be the generating function associated with $W_{\min}$. For clarity, let's fixed the central charge as $c_{2,2l+1}$, and denote $R_{l}=R_{W_{\min},0}$.
Consider the matrix
\begin{equation}\label{4}
M_{l}:=\sum_{n=0}^{\infty} R_{l}^{n+1}z^{n}=R_{l}\left(\sum_{n=0}^{\infty} R_{l}^{n}z^{n}\right)=\frac{R_{l}}{1-R_{l}z}.
\end{equation}
Notice that the von Neumann series expansion holds near $z=0$, and the power series $f_{l}(z)$ is just the $(1,1)$-entry of $M_{l}$. By linear algebra facts, we can check that $\det R_{l}=\pm 1$, for all $l$, so $R_{l}^{-1}$ exists. We can actually compute
$$
R_{l}^{-1}=
\begin{pmatrix}
    0 & 0 & \cdots &   & 1\\
    0 &   &        & 1 & -1\\
    \vdots & & & & 0\\
    0 & 1 & -1\\
    1 & -1 & 0 & \cdots & 0
\end{pmatrix}
$$
Thus, equation (\ref{4}) becomes
\begin{equation}
M_{l}=\sum_{n=0}^{\infty} R_{l}^{n+1}z^{n}=R_{l}\left(\sum_{n=0}^{\infty} R_{l}^{n}z^{n}\right)=\frac{R_{l}}{1-R_{l}z}=\frac{1}{R_{l}^{-1}-z}.
\end{equation}
Actually, $R_{l}^{-1}-z$ means $R_{l}^{-1}-z\mathrm{Id}_{l}$.
Notice that 
$$
g_{l}(z)=(M_{l})_{1,1}=\left(\frac{1}{R_{l}^{-1}-z}\right)_{1,1}=\frac{1}{\det \left(R_{l}^{-1}-z\right)}\cdot \det C_{l}
$$
where $C_{l}$ is the bottom-right $(l-1)\times (l-1)$-minor of $(R^{-1}_{l}-z)$. Observe that the central $(l-2)\times (l-2)$ sub-matrix of $R_{l}^{-1}-z$ is just $R_{l-2}^{-1}-z$. That is,
$$
R_{l}^{-1}-z=
\begin{pmatrix}
    -z & 0 & \cdots & 0 & 1\\
    0  &   &        &   & -1\\
    \vdots & & R_{l-2}^{-1}-z & & 0\\
    0 & & & & \vdots\\
    1 & -1 & 0 & \cdots & -z
\end{pmatrix}
$$
Thus, we can compute its determinant, by expanding the leftmost column
$$
\begin{aligned}
    \det (R_{l}^{-1}-z)&=-z \det C_{l}+(-1)^{l+1}\det(\text{ upper right minor})\\
    &=-z\det C_{l}-\det \left(R_{l-2}^{-1}-z\right)
\end{aligned}
$$
Notice that we haven't use the parity of $l$ yet, but the term $(-1)^{l+1}$ cancels anyway. Let's calculate the determinant of $C_{l}$, by expanding the rightmost column, and we get
$
\det C_{l}=-z\det \left(R_{l-2}^{-1}-z\right)-\det C_{l-2}.
$
Let $B_{l}=R_{l-2}^{-1}-z$. Consider its $(1,1)$-entry, and we get
\begin{equation}\label{5}
\begin{aligned}
(M_{l})_{1,1}&=\left(B_{l}^{-1}\right)_{1,1}=\frac{\det C_{l}}{\det B_{l}}
=\frac{\det C_{l}}{-z\det C_{l}-\det B_{l-2}}=\frac{1}{-z-\frac{\det B_{l-2}}{\det C_{l}}}\\
&=\frac{1}{-z-\frac{\det B_{l-2}}{-z\det B_{l-2}-\det C_{l-2}}}=\frac{1}{-z-\frac{1}{-z-\frac{\det C_{l-2}}{\det B_{l-2}}}}.
\end{aligned}
\end{equation}

Case 1: $l$ is even.\\
Claim: $f_{l}=\frac{\det C_{l}}{\det B_{l}}$\\
We prove it by induction. We can check that it is true for $l=2$, since
$$
\left(A_{2}^{-1}-z\right)^{-1}=
\begin{pmatrix}
    -z & 1\\
    z & -1-z
\end{pmatrix}.
$$
Thus, considering the equation \ref{5}, it directly follows inductively, for all even $l$.

Case 2: $l$ is odd.\\
Define $h_{l}(z)$ as the following:
$$
    h_{1}(z)=-z+1, \ \ 
    h_{l+1}(z)=\frac{1}{-h_{l}(z)-1}.
$$
Notice that $h_{l}(z)$ is similar to $f_{l+1}(z)$, except that the bottom layer is $-z+1$, instead of $-z-1$. It is easy to check that
$
\left(\left(R_{l}^{-1}-z\right)^{-1}\right)_{1,1}=h_{l}(z),
$
for $l=1$ and $l=3$. Similarly, the induction step follows from the equation (\ref{5}).
\begin{flushright}
    Q.E.D.
\end{flushright}

\bibliographystyle{amsalpha}

\bibliography{reference}

@article {Ara12,
    AUTHOR = {Arakawa, Tomoyuki},
     TITLE = {A remark on the {$C_2$}-cofiniteness condition on vertex
              algebras},
   JOURNAL = {Math. Z.},
  FJOURNAL = {Mathematische Zeitschrift},
    VOLUME = {270},
      YEAR = {2012},
    NUMBER = {1-2},
     PAGES = {559--575},
      ISSN = {0025-5874,1432-1823},
   MRCLASS = {17B69 (17B68)},
  MRNUMBER = {2875849},
MRREVIEWER = {Hai\ Sheng\ Li},
       DOI = {10.1007/s00209-010-0812-4},
       URL = {https://doi.org/10.1007/s00209-010-0812-4},
}

@article {ALY,
    AUTHOR = {Arakawa, Tomoyuki and Lam, Ching Hung and Yamada, Hiromichi},
     TITLE = {Parafermion vertex operator algebras and {$W$}-algebras},
   JOURNAL = {Trans. Amer. Math. Soc.},
  FJOURNAL = {Transactions of the American Mathematical Society},
    VOLUME = {371},
      YEAR = {2019},
    NUMBER = {6},
     PAGES = {4277--4301},
      ISSN = {0002-9947,1088-6850},
   MRCLASS = {17B69 (17B65)},
  MRNUMBER = {3917223},
MRREVIEWER = {Pierluigi\ M\"oseneder Frajria},
       DOI = {10.1090/tran/7547},
       URL = {https://doi.org/10.1090/tran/7547},
}

@article{AYY,
   title={$\mathbb{Z}_k$-code vertex operator algebras},
   volume={73},
   ISSN={0025-5645},
   url={http://dx.doi.org/10.2969/jmsj/83278327},
   DOI={10.2969/jmsj/83278327},
   number={1},
   journal={Journal of the Mathematical Society of Japan},
   publisher={Mathematical Society of Japan (Project Euclid)},
   author={ARAKAWA, Tomoyuki and YAMADA, Hiromichi and YAMAUCHI, Hiroshi},
   year={2021},
   month=jan }

@article {B1,
    AUTHOR = {Borcherds, Richard E.},
     TITLE = {Vertex algebras, {K}ac-{M}oody algebras, and the {M}onster},
   JOURNAL = {Proc. Nat. Acad. Sci. U.S.A.},
  FJOURNAL = {Proceedings of the National Academy of Sciences of the United
              States of America},
    VOLUME = {83},
      YEAR = {1986},
    NUMBER = {10},
     PAGES = {3068--3071},
      ISSN = {0027-8424},
   MRCLASS = {17B67 (17B10 20D08)},
  MRNUMBER = {843307},
MRREVIEWER = {S.\ I.\ Gel\cprime fand},
       DOI = {10.1073/pnas.83.10.3068},
       URL = {https://doi.org/10.1073/pnas.83.10.3068},
}

@inproceedings {Bea96,
    AUTHOR = {Beauville, Arnaud},
     TITLE = {Conformal blocks, fusion rules and the {V}erlinde formula},
 BOOKTITLE = {Proceedings of the {H}irzebruch 65 {C}onference on {A}lgebraic
              {G}eometry ({R}amat {G}an, 1993)},
    SERIES = {Israel Math. Conf. Proc.},
    VOLUME = {9},
     PAGES = {75--96},
 PUBLISHER = {Bar-Ilan Univ., Ramat Gan},
      YEAR = {1996},
   MRCLASS = {17B67 (14D20 17B68 17B69 81T40)},
  MRNUMBER = {1360497},
MRREVIEWER = {Alex\ Jay\ Feingold},
}

@article {BG,
    AUTHOR = {Belkale, P. and Gibney, A.},
     TITLE = {Basepoint free cycles on {$\overline{{\rm M}}_{0,n}$} from
              {G}romov-{W}itten theory},
   JOURNAL = {Int. Math. Res. Not. IMRN},
  FJOURNAL = {International Mathematics Research Notices. IMRN},
      YEAR = {2021},
    NUMBER = {2},
     PAGES = {855--884},
      ISSN = {1073-7928,1687-0247},
   MRCLASS = {14N35 (14H10)},
  MRNUMBER = {4201956},
MRREVIEWER = {Luca\ Battistella},
       DOI = {10.1093/imrn/rnz184},
       URL = {https://doi.org/10.1093/imrn/rnz184},
}

@article {BL,
    AUTHOR = {Beauville, Arnaud and Laszlo, Yves},
     TITLE = {Conformal blocks and generalized theta functions},
   JOURNAL = {Comm. Math. Phys.},
  FJOURNAL = {Communications in Mathematical Physics},
    VOLUME = {164},
      YEAR = {1994},
    NUMBER = {2},
     PAGES = {385--419},
      ISSN = {0010-3616,1432-0916},
   MRCLASS = {14D20 (14H60 17B67)},
  MRNUMBER = {1289330},
MRREVIEWER = {Emma\ Previato},
       URL = {http://projecteuclid.org/euclid.cmp/1104270837},
}

@misc{Cha,
      title={Positivity of coinvariant divisors on $\overline{\mathrm{M}}_{0,n}$ and the parafermions}, 
      author={Avik Chakravarty},
      year={2025},
      eprint={2506.17593},
      archivePrefix={arXiv},
      primaryClass={math.AG},
      url={https://arxiv.org/abs/2506.17593}, 
}

@misc{Choi,
      title={Conformal Block Divisors for Discrete Series Virasoro VOA $\text{Vir}_{2k+1,2}$}, 
      author={Daebeom Choi},
      year={2025},
      eprint={2502.21270},
      archivePrefix={arXiv},
      primaryClass={math.AG},
      url={https://arxiv.org/abs/2502.21270}, 
}

@article {DG,
    AUTHOR = {Damiolini, Chiara and Gibney, Angela},
     TITLE = {On global generation of vector bundles on the moduli space of
              curves from representations of vertex operator algebras},
   JOURNAL = {Algebr. Geom.},
  FJOURNAL = {Algebraic Geometry},
    VOLUME = {10},
      YEAR = {2023},
    NUMBER = {3},
     PAGES = {298--326},
      ISSN = {2313-1691,2214-2584},
   MRCLASS = {14H10 (14D20 14D21 17B69 81R10)},
  MRNUMBER = {4583950},
       DOI = {10.14231/ag-2023-010},
       URL = {https://doi.org/10.14231/ag-2023-010},
}

@article {DGT22b,
    AUTHOR = {Damiolini, Chiara and Gibney, Angela and Tarasca, Nicola},
     TITLE = {Conformal blocks from vertex algebras and their connections on
              {$\overline{ M}_{g, n}$}},
   JOURNAL = {Geom. Topol.},
  FJOURNAL = {Geometry \& Topology},
    VOLUME = {25},
      YEAR = {2021},
    NUMBER = {5},
     PAGES = {2235--2286},
      ISSN = {1465-3060,1364-0380},
   MRCLASS = {14H10 (14C17 16D90 17B69 81R10 81T40)},
  MRNUMBER = {4310890},
MRREVIEWER = {Zhenbo\ Qin},
       DOI = {10.2140/gt.2021.25.2235},
       URL = {https://doi.org/10.2140/gt.2021.25.2235},
}

@incollection {DGK1,
    AUTHOR = {Damiolini, Chiara and Gibney, Angela and Krashen, Daniel},
     TITLE = {Factorization presentations},
 BOOKTITLE = {Higher dimensional algebraic geometry---a volume in honor of
              {V}. {V}. {S}hokurov},
    SERIES = {London Math. Soc. Lecture Note Ser.},
    VOLUME = {489},
     PAGES = {163--191},
 PUBLISHER = {Cambridge Univ. Press, Cambridge},
      YEAR = {2025},
      ISBN = {978-1-009-39624-0},
   MRCLASS = {14H10 (17B69 81R10 81T40)},
  MRNUMBER = {4844631},
}

@article{DGT22a,
  author = {Damiolini, Chiara and Gibney, Angela and Tarasca, Nicola},
  journal = {Annales Scientifiques de l'École Normale Supérieure},
  title = {On Factorization and Vector Bundles of Conformal Blocks from Vertex Algebras},
  doi = {https://smf.emath.fr/publications/sur-la-factorisation-et-les-fibres-vectoriels-de-blocs-conformes-des-algebres-vertex?language=en},
  volume = {57},
  number = {1},
  pages = {241-292},
  year = {2024}
}

@incollection {DGT3,
    AUTHOR = {Damiolini, Chiara and Gibney, Angela and Tarasca, Nicola},
     TITLE = {Vertex algebras of {C}oh{FT}-type},
 BOOKTITLE = {Facets of algebraic geometry. {V}ol. {I}},
    SERIES = {London Math. Soc. Lecture Note Ser.},
    VOLUME = {472},
     PAGES = {164--189},
 PUBLISHER = {Cambridge Univ. Press, Cambridge},
      YEAR = {2022},
      ISBN = {978-1-108-79250-9; 978-1-108-87006-1},
   MRCLASS = {17B69 (14H10 14H81)},
  MRNUMBER = {4381900},
MRREVIEWER = {Philsang\ Yoo},
}

@article {DLM,
    AUTHOR = {Dong, Chongying and Li, Haisheng and Mason, Geoffrey},
     TITLE = {Modular-invariance of trace functions in orbifold theory and
              generalized {M}oonshine},
   JOURNAL = {Comm. Math. Phys.},
  FJOURNAL = {Communications in Mathematical Physics},
    VOLUME = {214},
      YEAR = {2000},
    NUMBER = {1},
     PAGES = {1--56},
      ISSN = {0010-3616,1432-0916},
   MRCLASS = {17B69 (11F22)},
  MRNUMBER = {1794264},
MRREVIEWER = {Vassily\ Gorbounov},
       DOI = {10.1007/s002200000242},
       URL = {https://doi.org/10.1007/s002200000242},
}

@misc{DLM2,
      title={Regularity of rational vertex operator algebras}, 
      author={Chongying Dong and Haisheng Li and Geoffrey Mason},
      year={1995},
      eprint={q-alg/9508018},
      archivePrefix={arXiv},
      primaryClass={q-alg},
      url={https://arxiv.org/abs/q-alg/9508018}, 
}

@article{DLWY,
    author = "Dong, Chongying and Lam, Ching Hung and Wang, Qing and Yamada, Hiromichi",
    title = "{The Structure of parafermion vertex operator algebras}",
    eprint = "0904.2758",
    archivePrefix = "arXiv",
    primaryClass = "math.QA",
    doi = "10.1007/s00220-010-1114-8",
    journal = "Commun. Math. Phys.",
    volume = "299",
    pages = "783--792",
    year = "2010"
}

@incollection {DMZ94,
    AUTHOR = {Dong, Chongying and Mason, Geoffrey and Zhu, Yongchang},
     TITLE = {Discrete series of the {V}irasoro algebra and the moonshine
              module},
 BOOKTITLE = {Algebraic groups and their generalizations: quantum and
              infinite-dimensional methods ({U}niversity {P}ark, {PA},
              1991)},
    SERIES = {Proc. Sympos. Pure Math.},
    VOLUME = {56, Part 2},
     PAGES = {295--316},
 PUBLISHER = {Amer. Math. Soc., Providence, RI},
      YEAR = {1994},
      ISBN = {0-8218-1541-5},
   MRCLASS = {17B69 (17B68)},
  MRNUMBER = {1278737},
MRREVIEWER = {F\"usun\ Akman},
       DOI = {10.1090/pspum/056.2/1278737},
       URL = {https://doi.org/10.1090/pspum/056.2/1278737},
}

@article {Dong93,
    AUTHOR = {Dong, Chongying},
     TITLE = {Vertex algebras associated with even lattices},
   JOURNAL = {J. Algebra},
  FJOURNAL = {Journal of Algebra},
    VOLUME = {161},
      YEAR = {1993},
    NUMBER = {1},
     PAGES = {245--265},
      ISSN = {0021-8693,1090-266X},
   MRCLASS = {17B65 (81R10)},
  MRNUMBER = {1245855},
MRREVIEWER = {Vjatcheslav\ Futorny},
       DOI = {10.1006/jabr.1993.1217},
       URL = {https://doi.org/10.1006/jabr.1993.1217},
}

@book {DL93,
    AUTHOR = {Dong, Chongying and Lepowsky, James},
     TITLE = {Generalized vertex algebras and relative vertex operators},
    SERIES = {Progress in Mathematics},
    VOLUME = {112},
 PUBLISHER = {Birkh\"auser Boston, Inc., Boston, MA},
      YEAR = {1993},
     PAGES = {x+202},
      ISBN = {0-8176-3721-4},
   MRCLASS = {17B69 (17B65 81R10 81T40)},
  MRNUMBER = {1233387},
MRREVIEWER = {Naihuan\ Jing},
       DOI = {10.1007/978-1-4612-0353-7},
       URL = {https://doi.org/10.1007/978-1-4612-0353-7},
}

@incollection {F2,
    AUTHOR = {Fakhruddin, Najmuddin},
     TITLE = {Chern classes of conformal blocks},
 BOOKTITLE = {Compact moduli spaces and vector bundles},
    SERIES = {Contemp. Math.},
    VOLUME = {564},
     PAGES = {145--176},
 PUBLISHER = {Amer. Math. Soc., Providence, RI},
      YEAR = {2012},
      ISBN = {978-0-8218-6899-7},
   MRCLASS = {14C17 (14D23 81T40)},
  MRNUMBER = {2894632},
MRREVIEWER = {Dmitry\ Kerner},
       DOI = {10.1090/conm/564/11148},
       URL = {https://doi.org/10.1090/conm/564/11148},
}

@article {F3,
    AUTHOR = {Faltings, Gerd},
     TITLE = {A proof for the {V}erlinde formula},
   JOURNAL = {J. Algebraic Geom.},
  FJOURNAL = {Journal of Algebraic Geometry},
    VOLUME = {3},
      YEAR = {1994},
    NUMBER = {2},
     PAGES = {347--374},
      ISSN = {1056-3911,1534-7486},
   MRCLASS = {14D20},
  MRNUMBER = {1257326},
MRREVIEWER = {Jean-Marc\ Dr\'ezet},
}

@article {FHL,
    AUTHOR = {Frenkel, Igor B. and Huang, Yi-Zhi and Lepowsky, James},
     TITLE = {On axiomatic approaches to vertex operator algebras and
              modules},
   JOURNAL = {Mem. Amer. Math. Soc.},
  FJOURNAL = {Memoirs of the American Mathematical Society},
    VOLUME = {104},
      YEAR = {1993},
    NUMBER = {494},
     PAGES = {viii+64},
      ISSN = {0065-9266,1947-6221},
   MRCLASS = {17B37},
  MRNUMBER = {1142494},
MRREVIEWER = {Chong\ Ying\ Dong},
       DOI = {10.1090/memo/0494},
       URL = {https://doi.org/10.1090/memo/0494},
}

@article {FLM2,
    AUTHOR = {Frenkel, I. B. and Lepowsky, J. and Meurman, A.},
     TITLE = {A natural representation of the {F}ischer-{G}riess {M}onster
              with the modular function {$J$}\ as character},
   JOURNAL = {Proc. Nat. Acad. Sci. U.S.A.},
  FJOURNAL = {Proceedings of the National Academy of Sciences of the United
              States of America},
    VOLUME = {81},
      YEAR = {1984},
    NUMBER = {10},
     PAGES = {3256--3260},
      ISSN = {0027-8424},
   MRCLASS = {20D08 (17B67 20C20)},
  MRNUMBER = {747596},
       DOI = {10.1073/pnas.81.10.3256},
       URL = {https://doi.org/10.1073/pnas.81.10.3256},
}

@article {FZ92,
    AUTHOR = {Frenkel, Igor B. and Zhu, Yongchang},
     TITLE = {Vertex operator algebras associated to representations of
              affine and {V}irasoro algebras},
   JOURNAL = {Duke Math. J.},
  FJOURNAL = {Duke Mathematical Journal},
    VOLUME = {66},
      YEAR = {1992},
    NUMBER = {1},
     PAGES = {123--168},
      ISSN = {0012-7094,1547-7398},
   MRCLASS = {17B68 (17B67)},
  MRNUMBER = {1159433},
MRREVIEWER = {Geoffrey\ Mason},
       DOI = {10.1215/S0012-7094-92-06604-X},
       URL = {https://doi.org/10.1215/S0012-7094-92-06604-X},
}

@article {G,
    AUTHOR = {Ghouila-Houri, Alain},
     TITLE = {Caract\'erisation des matrices totalement unimodulaires},
   JOURNAL = {C. R. Acad. Sci. Paris},
  FJOURNAL = {Comptes Rendus Hebdomadaires des S\'eances de l'Acad\'emie des
              Sciences},
    VOLUME = {254},
      YEAR = {1962},
     PAGES = {1192--1194},
      ISSN = {0001-4036},
   MRCLASS = {15.48},
  MRNUMBER = {132752},
MRREVIEWER = {O.\ Taussky-Todd},
}

@article {GKM,
AUTHOR = {Gibney, Angela and Keel, Sean and Morrison, Ian},
TITLE = {Towards the ample cone of {$\overline{M}_{g,n}$}},
JOURNAL = {J. Amer. Math. Soc.},
FJOURNAL = {Journal of the American Mathematical Society},
    VOLUME = {15},
      YEAR = {2002},
    NUMBER = {2},
     PAGES = {273--294},
      ISSN = {0894-0347,1088-6834},
   MRCLASS = {14H10 (14C17 14E30)},
  MRNUMBER = {1887636},
MRREVIEWER = {Dan\ Avritzer},
       DOI = {10.1090/S0894-0347-01-00384-8},
       URL = {https://doi.org/10.1090/S0894-0347-01-00384-8},
}

@misc{GL,
      title={Applications of the factorization theorem of conformal blocks to vertex operator algebras}, 
      author={Xu Gao and Jianqi Liu},
      year={2025},
      eprint={2508.01294},
      archivePrefix={arXiv},
      primaryClass={math.QA},
      url={https://arxiv.org/abs/2508.01294}, 
}

@article {GR,
    AUTHOR = {Gannon, Terry and Riesen, Andrew},
     TITLE = {Orbifolds of pointed vertex operator algebras {I}},
   JOURNAL = {Adv. Math.},
  FJOURNAL = {Advances in Mathematics},
    VOLUME = {482},
      YEAR = {2025},
     PAGES = {Paper No. 110546},
      ISSN = {0001-8708,1090-2082},
   MRCLASS = {17B69 (16T05 18M20)},
  MRNUMBER = {4968024},
       DOI = {10.1016/j.aim.2025.110546},
       URL = {https://doi.org/10.1016/j.aim.2025.110546},
}

@article {T,
    AUTHOR = {Thaddeus, Michael},
     TITLE = {Stable pairs, linear systems and the {V}erlinde formula},
   JOURNAL = {Invent. Math.},
  FJOURNAL = {Inventiones Mathematicae},
    VOLUME = {117},
      YEAR = {1994},
    NUMBER = {2},
     PAGES = {317--353},
      ISSN = {0020-9910,1432-1297},
   MRCLASS = {14D20 (14H60 32L05 81T40)},
  MRNUMBER = {1273268},
MRREVIEWER = {Steven\ B.\ Bradlow},
       DOI = {10.1007/BF01232244},
       URL = {https://doi.org/10.1007/BF01232244},
}

@article {KN,
    AUTHOR = {Kumar, Shrawan and Narasimhan, M. S. and Ramanathan, A.},
     TITLE = {Infinite {G}rassmannians and moduli spaces of {$G$}-bundles},
   JOURNAL = {Math. Ann.},
  FJOURNAL = {Mathematische Annalen},
    VOLUME = {300},
      YEAR = {1994},
    NUMBER = {1},
     PAGES = {41--75},
      ISSN = {0025-5831,1432-1807},
   MRCLASS = {14D20 (14H60 14K25 17B67 22E67 81T40)},
  MRNUMBER = {1289830},
MRREVIEWER = {Emma\ Previato},
       DOI = {10.1007/BF01450475},
       URL = {https://doi.org/10.1007/BF01450475},
}

@article {LGTDO,
    AUTHOR = {Liu, Shuangzhe and Trenkler, G\"otz and Kollo, T\~onu and von
              Rosen, Dietrich and Baksalary, Oskar Maria},
     TITLE = {Professor {H}einz {N}eudecker and matrix differential
              calculus},
   JOURNAL = {Statist. Papers},
  FJOURNAL = {Statistical Papers},
    VOLUME = {65},
      YEAR = {2024},
    NUMBER = {4},
     PAGES = {2605--2639},
      ISSN = {0932-5026,1613-9798},
   MRCLASS = {99-01},
  MRNUMBER = {4753734},
       DOI = {10.1007/s00362-023-01499-w},
       URL = {https://doi.org/10.1007/s00362-023-01499-w},
}

@misc{Milas96,
      title={Tensor product of Vertex operator algebras}, 
      author={Antun Milas},
      year={1996},
      eprint={q-alg/9602026},
      archivePrefix={arXiv},
      primaryClass={q-alg},
      url={https://arxiv.org/abs/q-alg/9602026}, 
}

@misc{M,
      title={On rationality for $C_2$-cofinite vertex operator algebras}, 
      author={Robert McRae},
      year={2021},
      eprint={2108.01898},
      archivePrefix={arXiv},
      primaryClass={math.QA},
      url={https://arxiv.org/abs/2108.01898}, 
}

@misc{Miy,
      title={Modular invariance of vertex operator algebras satisfying $C_2$-cofiniteness}, 
      author={Masahiko Miyamoto},
      year={2002},
      eprint={math/0209101},
      archivePrefix={arXiv},
      primaryClass={math.QA},
      url={https://arxiv.org/abs/math/0209101}, 
}

@article {NT,
    AUTHOR = {Nagatomo, Kiyokazu and Tsuchiya, Akihiro},
     TITLE = {Conformal field theories associated to regular chiral vertex
              operator algebras. {I}. {T}heories over the projective line},
   JOURNAL = {Duke Math. J.},
  FJOURNAL = {Duke Mathematical Journal},
    VOLUME = {128},
      YEAR = {2005},
    NUMBER = {3},
     PAGES = {393--471},
      ISSN = {0012-7094,1547-7398},
   MRCLASS = {81T40 (17B69 81R10)},
  MRNUMBER = {2145740},
MRREVIEWER = {Hai\ Sheng\ Li},
       DOI = {10.1215/S0012-7094-04-12831-3},
       URL = {https://doi.org/10.1215/S0012-7094-04-12831-3},
}

@book {SS,
    AUTHOR = {Sch\"utt, Matthias and Shioda, Tetsuji},
     TITLE = {Mordell-{W}eil lattices},
    SERIES = {Ergebnisse der Mathematik und ihrer Grenzgebiete. 3. Folge. A
              Series of Modern Surveys in Mathematics [Results in
              Mathematics and Related Areas. 3rd Series. A Series of Modern
              Surveys in Mathematics]},
    VOLUME = {70},
 PUBLISHER = {Springer, Singapore},
      YEAR = {2019},
     PAGES = {xvi+431},
      ISBN = {978-981-32-9300-7; 978-981-32-9301-4},
   MRCLASS = {11G05 (11F80 11H06 14J20 14J26 14J27 14J28)},
  MRNUMBER = {3970314},
MRREVIEWER = {\'Alvaro\ Lozano-Robledo},
       DOI = {10.1007/978-981-32-9301-4},
       URL = {https://doi.org/10.1007/978-981-32-9301-4},
}

@incollection {TUY,
    AUTHOR = {Tsuchiya, Akihiro and Ueno, Kenji and Yamada, Yasuhiko},
     TITLE = {Conformal field theory on universal family of stable curves
              with gauge symmetries},
 BOOKTITLE = {Integrable systems in quantum field theory and statistical
              mechanics},
    SERIES = {Adv. Stud. Pure Math.},
    VOLUME = {19},
     PAGES = {459--566},
 PUBLISHER = {Academic Press, Boston, MA},
      YEAR = {1989},
      ISBN = {0-12-385342-7},
   MRCLASS = {81T40 (14H15 17B67 17B81 32G15)},
  MRNUMBER = {1048605},
MRREVIEWER = {Yukihiko\ Namikawa},
       DOI = {10.2969/aspm/01910459},
       URL = {https://doi.org/10.2969/aspm/01910459},
}

@article{V,
title = {Fusion rules and modular transformations in 2D conformal field theory},
journal = {Nuclear Physics B},
volume = {300},
pages = {360-376},
year = {1988},
issn = {0550-3213},
doi = {https://doi.org/10.1016/0550-3213(88)90603-7},
url = {https://www.sciencedirect.com/science/article/pii/0550321388906037},
author = {Erik Verlinde}
}

@article {Wang93,
    AUTHOR = {Wang, Weiqiang},
     TITLE = {Rationality of {V}irasoro vertex operator algebras},
   JOURNAL = {Internat. Math. Res. Notices},
  FJOURNAL = {International Mathematics Research Notices},
      YEAR = {1993},
    NUMBER = {7},
     PAGES = {197--211},
      ISSN = {1073-7928,1687-0247},
   MRCLASS = {17B68 (81R10)},
  MRNUMBER = {1230296},
MRREVIEWER = {Chong\ Ying\ Dong},
       DOI = {10.1155/S1073792893000212},
       URL = {https://doi.org/10.1155/S1073792893000212},
}

@book {U08,
    AUTHOR = {Ueno, Kenji},
     TITLE = {Conformal field theory with gauge symmetry},
    SERIES = {Fields Institute Monographs},
    VOLUME = {24},
 PUBLISHER = {American Mathematical Society, Providence, RI; Fields
              Institute for Research in Mathematical Sciences, Toronto, ON},
      YEAR = {2008},
     PAGES = {viii+168},
      ISBN = {978-0-8218-4088-7},
   MRCLASS = {81T40 (17B69 17B81 81R10)},
  MRNUMBER = {2433154},
MRREVIEWER = {Domenico\ Fiorenza},
       DOI = {10.1090/fim/024},
       URL = {https://doi.org/10.1090/fim/024},
}

\end{document}